\documentclass[11pt]{amsbook}
\usepackage{fourier}
\usepackage{mathrsfs}

\usepackage{scalerel}

\newcommand\reallywidehat[1]{\arraycolsep=0pt\relax%
\begin{array}{c}
\stretchto{
  \scaleto{
    \scalerel*[\widthof{\ensuremath{#1}}]{\kern-.5pt\bigwedge\kern-.5pt}
    {\rule[-\textheight/2]{1ex}{\textheight}} 
  }{\textheight} %
}{0.5ex}\\           
#1\\                 
\rule{-1ex}{0ex}
\end{array}}

\usepackage{amsthm,amsmath,amssymb,amscd, amsxtra}
\usepackage{latexsym}
\usepackage[all]{xy}
\usepackage{comment}
\usepackage{appendix}
\usepackage{pdfsync}

\usepackage{graphicx}
\usepackage{wrapfig}
\usepackage{caption}
\usepackage{mathtools}
\usepackage{epigraph}   
\usepackage{tikz-cd}
\usepackage{tikz}
\tikzset{math3d/.style=
    {x= {(-0.353cm,-0.353cm)}, z={(0cm,1cm)},y={(1cm,0cm)}}}

\usepackage{color}
\definecolor{Chocolat}{rgb}{0.36, 0.2, 0.09}
\definecolor{BleuTresFonce}{rgb}{0.215, 0.215, 0.36}

\usepackage[colorlinks,final,hyperindex, ]{hyperref}
\hypersetup{citecolor=BleuTresFonce, linkcolor=Chocolat, urlcolor  = BleuTresFonce}
\usepackage[noabbrev,capitalize]{cleveref}
 
\DeclareMathAlphabet{\mathbbold}{U}{bbold}{m}{n}

\newcommand{\BCH}{\mathrm{BCH}}
\def\k{\mathbbold{k}}
\def\dd{\mathrm{d}}
\def\Ai{\calA_\infty}
\def\Li{\calL_\infty}
\def\sAi{\calS\calA_\infty}
\def\sLi{\calS\calL_\infty}
\def\dsLi{\calS^{-1}\calL_\infty}
\def\Tw{\mathrm{Tw}}
\def\at{\alpha}
\def\F{\mathrm{F}}
\def\G{\mathrm{G}}
\def\NN{\mathbb{N}}
\def\RR{\mathbb{R}}
\def\wA{\widehat{A}}
\def\wo{\widehat{\otimes}}
\newcommand{\cc}{\circledcirc}

\DeclareMathAlphabet{\pazocal}{OMS}{zplm}{m}{n}

\def\a{\mathfrak{a}}
\newcommand{\Sy}{\mathbb{S}}

\def\calA{\pazocal{A}}

\def\calC{\pazocal{C}}

\def\calF{\pazocal{F}}
\def\calG{\pazocal{G}}

\def\calL{\pazocal{L}}

\def\calP{\pazocal{P}}
\def\calQ{\pazocal{Q}}

\def\calS{\pazocal{S}}
\def\S1{\pazocal{S}^{-1}}
\def\calT{\pazocal{T}}
\def\calU{\pazocal{U}}

\DeclareMathOperator{\Hom}{Hom}
\DeclareMathOperator{\eend}{end}
\DeclareMathOperator{\End}{End}
\DeclareMathOperator{\id}{id}

\def\colim{\mathop{\mathrm{colim}}}
\def\wcolim{\mathop{\widehat{\mathrm{colim}}}}

\DeclareMathOperator{\ad}{ad}
\DeclareMathOperator{\Def}{Def}
\DeclareMathOperator{\Der}{Der}

\DeclareMathOperator{\As}{As}
\DeclareMathOperator{\Ass}{Ass}
\DeclareMathOperator{\uAs}{uAs}

\DeclareMathOperator{\BV}{BV}
\DeclareMathOperator{\ncBV}{ncBV}

\DeclareMathOperator{\Gerst}{Gerst}
\DeclareMathOperator{\ncGerst}{ncGerst}

\DeclareMathOperator{\Gra}{Gra}

\DeclareMathOperator{\HyperCom}{HyperCom}
\DeclareMathOperator{\PreLie}{PreLie}
\DeclareMathOperator{\uPreLie}{uPreLie}
\DeclareMathOperator{\Perm}{Perm}
\DeclareMathOperator{\Dias}{Dias}
\DeclareMathOperator{\Dend}{Dend}

\DeclareMathOperator{\Lie}{Lie}
\DeclareMathOperator{\Com}{Com}

\DeclareMathOperator{\D}{D}

\def\Im{\mathop{\mathrm{Im}}}

\def\C{\pazocal{C}}
\def\I{\mathrm{I}}

\newcommand{\MC}{\mathrm{MC}}

\makeatletter

\theoremstyle{plain}

\newtheorem {theorem}{Theorem}[section]
\newtheorem{lemma}{Lemma}[section]
\newtheorem {corollary}{Corollary}[section]

\newtheorem {proposition}{Proposition}[section]

\theoremstyle{definition}

\newtheorem {definition}{Definition}[section]
\newtheorem {remark}{\sc Remark}[section]
\newtheorem {example}{\sc Example}[section]

\newcommand{\ac}{\scriptstyle \text{\rm !`}}



\thanks{S.S. was supported by the Netherlands Organisation for Scientific Research. B.V. was supported by the IUF and the grant ANR-14-CE25-0008-01 project SAT}

\begin{document}

\title{The twisting procedure}
\date{\today}

\author{Vladimir Dotsenko}
\address{School of Mathematics, Trinity College, Dublin 2, Ireland}
\email{vdots@maths.tcd.ie}

\author{Sergey Shadrin}
\address{Korteweg-de Vries Institute for Mathematics, University of Amsterdam, P. O. Box 94248, 1090 GE Amsterdam, The Netherlands}
\email{s.shadrin@uva.nl}

\author{Bruno Vallette}
\address{Laboratoire Analyse, G\'eom\'etrie et Applications, Universit\'e Paris Nord 13, Sorbonne Paris Cit\'e, CNRS, UMR 7539, 93430 Villetaneuse, France.}
\email{vallette@math.univ-paris13.fr}

\begin{abstract}
This paper provides a conceptual study of the twisting procedure, which amounts to create functorially new differential graded Lie algebras, associative algebras or operads (as well as their homotopy versions) from a  Mau\-rer--Cartan element. On the way, we settle the integration theory of complete pre-Lie algebras in order to describe this twisting procedure in terms of gauge group action. We give a criterion on quadratic operads for the existence of a meaningful twisting procedure of their associated categories of (homotopy) algebras. We also give a new presentation of the twisting procedure for operads \`a la Willwacher and we perform new homology computations of graph complexes. 
\end{abstract}

\maketitle

\newpage 

\epigraph{Il me semble qu'on pourrait tirer de ce travail une confirmation des points de vue suivants : d'abord l'int\'er\^et que pr\'esente l'\'etude de groupes d\'efinis \`a partir d'autres structures poss\'edant plusieurs op\'erations (par exemple des alg\`ebres de Lie). En effet, la simplicit\'e apparente des axiomes des groupes ne fait souvent que masquer une extr\^eme complexit\'e, et d'autres structures, plus riches par le nombre de leurs axiomes, se laissent plus facilement \'etudier. Il conviendrait donc de rechercher si d'autres structures alg\'ebriques pourraient permettre la construction de nouvelles cat\'egories de groupes.}{Michel Lazard, Ph.D. Thesis \cite{Lazard50}}

\setcounter{tocdepth}{1}

\tableofcontents

\chapter*{Introduction}
Seminal works of L. Maurer \cite{Maurer88} and E. Cartan \cite{Cartan04} investigating integrability of Lie algebras to (local) Lie groups effectively introduced what differential geometers now call the Maurer--Cartan $1$-form on principal $G$-bundles. In this language, the \emph{Maurer--Cartan equation} 
\[d\omega+{\textstyle \frac12}[\omega,\omega]=0\] becomes the flatness condition for the connection defined by that form. In general, a flat connection in a vector bundle $E\to M$ allows one to define a ``twisted de Rham differential'' on the sheaf of $E$-valued differential forms; this may be viewed as a de Rham cohomology version of the Riemann--Hilbert correspondence between vector bundles with flat connections and local systems. In the case of the Maurer--Cartan form, one deals with $\mathfrak{g}$-valued differential forms which form a differential graded Lie algebra; the \emph{twisted differential} of that algebra is of the form 
\[d+[\omega,\cdot]\ .\] It is well known that similar formulas are also meaningful for differential graded Lie algebras of more abstract nature, e.g. arising in deformation theory, rational homotopy theory, symplectic geometry, higher algebra, and quantum algebra: in each of those situations, appropriate Maurer--Cartan equations arise, and a solution to such equation, a \emph{Maurer--Cartan element}, can be used to produce a twisted version of the structure.\\

Let us offer a brief recollection of the many different ways Maurer--Cartan elements emerge in these fields.
The leading principle of deformation theory over a field of characteristic $0$ claims that any deformation problem can be encoded by a differential graded Lie algebra, see \cite{Lurie10, Pridham10, Toen17} for a precise statement and  proof. In this case, the considered structures correspond bijectively to the Maurer--Cartan elements. Then, Maurer--Cartan elements in the twisted differential graded Lie algebra correspond to deformations of the original structure. 
In rational homotopy theory, Maurer--Cartan elements of the Lie model \cite{Hinich97, Getzler09} of a space correspond to its points. In this case, the twisting procedure creates a Lie model of the same space but pointed at the given Maurer--Cartan element. 
The construction of the Floer cohomology of Lagrangian submanifolds  in symplectic geometry described in \cite{FOOO09I} is given by first  considering a curved homotopy associative algebra and then twisting it with a Maurer--Cartan element, when one exists, in order to produce a meaningful differential. 
In higher algebra, twisted homotopy Lie algebras together with some notion of $\infty$-morphisms, provide a suitable higher categorical enrichment for the categories of homotopy algebras over an operad \cite{DolgushevRogers17, DolgushevHoffnungRogers14}. 
Finally, in quantum algebra, the twisting procedure for operads themselves allows T. Willwacher to recover M. Kontsevich's graph complex and to prove that its degree $0$ cohomology group are given by the Grothendieck--Teichm\"uller Lie algebra \cite{Willwacher15}. Going even further with this interpretation, T. Willwacher was able to solve the following conjecture: the group of homotopy automorphisms of the little disks operad is isomorphic to the (pro-unipotent) Grothendieck--Teichm\"uller group, see also \cite{Fresse17II}. \\

Let us recall how the twisting procedure works precisely on the level of homotopy associative algebras; the case of homotopy Lie algebras is similarly. This kind of algebraic structure, also known as $\Ai$-algebra, is made up of a graded module $A$ equipped with multilinear operations $m_n : A^{\otimes n} \to A$, for $n\ge 1$, satisfying some relations. For instance, the first map $d\coloneqq m_1$ is a square-zero linear operator and thus a differential on $A$. Differential graded associative algebras corresponds to the case when the operations $m_n\equiv 0$ are trivial for $n\ge 3$. The \emph{Maurer--Cartan equation} is 
\begin{equation}\label{eq:MCeqIntro}
d(a)+\sum_{n \ge 2} m_n(a, \ldots, a)=0\ . 
\end{equation}
The twisted operations are obtained by plugging the element $a$ everywhere: 
\[m_n^a\coloneqq \sum_{r_0, \ldots, r_n\ge 0} \pm  m_{n+r_0+\cdots+r_n}\big(a^{r_0}, -, a^{r_1}, -, \ldots,  - , a^{r_{n-1}}, -,a^{r_n}  \big)\ , \]
for any $n\ge 1$, where the notation $a^r$ stands  for $a^{\otimes r}$. 
They form again an $\Ai$-algebra when the element $a$ is a Maurer-Cartan element.\\

Where does this result comes from conceptually? First, one notices that the Maurer--Cartan equation \eqref{eq:MCeqIntro} is made up of an infinite series and thus requires an underlying complete topology in order to be well defined. To do so, we consider filtered differential graded modules where the topology is defined by a basis of opens of the origin made up of decreasing sub-modules. 
This type of topology finds its source at least in the generalisation of the Lie theory to filtered Lie algebras and groups due to M. Lazard in his Ph.D. thesis \cite{Lazard50}. It became later omnipresent in commutative algebra, see N. Bourbaki \cite{Bourbaki61}, and  then in algebraic geometry, deformation theory, rational homotopy theory, and microlocal analysis. 
In the present paper, we want to treat all the known cases and new ones, of the twisting procedure for many categories of algebras over an operad and for various categories of operads themselves. In order to do so, we develop extensively the symmetric monoidal properties of filtered and complete differentials graded modules. \\

\smallskip

The deformation theory of algebras, including homotopy algebras, over an operad is faithfully encoded in a pre-Lie algebra of convolution type, see \cite[Chapter~10]{LodayVallette12}. The notion of a pre-Lie algebra sits between that of an associative algebra and of a Lie algebra: any associative algebra is a pre-Lie algebra and the skew-symmetrisation of the pre-Lie binary product renders a Lie bracket.
In \cite{DotsenkoShadrinVallette16}, we showed, under a strong weight graded assumption, that pre-Lie algebras can be integrated. In the present paper, we show that this integration theory extends to the complete setting. We actually have very few new arguments to give since the symmetric monoidal properties, mentioned in the above paragraph,  ensure that the operadic computations performed in loc. cit. hold in the complete setting as well. This part can also be seen as a generalisation of M. Lazard integration theory of complete Lie algebras \cite{Lazard50} to complete pre-Lie algebras. 
This treatment gives rise to the biggest \emph{deformation gauge group} of algebraic structures modelled by operads that we are aware of. \\

Already, the action of the simplest elements of this deformation gauge group is rich: it produces the twisting procedure as follows. 
The notion of a \emph{curved} $\Ai$-algebra is defined like the notion of $\Ai$-algebra but with one more ``operation'', the \emph{curvature} $m_0 : \k\cong A^{\otimes 0} \to A$, which amounts to the data of a particular element of $A$. The convolution pre-Lie algebra which encodes curved $\Ai$-algebras is an extension of the 
convolution pre-Lie algebra which encodes $\Ai$-algebras: the latter is equal to  the product of the former with $A$.
The elements of (the first layer of the filtration of) $A$ thus form a sub-group (with the addition) of the deformation gauge group of a curved $\Ai$-algebra, so they can act on this structure to induce functorially a new ones: the formulas are the ones of the twisting procedure. Any such twisted curved $\Ai$-algebra is actually an $\Ai$-algebra if and only if the twisted curvature $m_0^a$ vanishes, that is when $a$ satisfied the Maurer--Cartan equation. This gauge group interpretation of the twisting procedure allows us to reprove in straightfoward  and short way its various properties, notably the crucial ones related to complete curved $\Li$-algebras used in deformation theory , like in \cite{Getzler09, DolgushevRogers15}. \\

This result completes the programme undergone by the authors since \cite{DotsenkoShadrinVallette16}, where we try to describe ``all'' the functorial constructions of homotopy algebras by means of the deformation gauge group action. The Koszul hierarchy and the homotopy transfer theorem were treated in loc. cit., see also \cite{Markl15} for the former one. This programme agrees very much with the ideas of M. Lazard mentioned in the epigraph to this manuscript: one purpose of the present paper is to encompass some complicated and lengthy parts of the operadic calculus in a simple and compact way using group theory. \\

Why can one twist the differential graded Lie algebras and associative algebras and what about the other types of algebras? The above-mentioned conceptual understanding of the twisting procedure also allows us to give a criterion on a given quadratic operad for its categories of homotopy algebras to admit a meaningful twisting procedure. Roughly speaking, a category of homotopy algebras over a quadratic operad can be \emph{twisted} if and only if the Koszul dual category of algebras admits a coherent notion of \emph{unit}. Categories of homotopy Lie algebras and associative algebras are of course proved to be twistable, but the category of (homotopy) pre-Lie algebra is not. We present a new example: the category of homotopy permutative algebras is shown to be twistable, with detailed formulas.\\

The sum or the product of the components of an operad produces a pre-Lie algebra. However, even if pre-Lie algebras cannot be twisted in general, operads themselves can be twisted by their Maurer--Cartan elements. 
One can use this in order to describe the operads which encode homotopy Lie algebras or homotopy associative algebras \emph{together with} a Maurer--Cartan element. Pursuing in this direction, this shows how one can encode their twisting procedure in a precise operadic way. 
This theory was introduced by T. Willwacher in his seminal work on the Grothendieck--Teichm\"uller Lie algebra and Kontsevich's graph complexes \cite{Willwacher15}. It was later studied in details by V. Dolgushev and T. Willwacher in \cite{DolgushevWillwacher15}, see also the survey by V. Dolgushev and C. Rogers \cite{DolgushevRogers12}.
However, we believe that these last two references are great for the experts but they might not be that accessible to a wide audience since the core of the definition contains several technical points. 
Our purpose here is  to provide the literature with a gentle introduction, based on a new presentation, to this key notion in order to popularise it. \\

We conclude this study with new computations of the cohomology groups of twisted operads: we treat the case of the  classical operads encoding respectively Gerstenhaber and Batalin--Vilkovisky algebras and the case of their   nonsymmetric analogous operads introduced recently in \cite{DotsenkoShadrinVallette15}. This latter cases give rise to interesting  non-commutative graph complexes.

\subsection*{\sc Layout}
In \cref{sec:OptheoyFilMod}, we consider the categories of filtered and then complete differential graded modules; we establish their various symmetric mo\-no\-id\-al properties in order to develop their operadic theory. 
In \cref{sec:TopoDefTh}, we settle the integration theory of complete differential graded left-unital pre-Lie algebras; this gives rise to a gauge group which is shown to govern the deformation theory of homotopy algebras over a non-necessarily augmented operads. 
\cref{sec:GaugeTwist} contains the easiest application of the previous section: the action of the arity $0$ elements of the deformation gauge group is shown to give the twisting procedure for (shifted or not, curved or not, homotopy or not) differential graded Lie algebras or associative algebras. From this conceptual interpretation, we easily derive ``all'' the properties of the twisting procedure. In \cref{subsec:TwistableAlg}, we give a criterion on quadratic operads for their categories of homotopy algebras to admit a meaningful twisting procedure. 
In \cref{sec:TwNsOp}, we give a quick treatment of T. Willwacher's twisting procedure for operads using a new point of view. 
Finally, \cref{sec:Computations} provides full computations of the cohomology groups of twisted operads. 

\subsection*{\sc Conventions}
We basically work over a ring $\k$, but we add extra properties when needed: at some point in 
\cref{sec:TopoDefTh}, we require $\k$ to be a field of characteristic $0$ and in \cref{sec:GaugeTwist} and \cref{sec:TwNsOp} we need that the ring $\k$ contains $\mathbb{Q}$ and to work with flat $\k$-modules. 
The results of \cref{subsec:TwGerst} and \cref{subsec:TwBV}  hold over a field of characteristic $0$.  We use homological degree convention, for instance differentials have degree $-1$, except in \cref{sec:Computations} where we use cohomological degree convention. 
The various operadic conventions and notations come from \cite{LodayVallette12}. For instance, we use ``ns operads'' to mean ``nonsymmetric'' operads. 
In the present text,  the term  ``cooperad'' covers the algebraic structure defined by \emph{partial} or \emph{infinitesimal decomposition maps} $\Delta_{(1)} : \calC \to \calC\, {\circ}_{(1)}\, \calC$, see \cite[Section~6.1.4]{LodayVallette12}.  The upshot of such maps is denoted by  $\mu \circ_i \nu$, where this convention stands for the 2-vertices tree made up of the corolla $\nu$ grafted at the $i^{\textrm{th}}$-leaf of the corolla $\mu$. 

\subsection*{\sc Acknowledgements}
It is pleasure to thank Joan Bellier-Mill\`es, Ricardo Campos, Joana Cirici, Gabriel Drummond-Cole,  Daniel Robert-Nicoud, Jim Sta\-sh\-eff, and Tho\-mas Willwacher  for interesting discussions. 
The authors would like to thank 
the University of Amsterdam, 
Trinity College Dublin,
the University Nice Sophia Antipolis, the University Paris 13, and the 
 Max Planck Institute for Mathematics in Bonn for the excellent working conditions during the preparation of this paper. 
B.V. expresses his appreciation to the Laboratoire J.A. Dieu\-don\-n\'e of the University Nice Sophia Antipolis for the particularly generous hospitality over the last two years. 
\newpage
\chapter{Operad theory for filtered and complete modules}\label{sec:OptheoyFilMod}
In algebra, one  has to consider infinite series on many occasions. 
In order to make sense, these formulas require an extra topological assumption on the underlying  module. 
In this section, we first recall the notion of filtered and then complete modules which provides us with such a complete topology. 
This type of topology considered for instance in algebraic geometry, deformation theory, rational homotopy theory and microlocal analysis,  finds its source in the generalisation of the Lie theory to filtered Lie algebras and groups due to M. Lazard in his PhD thesis \cite{Lazard50}. It became later omnipresent in commutative algebra, see N. Bourbaki \cite{Bourbaki61}.

In this section, we establish the various properties for the monoidal structures on the categories of filtered modules and complete modules and for their associated monoidal functors. The main goal is to develop the  theory of  operads and operadic algebras in this context. 
First, this allows us to compare the various categories of filtered and complete algebras and recover conceptually the various  known definitions of filtered complete algebras that one can find in the literature. Then, as we will see in the next sections, this allows us to get for free operadic results on the complete setting since the previously performed operadic calculus still hold in this generalised context  by the above-mentioned monoidal properties. 

This present exposition shares common points with that of 
P. Deligne \cite[Section~1]{Deligne71}, and that of M. Markl \cite[Chapter~$1$]{Markl12}; it  is close to B. Fresse treatment \cite[Section~$7.3$]{Fresse17I}, though we did not find there all the results  that we need in the present paper and in further ones like \cite{Robert-NicoudVallette18}.

\section{Filtered algebras}

\begin{definition}[Filtered module]
A \emph{filtered  module} $(A,\F)$ is a $\k$-module $A$ equipped with a filtration 
$$A=\F_0 A \supset \F_1 A \supset \F_2 A \supset \cdots \supset \F_k A \supset F_{k+1}A \supset \cdots$$
made up of submodules.
\end{definition}

This condition implies that the subsets $\{x+\F_k A\, | \,  x\in A, k\in \NN\}$
form  a neighbourhood basis of a first-countable topology on $A$, which is thus a Fr\'echet--Urysohn space and so a sequential space. Since this topology is induced by submodules, any filtered module is trivially a topological module, that is  the scalar multiplication and the sum of elements are continuous maps, when one considers the discrete topology on the ground ring $\k$.

\begin{example}
Let $I$ be an ideal of the ring $\k$ and the $A$ be a $\k$-module. The submodules $\F_k A\coloneqq I^k A$, for $k\ge 0$, form a filtration of $A$ and the associated topology is called the \emph{$I$-adic topology} of $A$. 
\end{example}

\begin{lemma}\label{lem:Fkclosed}
The subsets $\F_k A$, for $k\in \NN$, are closed with respect to this topology. 
\end{lemma}

\begin{proof}
Let $\left\{x_n\in \F_k A\right\}_{n\in \NN}$ be a sequence converging to $x\in A$. There exists $N\in \NN$ such that, for all $n\ge N$, we have $x_n-x\in \F_k A$. 
Since the element $x_N$ lives in $\F_k A$, which is a submodule of $A$, this implies that $x$ lives in $\F_k A$ too. 
\end{proof}

Let $(A, \F)$ and $(B, \G)$ be two filtered modules. We consider the following induced filtration on the mapping space $\Hom(A,B)$: 
$$
 \calF_k \Hom(A, B)\coloneqq\left\{
f :A \to B\ | \ f(\F_n A)\subset \G_{n+k} B\ , \ \forall n\in \NN
\right\}\ .
$$
This filtration endows $\hom(A,B)\coloneqq\calF_0 \Hom(A, B)$ with a filtered module structure. 

\begin{remark}
Notice that any map in $\calF_k \Hom(A, B)$ is continuous with respect to the associated topologies, but it is not true that any continuous map is of this  form. 
\end{remark}

\begin{definition}[Filtered map]
A \emph{filtered map} $f : (A, \F) \to (B, \G)$ between two filtered modules is an element $f\in \hom(A,B)$, that is a linear  map  preserving the respective filtrations: $f(\F_n A)\subset \G_{n} B$, for all $n\in\NN$. 
\end{definition}

The induced filtration on  the tensor product of two filtered modules is given by 
\[\calF_k (A\otimes B)\coloneqq\sum_{n+m=k} \Im \big(\F_n A \otimes \G_m B\to A\otimes B\big)\ .\]

\begin{lemma}\label{lem:FilModSymCat}
The category of filtered modules, with filtered maps equipped with the aforementioned internal hom and filtration on  tensor products forms a bicomplete closed symmetric  monoidal  category, whose monoidal product preserves colimits. 
\end{lemma}

\begin{proof}
Given a collection $\left(A^i, \F^i\right)_{i\in \mathcal{I}}$ of filtered modules, their coproduct is given by $A\coloneqq\bigoplus_{i\in \mathcal{I}} A^i$  with filtration 
\[ \F_k A\coloneqq\left\{a_{i_1}+\cdots+a_{i_n}\ | \ a_{i_j}\in \F^{i_j}_k A^{i_j}, \ \forall j\in \{1, \ldots, n\}\right\} 
\]
and their product is given by $B\coloneqq\prod_{i\in \mathcal{I}} A^i$ with filtration 
\[
\G_k B\coloneqq\left\{(a_{i})_{i\in \mathcal{I}}\ | \ a_{i}\in \F^i_k A^{i}, \ \forall i\in \mathcal{I}\right\} \ .
\]
Cokernels of filtered maps $f : (A,\F) \to (B, \G)$ are given by $p : B \twoheadrightarrow B/\Im f$ equipped with the  filtration $p(\G_k)\cong\G_k B/(\Im f \cap \G_k B)$.
Since this category is (pre)additive, all coequalizers of pairs $(f,g)$ are given by cokernels of differences $f-g$. So this category admits all colimits. 
In the same way, kernels of filtered maps $f : (A,\F) \to (B, \G)$ are given by $f^{-1}(0)$ with filtration $f^{-1}(0) \cap \F_k B$.
Since this category is (pre)additive, all equalizers of pairs $(f,g)$ are given by kernels of differences $f-g$. So this category admits all limits. 
Notice that this category, though additive, fails to be Abelian: in general, maps do not have categorical images. 

Finally, it is straightforward to check that the various structure maps of the (strong) symmetric closed monoidal  category of modules are filtered. From the above description of coproducts and cokernels, it is easily seen that the monoidal product preserves them, so it preserves all colimits.
\end{proof}

The properties of Lemma~\ref{lem:FilModSymCat} ensures that one can develop  operad theory  in the symmetric mo\-no\-id\-al category of filtered modules, see \cite[Chapter~$5$]{LodayVallette12}. For instance, an operad in this context will be called a  \emph{filtered operad}. 

\begin{example}
For a filtered module $(A, \F)$, the associated \emph{filtered endormorphism operad} 
  is the part of the endomorphism operad of $A$ made up of filtered maps, that is 
$$\mathrm{end}_A\coloneqq\left\{\hom(A^{\otimes n}, A)\right\}_{n\in \NN}\ .$$
An element of the filtered endomorphism operad is thus a linear map $f : A^{\otimes n} \to A$ satisfying 
$$      
f\left(\F_{k_1} A \otimes \cdots \otimes \F_{k_n} A\right) \subset \F_{k_1+\cdots+k_n} A
\ . $$
The full induced filtration on its underlying collection is given by
$$f\in \calF_k \eend_A(n)\ \   \text{when} \ \      
f\left(\F_{k_1} A \otimes \cdots \otimes \F_{k_n} A\right) \subset \F_{k_1+\cdots+k_n+k} A
\ .$$
\end{example}

\begin{definition}[Filtered $\calP$-algebra]\label{def:FilAlgebras}
Let $\calP$ be a filtered operad and let $(A, \F)$ be a filtered module. A \emph{filtered $\calP$-algebra structure} on $(A,\F)$ amounts to the data of a filtered morphism of operads 
$$\calP \to \eend_A\ .$$
\end{definition}

We denote the category of modules by $\mathsf{Mod}$ and that of filtered modules by $\mathsf{FilMod}$. 

\begin{proposition}\label{prop:1Adj}
The  functor $\sqcup \, : \,  \mathsf{FilMod}\to \mathsf{Mod}$, which forgets the filtration, admits a left adjoint full and faithful functor 
$$\vcenter{\hbox{
\begin{tikzcd}[column sep=1.2cm]
\mathrm{Dis}  \ : \ 
\mathsf{Mod} 
\arrow[r, harpoon, shift left=1ex, "\perp"']
&
\arrow[l, harpoon,  shift left=1ex]
\mathsf{FilMod} 
\ : \ \sqcup 
\end{tikzcd}
}}$$
given by the trivial filtration: 
$$\mathrm{Dis}(A)\coloneqq(A, A=\F^{\mathrm{tr}}_0 A\supset 0=\F_1 ^{\mathrm{tr}}A=\F_2^{\mathrm{tr}} A= \cdots) \ ,$$
which induces the discrete topology.
These two functors are strictly symmetric monoidal. 
\end{proposition}

\begin{proof}
It is straightforward to check the various properties.
\end{proof}

The monoidal part of this proposition ensures that the underlying collection, without the filtration, of any filtered operad is an operad. For instance, the filtered endomorphism operad $\eend_A$ is a strict suboperad of the endomorphism operad $\End_A$. In the other way round, any  operad can be seen as a filtered operad equipped with the trivial filtration, that is with discrete topology. 

\begin{example}
Definition~\ref{def:FilAlgebras}, applied to the ns operad $\As$ of associative algebras, produces the classical notion of filtered associative algebra, see \cite[Section~I.3]{Lazard50}, \cite[Chapter~$3$]{Bourbaki61}, or \cite[Appendix~A.1]{Quillen69}  for instance. In the case of the operad $\Lie$, we recover the notion of filtered Lie algebras of Lazard \cite{Lazard50}. 
All the operadic constructions therefore hold in this setting. For instance, the morphism of operads $\Lie \to \As$, viewed as a morphism of filtered operads produces automatically the universal enveloping Lie algebra in the filtered world by \cite[Section~$5.2.12$]{LodayVallette12}.
\end{example}

\begin{definition}[Filtered differential graded module]
A \emph{filtered differential graded module} is a differential graded module in the category of filtered modules. Such a data amounts to a collection $\{A_n\}_{n\in \mathbb{Z}}$ of filtered modules equipped with a square-zero degree $-1$ filtered map $d$.
\end{definition}

All the aforementioned results hold \textit{mutatis mutandis} for filtered dg modules. For instance, this operadic definition allows us to recover naturally the following notions of filtered homotopy algebras structures present in the literature. 

\begin{example}
A \emph{filtered curved $\Ai$-algebra} structure on a filtered graded module $(A, \F)$ amounts to the data of curved $\Ai$-algebra structure $(m_0, m_1, m_2, \allowbreak \cdots)$ on $A$ according to Definition~\ref{def:CurvedAinfty} such that the various structure maps satisfy 
$$m_n(\F_{k_1} A, \ldots, \F_{k_n} A)\subset F_{k_1+\cdots+k_n} A\ . $$
This definition corresponds to  the one given by Fukaya--Oh--Ohta--Ono   in \cite{FOOO09I}, with the only difference that  these authors consider  modules  with the so called the energy filtration, indexed by non-negative real numbers $\RR^+$. See \cref{subsec:OpInt} for more operadic details.

The present operadic definition of a filtered (shifted) curved $\Li$-algebra given in \cref{def:csLiAlg} recovers the usual one, which is used for instance by Dolgushev--Rogers in \cite{DolgushevRogers15, DolgushevRogers17} (with the further constraint $\F_0 A=\F_1 A$).
\end{example}

\section{Complete algebras}

Any decreasing filtration $A=\F_0 A \supset \F_1 A \supset \cdots$ 
induces a sequence of surjective maps, 
$$ \xymatrix{
0=A/\F_0 A & \ar@{->>}[l]_(0.4){p_0}
A/\F_1 A& \ar@{->>}[l]_(0.45){p_1}
A/\F_2 A & \ar@{->>}[l]_(0.45){p_2} 
A/\F_3 A&  \ar@{->>}[l]\cdots} \ ,$$
where $p_k$ is the reduction modulo $\F_k A$. Its limit, denoted by 
$$\wA\coloneqq\lim_{k\in \NN} A/\F_k A \ ,$$
is made up of elements of the following form 
$$ \wA=\big\{(x_0,  x_1, x_2, \ldots)\ | \ x_k\in A/\F_k, \ p_k(x_{k+1})=x_k\big\}\ .$$
If we denote the structure maps by $q_k : \wA \twoheadrightarrow A/\F_k A$, $(x_0,  x_1, x_2, \ldots)\mapsto x_k$, then the limit module $\wA$ is endowed with the following canonical filtration 
 $\widehat{\F}_k \wA \coloneqq\ker q_k=\{
 (0, \ldots, 0, x_{k+1}, x_{k+2}, \ldots)\}$. With the associated topology, it forms a complete Hausdorff space. 
 
Let us denote by $\pi_k : A \twoheadrightarrow A/\F_k A$ the canonical projections. The canonical map $\pi : A \to \wA$, $x\mapsto (\pi_0(x), \pi_1(x), \pi_2(x),\ldots)$, associated to them, is filtered and thus continuous. 

\begin{definition}[Complete module]
A \emph{complete module} is a filtered module $(A, \F)$
 such that the canonial morphism
$$\pi\ : \ A \stackrel{\cong}{\longrightarrow} \wA=\lim_{k\in \NN} A/\F_k A $$
is an isomorphism. 
\end{definition}

The kernel of the canonical map $\pi : A \to \wA$ is equal to the intersection of all the sub-modules $\F_k A$. Therefore, it is a monomorphism if and only if 
$$\cap_{k\in\NN} \F_k A=\{0\}\ ;$$ 
 this condition is equivalent  for the associated topology on $A$ to be Hausdorff. 
The canonical map $\pi$ is an epimorphism if and only if the  associated topological is complete, which explains the terminology chosen here.  When the canonical map is an isomorphism, it is an homeomorphism since $\pi^{-1}$ is filtered and thus continuous. 

\begin{remark}
One defines the \emph{valuation}  of an element $x\in A$ by 
$$\nu(x)\coloneqq k \quad \text{when} \quad x\in  \F_k A\backslash \F_{k+1} A\ .$$
When the associated topology is Hausdorff, we set, by convention, that $\nu(0)\coloneqq+\infty$, and the valuation is well defined. 
In this case, the topology is metric, with the distance given by 
$$d(x,y)=\frac{1}{\nu(y-x)+1}\ .$$
As a consequence filtered maps are uniformly continuous and the canonical map $\pi : A \hookrightarrow \wA$ makes $\wA$ into  the completion of $A$: $\wA$ is complete, contains $A$ as a dense subset, and is  unique up isometry for such a property.
\end{remark}

\begin{example}
The toy model here is the ring of polynomials $\k[X]$ with its $X$-adic filtration $\F_k\, \k[X]\coloneqq X^k \, \k[X]$. Its topology is Hausdorff but not complete. Its completion is the ring of formal power series $\widehat{\k[X]}\cong\k[[X]]$.
\end{example}

In any complete module, convergent series have the following simple form. 

\begin{lemma}\label{lem:Conv}
Let $(A, \F)$ be a complete module. The series $\sum_{n\in \NN} x_n$ associated to a sequence of elements $\{x_n\}_{n\in \NN}$ is convergent if and only if the sequence $x_n$ converges to $0$. 
\end{lemma}

\begin{proof}
We classically consider the sequence $X_n\coloneqq\sum_{k=0}^n x_k$, for $n\in \NN$. If the sequence $\{X_n\}$ converges, then it is a Cauchy sequence and so $x_n=X_n-X_{n-1}$ tends to $0$. In the other way round, if the sequence $\{x_n\}_{n\in \NN}$ tends to $0$, this means 
$$\forall k\in \NN,\,  \exists N\in \NN, \, \forall n\ge N, \ x_n\in \F_k A\ . $$
Since $\F_k A$ is a submodule of $A$, we have 
$$\forall k\in \NN,\, \exists N\in \NN,\, \forall m\ge n\ge N, \ X_m-X_n=x_m+ \cdots+x_{n+1}\in \F_k A\ , $$
that is the sequence $\{X_n\}$ is Cauchy is thus convergent. 
\end{proof}

We can consider the following first definition of a complete algebra over an operad. We will see later on at Theorem~\ref{thm:CompleteAlg} that it actually coincides with the conceptual one. 

\begin{definition}[Complete algebra]\label{Def:CompleteAlg}
A complete algebra over a filtered operad $\calP$ is a complete module endowed with a filtered $\calP$-algebra structure. 
\end{definition}

\begin{example}
The aforementioned example of formal power series $\k[[X]]$ is a complete algebra over the operad $\Ass$ (respectively, $\Com$), that is a complete associative (respectively, commutative associative) algebra.
\end{example}

We consider the full subcategory of filtered modules made up of complete modules and we denote it by $\mathsf{CompMod}$.

\begin{proposition}\label{prop:2Adj}
The completion of a filtered module defines a functor which is left adjoint to the forgetful functor: 
$$\vcenter{\hbox{
\begin{tikzcd}[column sep=1.2cm]
\widehat{} \ \ : \ 
\mathsf{FilMod} 
\arrow[r, harpoon, shift left=1ex, "\perp"']
&
\arrow[l, harpoon,  shift left=1ex]
\mathsf{CompMod} 
\ : \ \sqcup \ .
\end{tikzcd}
}}$$
\end{proposition}

\begin{proof}
This statement amounts basically to the following property: 
any filtered map $f : A \to B$, 
with target a complete module, 
factors uniquely through the canonical map 
$$\vcenter{\hbox{
\begin{tikzcd}[column sep=1cm, row sep=1cm]
A 
\arrow[r,"f"]
\arrow[d,"\pi"']
&
B \\ 
\wA 
\arrow[ur,"\exists ! \bar{f}"']
& ,
\end{tikzcd}
}}$$
which  is nothing but the universal property of the limit $\widehat{A}$. 
\end{proof}

In order to endow the category $\mathsf{CompMod}$ of complete modules with a mo\-no\-idal structure, one could first think at the underlying tensor product of filtered modules. But this one fails to preserve complete modules, as the following example show 
$$\xymatrix{
\k[[X]]\otimes \k[[Y]] \ \ \ar@{^{(}->}[r]^(0.41){\neq} &\ \  \widehat{\k[X]\otimes \k[Y]}\cong \k[[X,Y]]} \ .$$

\begin{definition}[Complete tensor product]
The \emph{complete tensor product} of two complete modules $(A, \F)$ and $(B, \G)$ is defined by the completion of their filtered tensor product:
$$ A\wo B\coloneqq\widehat{A\otimes B}\ . $$
\end{definition}

\begin{remark}
Notice that when the two filtered modules are not necessarily complete, the completion of their tensor product is equal to 
$\widehat{A\otimes B}\, \cong\, \wA \; \wo\;  \widehat{B}$ \ .
\end{remark}

\begin{lemma}
The category $\left(\mathsf{CompMod}, \wo\right)$ of complete modules equipped the complete tensor product is a bicomplete closed symmetric  monoidal  category wh\-ose monoidal product preserves  colimits. 
\end{lemma}

\begin{proof}
Let  $\left(A^i, \F^i\right)_{i\in \mathcal{D}}$ be a diagram of complete modules. 
Since the completion functor is a left adjoint functor, it  would preserve colimits if these latter ones exist. Therefore we define colimits in the category of complete modules by the formula 
\[
\displaystyle \wcolim_{i\in\mathcal{D}} A^i \coloneqq \reallywidehat{\displaystyle \colim_{i\in\mathcal{D}} A^i} \ .
\]
It is straightforward to check that they satisfy the universal property of colimits from the property of the completion functor. 
For instance, coproducts of complete modules are given by 
\[\widehat{\bigoplus_{i\in \mathcal{I}}} A^i\cong\widehat{\bigoplus_{i\in \mathcal{I}} A^i}\ , \]
and finite coproducts of complete modules are simply given by the finite direct sums of their underlying filtered module structure, since this latter one is already complete. 
In the other way round, since the forgetful functor from complete modules to filtered modules  is a right adjoint functor, it  would preserve limits if these latter ones exist. One can actually see that the formulas in the category of filtered modules (Lemma~\ref{lem:FilModSymCat}) for products and kernels once applied to complete modules render complete modules. Therefore the category of complete modules admits limits since it is a (pre)additive category. 
Like the category of filtered modules and for the exact same reasons, the category of complete modules is additive but fails to be abelian. 

The various axioms of a strong monoidal category are straightforward to check. 
The preservation of the colimits by the complete tensor product is automatic from its definition, the above characterisation of colimits and Lemma~\ref{lem:FilModSymCat}: 
\[
\left(\wcolim_{i\in\mathcal{D}} A^i\right) \wo B \cong 
\left(\reallywidehat{\displaystyle \colim_{i\in\mathcal{D}} A^i}\right)\wo B \cong 
\reallywidehat{\left(\displaystyle \colim_{i\in\mathcal{D}} A^i\right)\otimes B} \cong 
\reallywidehat{{\displaystyle \colim_{i\in\mathcal{D}}} \left(A^i\otimes B\right)} \cong 
\wcolim_{i\in\mathcal{D}} \left(A^i\wo B\right)\ .
\]
It remains to prove that this symmetric monoidal category is closed. To this end, it is enough to prove that the internal filtered hom of complete modules $A,B$ is complete. One first notices that 
$$\bigcap_{k\in\NN} \calF_k \hom(A,B)=  \hom(A, \cap_{k\in\NN} \G_k B)=\{0\}\ .$$
Now let $\{f^n : A \to B\}_{n\in \NN}$ be a Cauchy sequence of filtered maps. This means that 
$$\forall k\in \NN, \exists N\in \NN, \forall m,n\ge N, \ f^m-f^n\in  \calF_k \hom(A,B)\ . $$
Therefore, the sequence $\left\{f^n(a)\right\}_{n\in \NN}$ in $B$  is Cauchy for any $a\in A$ and thus converges since $B$ is complete. We denote by $f(a)$ its limit. Considering the discrete topology on the ground ring $\k$, the scalar multiplication and the sum of elements are continuous, so this assignment defines a linear map $f : A \to B$. When $a\in \F_l A$, the Cauchy sequence $\left\{f^n(a)\right\}_{n\in \NN}$ lives in $\G_l B$, which is closed by Lemma~\ref{lem:Fkclosed}. Hence, we have $f(a)\in G_l B$ and the whole map $f$ is filtered, that is $f\in \hom(A,B)$. Using again the argument that the $\G_{l+k} B$ are closed, one can see, after a passage to the limit, that 
$$\forall k\in \NN, \exists N\in \NN, \forall n\ge N, \ f^n-f\in  \calF_k \hom(A,B)\ , $$
since this means that $\left(f^n-f\right)(\F_l A)\subset \G_{l+k} B$. 
\end{proof}

\begin{remark}
The same result holds true \textit{mutatis mutandis} for  complete differential graded modules. 
\end{remark}

So one can develop operad theory in this setting. This produces automatically a notion of  \emph{complete dg  operads} together with their categories of complete dg algebras. The next proposition shows that there is nothing to change from the filtered case for the endomorphism operad. 

\begin{proposition}\label{prop:EndComplete}
The complete endomorphism operad of a complete dg module $A$ is canonically isomorphic to the filtered endormorphism operad:
\[\left\{\hom\big(A^{\widehat{\otimes} n}, A\big)\right\}_{n\in \NN}\cong \left\{\hom\big(A^{{\otimes} n}, A\big)\right\}_{n\in \NN}
=\eend_A \ .\]
\end{proposition}

\begin{proof}
The proof relies entirely on the universal property of the completion functor as described in the proof of Proposition~\ref{prop:2Adj}. Under the same notations, one can check that the bijection 
\[
\begin{array}{clc}
\hom\big(\wA, B\big) &\to&  \hom(A, B) \\
\bar{f} & \mapsto & \bar{f} \circ \pi
\end{array}
\]
is  a bijection of filtered modules, whenever the filtered module $B$ is complete. In the present case, this induces 
\[ \hom\big(A^{\widehat{\otimes} n}, A\big)\cong \hom\big(\widehat{A^{{\otimes} n}}, A\big)\cong \hom\big(A^{{\otimes} n}, A\big)\ ,\]
for any $n\in \NN$. 
\end{proof}

\begin{proposition}\label{prop:strict-laxmonoidal}
The completion functor 
\[\ {\widehat{ }} \ \, : \, (\mathsf{FilMod}, \otimes) \to  (\mathsf{CompMod}, \wo)\] 
is strong symmetric mo\-no\-idal and the forgetful functor 
\[\sqcup \, : \,  (\mathsf{CompMod}, \wo) \allowbreak \to \allowbreak  (\mathsf{FilMod}, \otimes)\] 
is lax symmetric monoidal.
\end{proposition}

\begin{proof}
The structure map for the monoidal structure of the completion functor is the isomorphism 
$\wA \; \wo\;  \widehat{B} \xrightarrow{\cong} \widehat{A\otimes B}$. 
The structure map for the monoidal structure of the forgetful  functor is the canonical map 
$A\otimes B \to A \wo B$. 
\end{proof}

One can iterate the above functors $\widehat{\mathrm{Dis}}  : \mathsf{Mod} \to \mathsf{FilMod} \to \mathsf{CompMod}$. Since the discrete topology is already complete, this composition of functors does not change the underlying module; it just  provides it with the trivial filtration.   

\begin{corollary}\label{coro:12Adj}
The following pair of functors are adjoint 
$$\vcenter{\hbox{
\begin{tikzcd}[column sep=1.2cm]
\widehat{\mathrm{Dis}} \ \ : \ 
\mathsf{Mod} 
\arrow[r, harpoon, shift left=1ex, "\perp"']
&
\arrow[l, harpoon,  shift left=1ex]
\mathsf{CompMod} 
\ : \ \sqcup \ ,
\end{tikzcd}
}}$$
where $\widehat{\mathrm{Dis}}(A)\coloneqq(A, \F^{\mathrm{tr}})$.
The functor $\widehat{\mathrm{Dis}}$ is a strict monoidal functor and the functor $\sqcup$ is a lax monoidal functor. 
\end{corollary}

\begin{proof}
This adjunction is actually obtained as the composite of the two adjunctions of Proposition~\ref{prop:1Adj} and 
Proposition~\ref{prop:2Adj}. The monoidal structures on the two functors are obtained as composite of two monoidal structures from Proposition~\ref{prop:1Adj} and Proposition~\ref{prop:strict-laxmonoidal}.
\end{proof}

The main point for  these six functors to be monoidal is that each of them sends an operad in the source category to an operad in the target category and similarly for their associated notion of algebras. 

\begin{proposition}
Any dg  operad $\calP$ is a filtered (respectively complete) dg operad $\calP$ when equipped with the trivial filtration. Any $\calP$-algebra $A$ is a filtered (respectively complete) $\calP$-algebra when equipped with the trivial filtration. The category of discrete (respectively filtered) $\calP$-algebra is a full subcategory of the category of complete $\calP$-algebras. 
\end{proposition}

\begin{proof}
This is a direct consequence of the monoidal structure of the functor 
${\mathrm{Dis}}$ from Proposition~\ref{prop:1Adj} (respectively 
$\widehat{\mathrm{Dis}}$ from Corollary~\ref{coro:12Adj}),  since this latter one  does not modify the underlying module.
\end{proof}

As a consequence, we will now work in the larger category of complete $\calP$-algebras and extend the various operadic results to that level. The key property that the completion functor is left adjoint to the forgetful functor allows us to get the following  simple descriptions for the notions of complete operads and complete $\calP$-algebras. 

\begin{theorem}\label{thm:CompleteAlg}\leavevmode
\begin{enumerate}
\item The structure of a complete dg operad $\calP$ is equivalent to the structure of a filtered dg operad on a complete dg $\Sy$-module $\calP$. 

\item
Let $\calP$ be  a complete dg operad. The data of a dg $\calP$-algebra structure in the monoidal category of complete dg modules is equivalent to a complete dg $\calP$-algebra structure, as defined above,  that is a filtered dg $\calP$-algebra structure on an underlying  complete dg module.

\item Let $\calP$ be a filtered dg operad. The data of a complete dg $\calP$-algebra structure, as defined above, is equivalent to the data of a complete dg $\widehat{\calP}$-algebra structure.

\end{enumerate}
\end{theorem}

\begin{proof}\leavevmode
\begin{enumerate}
\item 
Since the forgetful functor $\sqcup : \mathsf{CompMod} \to \mathsf{FilMod}$ is lax symmetric monoidal according to Proposition~\ref{prop:strict-laxmonoidal}, it sends any complete operad structure on $\calP$ to a filtered operad structure on the underlying complete dg $\Sy$-module $\calP$.  In details, recall that a filtered operad structure amounts to a collection of filtered maps $\gamma_n$:
\[\calP\circ \calP(n)\coloneqq\bigoplus_{k\in \NN}
\calP(k)\otimes_{\Sy_k}
\left(
\bigoplus_{i_1+\cdots+i_k=n}
\mathrm{Ind}_{\Sy_{i_1}\times \cdots \times \Sy_{i_k}}^{\Sy_n}\big(
\calP(i_1)\otimes \cdots \otimes \calP(i_k)
\big)
\right)
\to \calP(n) \ ,\]
satisfying some relations, see \cite[Section~$5.2.1$]{LodayVallette12}.
Similarly, a complete dg operad structure amounts to a collection of filtered maps $\widehat{\gamma}_n$:
\[\calP\, \widehat{\circ}\, \calP(n)\coloneqq\widehat{\bigoplus}_{k\in \NN}
\calP(k)\wo_{\Sy_k}
\left(
\widehat{\bigoplus_{i_1+\cdots+i_k=n}}
\mathrm{Ind}_{\Sy_{i_1}\times \cdots \times \Sy_{i_k}}^{\Sy_n}\big(
\calP(i_1)\wo \cdots \wo \calP(i_k)
\big)
\right)
\to \calP(n) \ ,\]
satisfying the same type of relations. 
Since  the completion functor is left adjoint,  it preserves colimits and thus coproducts, which implies that 
$\widehat{\calP\circ \calP}\cong \calP \, \widehat{\circ}\,  \calP$. By pulling back along the canonical completion map 
$\pi : \calP\circ \calP \to \calP\, \widehat{\circ}\,  \calP$,  any 
dg operad structure $\widehat{\gamma}$
 in complete dg modules induces a filtered dg operad structure  ${\gamma}=\pi\, \widehat{\gamma}$.
 In the other way round, any filtered dg operad structure  $\gamma$ factors through $\widehat{\calP\circ \calP}\cong \calP \, \widehat{\circ}\,  \calP$, that is through an operad structure  $\widehat{\gamma}$ in complete dg modules. 

\item For the second point, the arguments are similar. The lax symmetric monoidal  functor $\sqcup : \mathsf{CompMod} \to \mathsf{FilMod}$ sends  
 any  $\calP$-algebra structure in the monoidal category of complete dg modules to a 
$\calP$-algebra structure in the monoidal category of filtered dg modules. This latter structure amounts to a morphism of filtered dg operads 
$\rho\ : \ \calP
\to \eend_A$,
under point $(1)$. 
Since the filtered dg operad $\calP$ is complete and since the endomorphism operad associated to a complete dg module $A$ is the same in the filtered and the complete case, after Proposition~\ref{prop:EndComplete}, a complete dg $\calP$-algebra structure on $A$ amounts to a morphism of complete dg operads 
$\widehat{\rho}\ : \ \calP
\to \eend_A$, 
which thus coincides with the above type of maps since the morphisms in the category of complete modules are that of the category of filtered modules. 

\item The arguments are again the same: by the universal property of the completion, any morphism of filtered dg operads $\rho \ :\ \calP\to \eend_A$ is equivalent to a morphism of complete dg operads $\widehat{\rho}\ : \  \widehat{\calP}\to \eend_A$, when $A$ is complete.
 \end{enumerate}
 \end{proof}

This result shows that the terminology ``complete algebra'' chosen in Definition~\ref{Def:CompleteAlg} does not bring any ambiguity since the two possible notions are actually equivalent. Notice that this theorem applies to discrete dg operads and discrete dg $\calP$-algebras, when equipped with the trivial filtration. 

\begin{example}
The free complete $\calP$-algebra of a complete module $V$ over a complete operad $\calP$ is given by 
$$ \calP\, \widehat{\circ}\,  V=\widehat{\bigoplus_{n\in \NN}}
\calP(n)\wo_{\Sy_n} V ^{\wo n}\ .$$
When the operad $\calP$ and the  module $V$ are discrete but endowed with the trivial filtration, we recover  the free $\calP$-algebra $\calP\circ V$. 
But, when the filtration arises from the weight grading for which $V$ is concentrated in weight $1$, so that $\F_0 V= \F_1 V=V$ and $\F_n V=0$, for $n\ge 2$, 
the free complete $\calP$-algebra on $V$ is equal to 
$$ \calP\, \widehat{\circ}\,  V\cong \prod_{n\in \NN}
\calP(n)\otimes_{\Sy_n} V ^{\otimes n} \ ,$$
since its  underlying filtration is given by 
\[
\calF_k\left(\bigoplus_{n\in \NN}
\calP(n)\otimes_{\Sy_n} V^{\otimes n}\right)=\bigoplus_{n\ge k}
\calP(n)\otimes_{\Sy_n} V^{\otimes n}\ .
\]
In this way, we recover the notions of free complete associative algebra present in 
\cite[Section~I.4]{Lazard50} or free complete Lie algebras present in  \cite[Section~II.1]{Lazard50} and \cite{LawrenceSullivan14}, for instance. This allows us to get automatically, that is operadically, the universal enveloping algebra in the complete case. 
\end{example}

\begin{proposition}
Let $\calP$ be a dg operad. The forgetful functor embeds the category of complete dg $\calP$-algebras as a full subcategory of filtered dg $\calP$-algebras. The completion functors sends a filtered dg $\calP$-algebra to a complete dg $\calP$-algebra. 
These two functors again form a pair of adjoint functors, where the completion functor is left adjoint.  
\end{proposition}

\begin{proof}
This is a direct corollary of Proposition~\ref{prop:strict-laxmonoidal} and Theorem~\ref{thm:CompleteAlg}.
\end{proof}

\newpage
\chapter{Pre-Lie deformation theory in the complete setting}\label{sec:TopoDefTh}

In this section, we generalise Lazard's treatment of Lie theory to develop the integration theory of complete pre-Lie algebras. The main goal is to  apply it  to the operadic convolution algebra, in order to give rise in this setting to the main tool of algebraic deformation theory: \emph{the deformation gauge group}. The present development relies on the  previous operadic computations performed in  \cite{DotsenkoShadrinVallette16}; they hold in the present context \emph{without any changes} since the various formulas make sense thanks to the complete  topology considered from the previous section. This degree of generalisation allows us to define a suitable notion of $\infty$-morphism of homotopy algebras encoded by non-necessarily coaugmented cooperads, like curved $\Ai$-algebras or curved $\Li$-algebras. 

\section{Complete convolution algebra}\label{subsec:CompConvAlg}

\begin{definition}[Complete left-unital differential graded pre-Lie algebra]
A \emph{complete left-unital dg pre-Lie algebra} amounts to the data 
$\a=\left(
A,  \F, d, \star, 1 
\right)$ of a dg complete module equipped with a 
filtration preserving the bilinear product whose associator is right symmetric 
$$(a\star b)\star c- a\star(b\star c)  = (-1)^{|b||c|}\big( (a\star c)\star b-a\star(c\star b)\big)\ ,$$
such that the differential $d$ is a derivation 
$$d(a\star b)=d(a) \star b + (-1)^{|a|} a \star d(b) $$
and such that the  element $1\in\F_0 A_0$ is a closed element $d(1)=0$ and a left unit 
$$ 1 \star a = a\ .$$ 
\end{definition}

The functor from dg operads to dg left-unital pre-Lie algebras \cite[Section~$5$]{DotsenkoShadrinVallette16} extends to a functor 
from complete dg operads to complete left-unital dg pre-Lie algebras: 
$$\xymatrix@R=0pt{
\mathsf{complete\ dg\ operads}\ar[r] & \mathsf{complete\ left}\textsf{-}\mathsf{unital \ dg \ pre}\textsf{-}\mathsf{Lie\ algebras}\\
\big(\calP,  \{\circ_i\}, \I \big)\ar@{|->}[r]& \left({\prod}_{n\in \NN} \calP(n), \star, 1\right)
\ ,}
$$
where the pre-Lie product $\star$ is given by the sum of the partial composition maps $\circ_i$ and 
and where the left-unit is given by $1\coloneqq(0, \I, 0, 0, \ldots)$. 
Recall that in any closed symmetric monoidal category the mapping space $\hom\left(\calC, \calP\right)$ from a cooperad $\calC$ to an operad $\calP$ forms a \emph{convolution operad}, see \cite[Section~$6.4.1$]{LodayVallette12}. 
When $\calC$ is a filtered dg cooperad and when $\calP$ is a complete dg operad, the aforementioned construction associates a complete dg pre-Lie algebra structure to the complete convolution dg operad $\hom\left(\calC, \calP\right)$, with  pre-Lie product  equal to 
\[\xymatrix@C=30pt{f\star g =\calC \ar[r]^(0.55){\Delta_{(1)}} & \calC\; \widehat{\circ}_{(1)}\, \calC 
\ar[r]^(0.46){f\; \widehat{\circ}_{(1)}\; g} & 
\calP\; \widehat{\circ}_{(1)}\, \calP  \ar[r]^(0,6){\gamma_{(1)}} & \calP \ ,}\]
where the various infinitesimal notions \cite[Section~$6.1$]{LodayVallette12} are considered in complete setting. 

\begin{proposition}
The space of equivariant maps from a filtered dg cooperad $\calC$ to a complete dg operad~$\calP$
$$\hom_\Sy\big(\calC, \calP\big)\coloneqq{\prod}_{n\in \NN} \hom_{\Sy_n}\big(\calC(n), \calP(n)\big)$$
forms a complete left-unital dg pre-Lie algebra 
$$\left(\hom_\Sy\big(\calC, \calP\big), \partial, \star, 1 \right)$$
and thus a complete dg Lie algebra 
$$\left(\hom_\Sy\big(\calC, \calP\big), \partial, [\;,\,]\right)\ ,$$
by skew-symmetrization of the pre-Lie product: $[a,b]\coloneqq a\star b - (-1)^{|a||b|} b\star a$\ .
\end{proposition}

\begin{proof}
As in the classical case, one can see that the space of equivariant maps is stable under the pre-Lie product and that the functor from dg pre-Lie algebras to dg Lie algebras extends to the complete level. 
\end{proof}

\begin{definition}[Complete convolution algebra]
Such a complete dg (pre)-Lie algebra will be called a \emph{complete convolution (pre)-Lie algebra}. 
\end{definition}

The crucial object in the study of the deformation theory of algebraic structures \cite{LodayVallette12, DotsenkoShadrinVallette16} is the complete convolution algebra associated to a filtered  cooperad $\calC$ and the complete endomorphism operad of a complete module $A$: 
\[
\a_{\, \calC, A}\coloneqq
\big(
\hom_\Sy\big(\calC, {\eend}_A\big), \partial, \star, 1 
\big)\ .
\]
The present definition, in the complete setting, is the most general we are aware of. 
It will be mandatory in \cref{sec:GaugeTwist} in order to give a (gauge) group interpretation to the twisting procedure of algebraic structures 

\section{Integration theory for complete pre-Lie algebras}\label{subsec:CompleteLieInt}
In \cite{DotsenkoShadrinVallette16}, we developed the integration theory of pre-Lie algebras under a strong weight grading assumption. Unfortunately, the deformation theory of many algebraic structures, like the ones that we will study in \cref{sec:GaugeTwist}, require to use curved Koszul dual cooperads which are not weight graded. In this case, the relevant object is the aforementioned complete convolution algebra $\a_{\, \calC, A}$. In this section, we extend the integration theory of weight graded left-unital dg pre-Lie algebras to the complete setting. This section can also be seen as the generalisation of the integration theory of complete Lie algebras of Lazard's thesis \cite{Lazard50} to left-unital complete pre-Lie algebras. In this section,  the ground ring $\k$ is assumed to contain the field $\mathbb{Q}$.  \\

Let us first recall the notion of a filtered and thus topological group \cite[Section~I.2]{Lazard50} defined in a  way similar way than  that of filtered and thus topological algebra. 

\begin{definition}[Filtered group and complete group]
A \emph{filtered group} am\-oun\-ts to the data $(G, \F, \cdot, e)$ of a group $(G,  \cdot, e)$ endowed with a filtration $\F$ made up of sub-groups
$$G=\F_1 G \supset \F_2 G \supset  \cdots \supset \F_k G  \supset \cdots \ ,$$
such that the commutator satisfies $xyx^{-1}y^{-1}\in \F_{i+j} G$ for $x\in \F_i G$ and $y\in \F_j G$. 
A filtered group is called \emph{complete} when the underlying topology is Hausdorff and complete. 
\end{definition}

Let $\a=\left(A,  \F, d, \star, 1 \right)$ be a complete left-unital dg pre-Lie algebra. By \cite[Chapitre~II]{Lazard50}, we get the following definition of the gauge group, via the  associated  complete dg Lie algebra.

\begin{definition}[Gauge group]
The \emph{gauge group} associated to a complete left-unital dg pre-Lie algebra $\a$ 
 is defined by 
$$\Gamma\coloneqq\big(
\F_1 A_0,  \BCH(\; ,\,), 0
\big)\ ,$$
where $\BCH$ refers to the Baker--Campbell--Hausdorff formula.
\end{definition}

 This latter one is convergent for elements in $\F_1 A_0$: 
let $x, y\in \F_1 A_0$ and recall that their BCH product is given by a series of the form 
$$\BCH(x,y)=\underbrace{x+y}_{\in \F_1 A_0}+\underbrace{\frac{1}{2}[x,y]}_{\in \F_2 A_0} + \underbrace{\frac {1}{12}\big([[x,y], y]+[[y,x], x]\big)}_{\in \F_3 A_0}+\cdots \ ,$$
which is thus  convergent by Lemma~\ref{lem:Conv}.

\begin{proposition}
The gauge group is a complete group, with respect to the following filtration:  
$$\F_1 \Gamma\coloneqq\F_1 A_0 \supset \F_2\Gamma\coloneqq \F_2 A_0 \supset  \cdots \supset \F_k\Gamma\coloneqq\F_k A_0  \supset \cdots \ .$$
\end{proposition}

\begin{proof}
The aforementioned form of the BCH product shows that each $\F_k\Gamma$ is a subgroup of the gauge group. It is known that 
\[
\BCH(\BCH(\BCH(x,y),-x), -y)=[x,y]+\cdots\ , 
\]
where the higher terms are iteration of brackets of at least one $x$ and one $y$ each time. Therefore, the commutators satisfy $\BCH(\BCH(\BCH(x,y),-x), -y)\in \F_{i+j} \Gamma$ for $x\in \F_i \Gamma$ and $y\in \F_j \Gamma$.
\end{proof}

For any element $\lambda\in \F_1 A$, we consider the following right iteration of the pre-Lie product 
$$\lambda^{\star n}\coloneqq\underbrace{(\cdots((\lambda \star \lambda) \star \lambda)\cdots )\star \lambda}_{n\  \text{times}}\in \F_n A\ . $$

\begin{definition}[Pre-Lie exponential]
The \emph{pre-Lie exponential} of an element $\lambda \in \F_1 A$ is 
defined  by the following convergent  series 
$$e^\lambda\coloneqq1 +\lambda + \frac{\lambda^{\star 2}}{2!} + \frac{\lambda^{\star 3}}{3!} +\cdots \ . $$
\end{definition}

We consider the set of \emph{group-like elements} defined by 
\[G\coloneqq1+\F_1 A_0=\left\{1+x, \ x\in  \F_1 A_0\right\}\ ,\]
with basis of open sets at $1$ defined by $\{1+\F_k A_0\}_{k \ge 1}$.

\begin{definition}[Pre-Lie logarithm]
The \emph{pre-Lie logarithm} of a group-like element is defined by the convergent  Magnus expansion series
$$\ln(1+x)=\Omega(x)\coloneqq x - \frac12 x\star x + \frac14 x\star (x \star x)
+ \frac{1}{12}(x\star x)\star x+\cdots \ , $$
 see \cite{AgrachevGamkrelidze80, Chapoton02, Manchon11} for more details. 
\end{definition}

\begin{proposition}\label{prop:preLieExpLog}
The pre-Lie exponential and the pre-Lie logarithm are inverse filtered bijections 
$$\xymatrix@C=30pt{\mathrm{exp}\ : \F_1 A_0 \ar@{<->}[r]^(0.43){\cong} & G=1+\F_1 A_0\ :\ \ln\ .
}$$
\end{proposition}

\begin{proof}The pre-Lie Magnus expansion is known to be a formal inverse of the pre-Lie exponential, this means precisely that it is its inverse in the free complete left-unital pre-Lie algebra on one variable $a$ of weight $1$, which we denote by 
$\widehat{\mathrm{uPreLie}}(a)$. For any $x\in \F_1 A_0$, we consider the unique morphism of complete left-unital pre-Lie algebras $\widehat{\mathrm{uPreLie}}(a)\to A_0$, which sends $a$ to $x$. The image of the relation 
$e^{\ln(1+a)}=a$ is $e^{\ln(1+x)}=x$. 

From the formulas of the pre-Lie exponential and the pre-Lie logarithm, one can see that they send bijectively $\F_k A_0$ to $1+\F_k A_0$, for any $k\ge 1$, and vice versa. \end{proof}

Recall that, in any pre-Lie algebra, one defines  \emph{symmetric braces} operations by the following formulas:
$$\begin{aligned}
\{ a; \}\coloneqq&a \\
\{ a; b_1\}\coloneqq&a\star b_1 \\
\{ a; b_1,\ldots, b_n\}
\coloneqq&  \{ \{a; b_1,\ldots, b_{n-1}\}; b_n\}- \displaystyle \sum_{i=1}^{n-1} \{a; b_1, \ldots, b_{i-1}, 
\{b_i; b_n\}, b_{i+1}, \ldots, b_{n-1}\}   \ .
\end{aligned}
$$
Any elements $a\in A$ and $b\in \F_1 A$ satisfy $\{a; \underbrace{b, \ldots, b}_{n}\}\in \F_n A$. So we can  consider the following \emph{circle product} defined by the convergent series made up of the iterated   symmetric braces:
$$a \circledcirc (1+b) \coloneqq \sum_{n\ge 0}  {\displaystyle \frac{1}{n!}}  \{a; \underbrace{b, \ldots, b}_{n}\}\ ,$$
which is linear only on the left-hand side and unital, $a\circledcirc 1 =a$ and $1 \circledcirc (1+b)=1+b$. 

\begin{theorem}\label{thm:GaugeGrpBis}
The data 
$$\mathfrak{G}\coloneqq\left( G, \{1+\F_k A_0\}_{k \ge 1}, \circledcirc, 1\right)$$ forms a complete  group, which is filtered isomorphic, thus homeomorphic, to the gauge group under the pre-Lie exponential and the pre-Lie logarithm respectively.
\end{theorem}

\begin{proof}
The arguments of \cite[Section~$4$]{DotsenkoShadrinVallette16} hold true in the formal case, this is in the free complete left-unital pre-Lie algebra on two variables  of weight $1$, especially the proof of \cite[Proposition~$3$]{DotsenkoShadrinVallette16}. By the same argument as in the above proof of Proposition~\ref{prop:preLieExpLog}, the conclusions of  \cite[Theorem~$2$]{DotsenkoShadrinVallette16} hold in the present complete case: $\mathfrak{G}$ is isomorphic to the gauge group under the pre-Lie exponential and the pre-Lie logarithm respectively. Proposition~\ref{prop:preLieExpLog} implies that $\mathfrak{G}$ is a complete group, which is filtered isomorphic, thus homeomorphic, to the gauge group.
\end{proof}

\section{Deformation theory in the complete pre-Lie setting}\label{subsec:DefTheoCompPreLie}

\begin{definition}[Maurer--Cartan element]
A \emph{Maurer--Cartan element} in a dg pre-Lie algebra $\a=\left(A, d, \star, 1 \right)$ is  a degree $-1$ element 
$\bar{\alpha}$ satisfying the Maurer--Cartan equation:
 $$d(\bar{\alpha})+\bar{\alpha}\star \bar{\alpha}=0\ . $$
 \end{definition}

In some filtered dg pre-Lie algebras, like the convolution pre-Lie algebra considered in the next section, there exits an element $\delta \in  \F_0 A_{-1}$ such that $[\delta, x]=\delta \star x -(-1)^{|x|} x \star \delta=d(x)$. In this case, the differential is \emph{internal} and 
we consider instead the following Maurer--Cartan elements $\alpha\coloneqq\delta+\bar{\alpha}$ satisfying the equivalent square-zero equation $\alpha \star \alpha=0$. 

\begin{proposition}\label{prop:GaugeGroupAction}
Let $\a=\left(A,  \F, d, \star, 1 \right)$ be a complete left-unital dg pre-Lie algebra. The gauge group $\Gamma$ acts  on the set of Maurer--Cartan elements under the formula 
$$\lambda. \alpha\coloneqq e^{\ad_\lambda}({\alpha}) \quad \text{or \ equivalently} \quad 
\lambda. \bar{\alpha}\coloneqq\frac{e^{\ad_\lambda}-\id}{\ad_\lambda}(d\lambda)+e^{\ad_\lambda}(\bar\alpha)
\ .$$
This action is continuous, that is  the application 
\[
\begin{array}{clc}
\Gamma\times \mathrm{MC}(\a)& \to&\mathrm{MC}(\a)\\
(\lambda,  \bar{\alpha}) & \mapsto &\lambda. \bar{\alpha} 
\end{array}\]
is continuous in both variables. 
\end{proposition}

\begin{proof}
For simplicity of the proof, we work in the one-dimensional extension $A \oplus \k \delta$ of the complete left-unital dg pre-Lie algebra $\a$ mentioned above. Since the elements $\lambda$ of the gauge group $\Gamma$ live in $\F_1 A_0$, each term $\frac{1}{n!}\ad_\lambda^n(\alpha)$ of the series $e^{\ad_\lambda}({\alpha})$ live in $\F_n A_{-1}$. Therefore, this series converges and the action is well-defined. All the formulas for a group action holds formally, that is in the free complete left-unital pre-Lie algebra on several variables  of weight $1$, so they hold true in the complete case. 

Let us fix a Maurer--Cartan element $\alpha \in \mathrm{MC}(\a)$ and let us consider a  sequence $\{\lambda_n\}_{n\in \NN}$ of elements of the gauge group $\Gamma$ which tends to an element $\lambda$, that is 
\[\forall k\in \NN,\,  \exists N\in \NN, \, \forall n\ge N,\  \lambda_n-\lambda\in \F_k A_0\ . \]
So, we have $\forall k\in \NN,\,  \exists N\in \NN, \, \forall n\ge N, $
\[\begin{aligned}
e^{\ad_{\lambda_n}}({\alpha})- e^{\ad_\lambda}({\alpha})
&=\sum_{m=1}^\infty \frac{1}{m!}\left(\ad_{\lambda_n}^m -\ad_{\lambda}^m\right)(\alpha)\\
&=\sum_{m=1}^\infty \underbrace{\frac{1}{m!}
\left(
\sum_{l=1}^m \ad_{\lambda}^{l-1} \circ \ad_{\lambda_n-\lambda}\circ \ad_{\lambda_n}^{m-l} \right)(\alpha)}_{\in \F_{k+m-1} A_{-1}}\in \F_{k} A_{-1} \ .
\end{aligned}
\]
The continuity in the second variable is proved similarly.
\end{proof}

\begin{theorem}\label{thm:GaugeActionBis}
In any complete left-unital dg pre-Lie algebra, the gauge group $\mathfrak{G}$ acts continuously on the set of Maurer--Cartan elements by the formula: 
$$e^\lambda\cdot\at\coloneqq\lambda.\alpha= \left(e^\lambda \star \at\right) \circledcirc e^{-\lambda}\ .$$
\end{theorem}

\begin{proof}
This is a direct corollary of Theorem~\ref{thm:GaugeGrpBis}, Proposition~\ref{prop:GaugeGroupAction}, and \cite[Proposition~$5$]{DotsenkoShadrinVallette16}.
\end{proof}

\section{Operadic deformation theory in the complete setting}\label{sec:ComOpDefTh}
We aim to apply the results of the previous section  to the complete convolution algebra associated to a filtered dg   cooperad $\calC$ and the complete endomorphism operad of a complete dg module $A$: 
$$\a_{\, \calC, A}\coloneqq
\big(
\hom_\Sy\big(\calC, {\eend}_A\big), \partial, \star, 1 
\big)\ .
$$
Notice that, in this case, the internal differential element is given by 
$\delta\ : \  \calC \to \I\to \k \partial_A$, where the map $\calC \to \I$ is the counit of the cooperad $\calC$.
The interest in this example of application lies in the following interpretation of the Maurer--Cartan elements of this convolution algebra.
 
\begin{proposition}\label{prop:MCOmegaC}
Let $\calC$ be a filtered  dg cooperad and let $A$ be a complete dg module. The set of Maurer--Cartan elements of the complete convolution algebra $\a_{\, \calC, A}$ is in natural one-to-one correspondence with the complete $\Omega \calC$-algebra structures on $A$. 
\end{proposition}

\begin{proof}
Let us start with the case when the filtered dg cooperad $\calC$ is coaugmented. 
In this case,  the result holds true in the classical discrete setting \cite[Theorem~$6.5.7$]{LodayVallette12} and we use the same arguments and computations here. By Theorem~\ref{thm:CompleteAlg}, a complete $\Omega \calC$-algebra structures on $A$ is a morphism $\Omega \calC \to {\eend}_A$ of filtered dg operads. The underlying operad of the cobar construction of $\calC$ is a free operad generated by the desuspension of $\overline{\calC}$. Since $\calC$ is filtered, this free operad is equal to the filtered free operad on the same space of generators; this can be seen for instance through the formula \cite[Proposition~$5.6.3$]{LodayVallette12}. Therefore we get:
\[
\Hom_{\textsf{CompOp}}\left(\widehat{\Omega \calC}, {\eend}_A\right)\cong
\Hom_{\textsf{FilOp}}\left(\Omega \calC, {\eend}_A\right)\cong \hom_{\Sy}\left(\overline{\calC}, {\eend}_A\right)_{-1}\ , 
\]
where the compatibility with the differentials on the left-hand side corresponds to the Maurer--Cartan equation on the right-hand side, by the same computation as in the classical case. 

When the cooperad $\calC$ is not coaugmented, like in the examples of \cref{sec:GaugeTwist}, we consider the cobar construction $\Omega \calC \coloneqq \calT(s^{-1}\calC)$ with similar differential induced by the internal differential of $\calC$ and its partial decomposition maps. In this case,  Maurer--Cartan elements in the convolution algebra associated to a complete graded module $A$ correspond to complete $\Omega \calC$-algebra structures. 
\end{proof}

\begin{remark}
This proposition applies to the Koszul dual dg cooperad $\calC\coloneqq\calP^{\ac}$ of a Koszul operad $\calP$. In this case, Maurer--Cartan elements of the convolution algebra $\a_{\, \calP^{\ac}, A}$ are nothing but complete $\calP_\infty$-algebra structures on $A$. This way, we can develop the deformation theory of complete $\Ai$-algebras or complete $\Li$-algebras for instance, see next section. 
\end{remark}

The gauge group obtained by integrating the underlying complete pre-Lie algebra indeed acts of the set of Maurer--Cartan elements, but we can actually consider the following bigger gauge group. 

\begin{definition}[Deformation gauge group]
The \emph{deformation gauge group} associated 
to a filtered dg cooperad $\calC$ and a  complete dg module $A$ is defined by 
\begin{align*}
\widetilde{\Gamma}\coloneqq&\left(
\calF_1 \hom\big(\calC(0), {\eend}_A(0)\big)_0\times 
\calF_1 \hom\big(\calC(1), {\eend}_A(1)\big)_0\times \right.\\
&\left.{\displaystyle{\prod}_{n\ge 2}} \hom_{\Sy_n}\big(\calC(n), {\eend}_A(n)\big)_0
,  \BCH(\; ,\,), 0
\right)\ .
\end{align*}
\end{definition}

Notice that 
\[
\widetilde{\Gamma}
\supset 
\Gamma=\calF_1\hom_\Sy\big(\calC, {\eend}_A\big)_0\ .
\]
We define the complete filtration of the deformation gauge group by $\F_1 \widetilde{\Gamma}\coloneqq\widetilde{\Gamma}$ and 
$\F_k \widetilde{\Gamma}\coloneqq\F_k {\Gamma}$, for all $k\ge 2$.
Similarly, we consider the following bigger set of group-like elements 
\begin{align*}
\widetilde{G}\coloneqq1+\widetilde{\Gamma}=&
\calF_1 \hom\big(\calC(0), {\eend}_A(0)\big)_0\times 
\left(1+\calF_1 \hom\big(\calC(1), {\eend}_A(1)\big)_0\right)\times \\ &
{\displaystyle{\prod}_{n\ge 2}} \hom_{\Sy_n}\big(\calC(n), {\eend}_A(n)\big)_0
\end{align*}
and its associated complete group structure
\[
\widetilde{\mathfrak{G}}\coloneqq\left(\widetilde{G}, \left\{1+\F_k \widetilde{\Gamma}\right\}_{k \ge 1},  \circledcirc, 1\right)\ .\]

\begin{proposition}\label{prop:Extension}
All the results of Sections~\ref{subsec:CompleteLieInt} and \ref{subsec:DefTheoCompPreLie} hold true on the level of the deformation gauge group $\widetilde{\Gamma}$ and its avatar $\widetilde{\mathfrak{G}}$ (except for the filtration property of the group commutators).
\end{proposition}

\begin{proof}
The only issue is the convergence of all the formulas with these more general objects. In the present case, the complete convolution algebra is graded by the arity and 
the various formulas involving the pre-Lie product $\star$, its Lie bracket, exponential and logarithm maps, as well as the associative product $\circledcirc$, are actually given by sums of labelled trees. 
So, at each fixed arity, they are made up of 
finite sums of terms of arity greater than $2$ and potentially infinite sums of terms with a finite number of terms of arity greater than $2$ and an infinite numbers of terms of arity $0$ and $1$. These latter infinite sums converge by the restriction to parts of filtration $1$ in arity $0$ and $1$ in the abovementioned definitions. 
\end{proof}

Let $\alpha$ and $\beta$ be two Maurer--Cartan elements corresponding to two $\Omega \calC$-algebra structures on $A$. 
 
 \begin{definition}[$\infty$-morphism]\label{def:InftyMorph}
An  \emph{$\infty$-morphism} $\alpha \rightsquigarrow \beta$ between $\alpha$ and $\beta$ is a  degree $0$ element 
$f : \calC \to {\eend}_A$, such that $f_{0}\in \calF_1 \hom\big(\calC(0), {\eend}_A(0)\big)_0$ and satisfying the  equation
\begin{eqnarray}\label{eqn=InftyMor}
f \star \at = \beta \cc f\ . 
\end{eqnarray}
\end{definition}
They can be composed under the formula $f \cc g$. 

\begin{remark}\label{RkTrick}
This definition includes the case where $\alpha$ corresponds to an $\Omega \calC$-algebra structure on a complete dg module $(A,\F)$ and $\beta$  an $\Omega \calC$-algebra structure on another complete dg module $(B,\G)$. In this case, we consider the convolution algebra $\hom_\Sy\big(\calC, {\eend}_{A\oplus B}\big)$ with $\alpha$ (respectively $\beta$) a Maurer--Cartan element supported on $\eend_A$ (respectively on $\eend_B$). This produces the notion of an $\infty$-morphism from $(A, \alpha)$ to $(B, \beta)$. One can notice that Equation~(\ref{eqn=InftyMor}) imposes, in general, that such an $\infty$-morphism takes values in $\hom(A^{\otimes n}, B)$, for $n\in \NN$.
\end{remark}

\begin{remark}
We need the assumption $f_{0}\in \calF_1 \hom\big(\calC(0), {\eend}_A(0)\big)_0$ for the right-hand side $\beta \cc f$ to be well-defined. Notice that in the present text, we consider only  cooperads defined by the partial definition $\Delta_{(1)} : \calC \to \calC\, \widehat{\circ}_{(1)}\, \calC$, which splits operations into two. When the cooperad $\calC$ has trivial arity $0$ part, i.e. $\calC(0)=0$, it admits a ``full'' coassociative decomposition map  $\Delta : \calC \to \calC\, \widehat{\circ} \, \calC$, which splits operations into two levels, see \cite[Section~$5.8.2$]{LodayVallette12}. 
In this case, the associative product $\beta \cc f$ receives the following simple description: 
\[
\beta\cc f : \calC
 \xrightarrow{\Delta} \calC\, \widehat{\circ} \, \calC
  \xrightarrow{\beta \widehat{\circ} f}  \eend_A \widehat{\circ} \eend_A
  \xrightarrow{\gamma}  \eend_A \ .
  \]
However, the full decomposition map for the cooperad $\calC$ fails to be  well-defined in the general case, but thanks to the assumption $f_{0}\in \calF_1 \hom\big(\calC(0), {\eend}_A(0)\big)_0$, the product $\cc$ is well-defined, by a convergent series. Therefore, the present complete pre-Lie integration theory is mandatory to define a suitable notion of $\infty$-morphism for the homotopy algebras having particular elements, like curved $\Ai$-algebras and curved $\Li$-algebras considered in the next section. 
\end{remark}

\begin{definition}[$\infty$-isotopy]
An \emph{$\infty$-isotopy} between $\alpha$ and $\beta$ is an $\infty$-mor\-phi\-sm $f : \calC \to {\eend}_A$ satisfying  
\[f_1\in 1+\calF_1 \hom\big(\calC(1), {\eend}_A(1)\big)_0\ .\]
\end{definition} 

\begin{theorem}\label{thm:DeligneGroupoidII}
For a filtered dg cooperad $\calC$ and for any complete dg module $A$, the set of  $\infty$-isotopies 
forms a subgroup of the group of  invertible $\infty$-morphisms of $\Omega \calC$-algebra structures on $A$, which is isomorphic to the deformation gauge group under the pre-Lie exponential map
$$ \widetilde{\Gamma} \cong \widetilde{\mathfrak{G}}= (\infty\textsf{-}\mathsf{iso}, \cc, \id_A)  \ . $$ 
The Deligne groupoid associated to the deformation gauge group  is isomorphic to the groupoid whose objects are $\Omega \calC$-algebras and whose 
morphisms are  $\infty$-isotopies
$$\mathsf{Deligne}\left(\a_{\calC,A}\right)\coloneqq\left(\mathrm{MC}(\a_{\calC,A}), \widetilde{\Gamma}
\right)\cong
\left(\Omega\calC\textsf{-}\mathsf{Alg}, \infty\textsf{-}\mathsf{iso}\right)
\ .$$
\end{theorem}

\begin{proof}
This is a direct corollary of Proposition~\ref{prop:MCOmegaC}, Theorem~\ref{thm:GaugeGrpBis} and Theorem~\ref{thm:GaugeActionBis}, via their generalisations given at Proposition~\ref{prop:Extension}. 
\end{proof}

\begin{remark}
Notice that, for a weight graded dg cooperad $\calC$ and a dg module $A$, both viewed as complete objects with respectively the weight filtration and the trivial filtration, the present constructions agree with the ones given in 
\cite[Section~$5$]{DotsenkoShadrinVallette16}.  So the present treatment is a strict generalisation of \textit{loc. cit.}
\end{remark}

\newpage
\chapter{The gauge action origin of the twisting procedure}\label{sec:GaugeTwist}

In this section, we provide a first application of the complete operadic deformation theory developped in the previous section. Namely, we deal with the easiest example of gauge action: when the gauge element is just an element of the underlying dg module, that is concentrated in arity $0$. 

In this way, we get a conceptual interpretation of  the  twisting procedure of  a complete $\calA_\infty$-algebra (or a complete $\calL_\infty$-algebra) by a Maurer--Cartan element. This result belongs to the series of results established recently by the authors which show that ``any functorial procedure'' producing new homotopy algebra structures from former ones are actually given by a suitable gauge action: Koszul hierarchy \cite{Markl15, DotsenkoShadrinVallette16}, Homotopy Transfer Theorem \cite{DotsenkoShadrinVallette16}, etc. . 

The procedure of twisting  a complete $\calA_\infty$-algebra with a Maurer--Cartan element is the non-commutative analogue of the twisting procedure  for complete $\calL_\infty$-algebras. The former one plays a seminal role in the construction of the Floer cohomology of Lagrangian submanifolds  \cite{FOOO09I} and the latter one is used in crucial ways in deformation theory, rational homotopy theory, higher algebra, and quantum algebra. Since the $\calA_\infty$-case is simpler, we first  recall it. Then we encode it conceptually in the operadic deformation theory language. This allows us to fix explicitly the sign conventions for instance. 

From the conceptual gauge action interpretation of the twisting procedure, we derive automatically ``all'' its known and useful properties. We conclude this section with a criterion on quadratic operads which ensures that the associated category of homotopy algebras admits a meaningful twisting procedure.

\section{Curved $\Ai$-algebras}\label{sec:CurvedAiAlg}

Let us start with the most simple ns operad: the ns operad $\uAs$, which encodes unital associative algebras. It is one-dimensional in each arity $\uAs(n)=\k \upsilon_n$, for $n\ge 0$, and concentrated in degree $0$; its operadic structure is straightforward: $\upsilon_k\circ_i\upsilon_l=\upsilon_{k+l-1}$\ . Let us now consider the linear dual ns cooperad $\uAs^*$, where we denote the dual basis by $\nu_n\coloneqq\upsilon_n^*$. The infinitesimal decomposition coproduct, which is the dual of the partial composition products,
 is equal to  
$$\Delta_{(1)}(\nu_n)=\sum_{p+q+r=n\atop p,q,r \ge 0} \nu_{p+1+r} \circ_{p+1} \nu_q\ .$$

We now consider the endomorphism ns operad $\End_A\coloneqq\left\{\Hom(A^{\otimes n}, \allowbreak A)\right\}_{n\ge 0}$ associated to any graded module $A$. 
 It turns out that the convolution ns operad $\Hom(\uAs^*, \End_A)$  is canonically isomorphic to $\End_A$ itself. Therefore the induced left unital pre-Lie algebra 
 \[\mathfrak{a}\coloneqq\left(\prod_{n\ge 0} \Hom(A^{\otimes n}, A),  \star, 1\right)\] is given by 
$$1=\id_A \quad \text{and} \quad f\star g = \sum_{i=1}^n f\circ_i g  \ ,$$
for $f\in \Hom(A^{\otimes n}, A)$.\\

Its Maurer-Cartan elements are degree $-1$ elements $\alpha=(m_0, m_1, \ldots, m_n,\ldots)$ satisfying the Maurer--Cartan equation 
$$\alpha\star \alpha=0\ . $$ 
Unfolding this definition, we have degree $-1$ operations $m_n : A^{\otimes n} \to A$, for $n\ge 0$, satisfying the following relations, under the convention $\theta\coloneqq m_0(1)\in A_{-1}$ and $d\coloneqq m_1$:
\begin{eqnarray*}
&\text{arity}\ 0:& d(\theta)=0\ , \\
&\text{arity}\  1:& d^2=- m_2(\theta, -)-m_2(-, \theta)\ , \\
&\text{arity}\  2:& \partial m_2= dm_2+m_2(d(-), -)+m_2(-,d(-))=
-m_3(\theta, -,-)-m_3(-,\theta, -)\\ &&\qquad\quad\ -m_3(-,-,\theta)\ . \\
&\text{arity}\  3:& \partial m_3=-m_2(m_2(-,-), -)-m_2(-,m_2(-,-))-m_4(\theta, -, -,-)\\&&\qquad\quad\ 
-m_4(-,\theta,  -,-)
-m_4(-,-,\theta, -)
-m_4(-,-,-,\theta)\ ,
\\
&\text{arity}\ n:& \partial\left( m_n\right)= -\sum_{p+q+r=n\atop 2\leq q \leq n}m_{p+1+r}\circ_{p+1} m_q
-\sum_{i=1}^{n+1} m_{n+1}(-, \cdots, -,  \underbrace{\theta}_{i^\text{th}}, -,\cdots, -)\ .
\end{eqnarray*}

\begin{definition}[Shifted curved $\Ai$-algebra]
The data  $(A, \theta, d, m_2, m_3, \ldots)$ of a graded module $A$ equip\-ped with a degree $-1$ element $\theta$ and degree $-1$ maps $d, m_2, m_3,\ldots$ satisfying the aforementioned equations is called a \emph{shifted curved $\Ai$-algebra}. The element $\theta$ is called the \emph{curvature}. 
\end{definition}

When the curvature $\theta$ vanishes, the operator $d$ squares to zero and thus gives rise to a differential, which is a derivation with respect to the binary operation $m_2$. The higher operations $m_n$, for $n\ge 3$, are then homotopies for the relation $-\sum_{p+q+r=n\atop 2\leq q \leq n}m_{p+1+r}\circ_{p+1} m_q=0$.

\begin{definition}[Shifted $\Ai$-algebra]
A chain complex $(A, d, m_2, m_3, \ldots)$ eq\-ui\-pp\-ed with  degree $-1$ maps $ m_2, m_3,\ldots$ satisfying the aforementioned equations with $\theta=0$ is called a \emph{shifted $\Ai$-algebra}. 
\end{definition}

Given a shifted curved $\Ai$-algebra structure $(A, \theta', d', m'_2, m'_3, \ldots)$, one can choose to work on the desuspension $s^{-1}A$ of the underlying graded module. The degree of the operations and the signs involved in their relations will thus be modified. The curvature $\theta\coloneqq s^{-1}\theta'$ now has degree $-2$, the unary operator $d : s^{-1}A \to s^{-1}A$ still has degree $-1$, and the operations $m_n : (s^{-1}A)^{\otimes n} \to s^{-1}A$ now have degree $n-2$. They satisfy the following signed relations: 
\begin{eqnarray*}
&\text{arity}\ 0:& d(\theta)=0\ , \\
&\text{arity}\  1:& d^2=m_2(\theta, -)-m_2(-, \theta)\ , \\
&\text{arity}\  2:& \partial m_2=
-m_3(\theta, -,-)+m_3(-,\theta, -)-m_3(-,-,\theta)\ . \\
&\text{arity}\  3:& \partial m_3=m_2(m_2(-,-), -)-m_2(-,m_2(-,-))+m_4(\theta, -, -,-) \\&&\qquad\quad\ 
-m_4(-,\theta,  -,-)
+m_4(-,-,\theta, -)
-m_4(-,-,-,\theta)\ ,
\\
&\text{arity} \ n:& \partial\left( m_n\right)= \sum_{p+q+r=n\atop 2\leq q \leq n}(-1)^{pq+r+1} m_{p+1+r}\circ_{p+1} m_q
\\&&\qquad\quad\ +\sum_{i=1}^{n+1} (-1)^{n-i} m_{n+1}(-, \cdots, -,  \underbrace{\theta}_{i^\text{th}}, -,\cdots, -)\ .
\end{eqnarray*}

\begin{definition}[Curved $\Ai$-algebra]\label{def:CurvedAinfty}
The data of a graded module $(A, \theta, \allowbreak d, \allowbreak m_2,\allowbreak m_3, \ldots)$ equipped with a degree $-2$ element $\theta$, a degree $-1$ map $d: A\to A$  and degree $n-2$ maps $m_n : A^{\otimes n}\to A$ satisfying the aforementioned equations is called a \emph{curved $\Ai$-algebra}. 
\end{definition}

Again, if in a curved $\Ai$-algebra, the curvature vanishes, then the operator $d$ becomes a differential and the higher operations can be interpreted as homotopies for the signed relations. 

\begin{definition}[$\Ai$-algebra]
A chain complex $(A, d, m_2, m_3, \ldots)$ equipped with maps $m_n : A^{\otimes n}\to A$ degree $n-2$, for $n\ge 2$, satisfying the aforementioned equations with $\theta=0$ is called an \emph{$\Ai$-algebra}. 
\end{definition}

\begin{remark}
To get the same notions in the filtered or complete setting, one just has to replace everywhere the endomorphism operad $\End_A$ by its the complete endomorphism suboperad $\eend_A$, made up of filtered maps.
\end{remark}

\section{Operadic interpretation}\label{subsec:OpInt}

Notice that the  ns operad $\uAs$ is actually isomorphic to the endomorphism ns operad associated to  a degree $0$ dimension $1$ module: $\uAs\cong \End_\k$. Similarly, the ns operad $\As$, which encodes (non-necessarily unital) associative algebras is isomorphic to $\As\cong \overline{\End}_\k$, where $\overline{\End}_\k$ is the ns suboperad 
of $\End_\k$ with trivial arity $0$ component: $\overline{\End}_\k(0)=0$. One can also consider 
the endomorphism ns operad associated to  a degree $1$ dimension $1$ module $\End_{\k s}$ and its ns suboperad 
$\overline{\End}_{\k s}$. The  abovementioned four definitions are obtained by considering the Maurer--Cartan elements of the convolution algebras associated the following linear dual cooperads: 
\[
\begin{array}{|l|c|c|}
\hline
& \text{curved} & \text{uncurved} \\
\hline
\rule{0pt}{11pt} \text{shifted} & {\End}_{\k} &\overline{\End}_{\k}\\
\hline
\rule{0pt}{11pt} \text{classical} &   {\End}_{\k s}& \overline{\End}_{\k s}\\
\hline
\end{array}
\]

The convolution pre-Lie algebra  $\Hom(\As^*, \End_A)\cong \prod_{n \ge 1}\lbrace\Hom(A^{\otimes n}, A)\rbrace$, wh\-ich encodes shifted $\Ai$-algebras,  is a pre-Lie subalgebra of $\Hom(\uAs^*, \End_A)\cong \allowbreak \prod_{n \ge 0}\lbrace\Hom(A^{\otimes n}, A)\rbrace$, which encodes shifted curved $\Ai$-algebras. But, as we will show in Section~\ref{subsec:TwistGrp}, it is crucial to encode the former notion in the latter bigger pre-Lie algebra, since there, the gauge group of symmetries is big enough to host the twisting procedure. 

\begin{remark}
Starting from the ns operad $\uAs$ encoding \emph{unital} associative algebras, one gets the notion of shifted \emph{curved} $\Ai$-algebras through the abovementioned process. This admits a conceptual explanation via the Koszul duality theory \cite{Positselski11, HirshMilles12}. 
\end{remark}

\section{The twisting procedure as gauge group action}\label{subsec:TwistGrp}
In this section, we pass to the complete setting and we  apply the general theory developed in Section~\ref{sec:TopoDefTh} to the discrete so complete ns cooperad  $\C=\End^c_{\k s^{-1}}\coloneqq\End_{\k s}^*$. This will allow us to  treat the 
deformation theory of curved $\Ai$-algebras. We chose this particular case, since the sign issue is a complicated problem in operad theory; the reader interested in the shifted case  has just to remove ``all'' the signs. 
The  cooperad $\calC$  is spanned by  one element 
$$\nu_n : \left(s^{-1}\right)^n\mapsto (-1)^{\frac{n(n-1)}{2}} s^{-1}$$ 
of degree $n-1$, in each arity $n\ge 0$. Its infinitesimal decomposition coproduct is given by 
$$\Delta_{(1)}\left(\nu_n\right)
=\sum_{p+q+r=n\atop  p,q,r \ge 0}(-1)^{p(q+1)} \nu_{p+1+r}\circ_{p+1} \nu_q
\ . $$

\begin{remark}
With the sign convention $\nu_n : \left(s^{-1}\right)^n\mapsto  s^{-1}$, we get the same signs as the ones of  \cite[Chapter~9]{LodayVallette12}. With the present convention, we actually get the signs which are more common in the existing literature. 
\end{remark}

To any complete graded module $A$, we  associate the complete graded left-unital convolution pre-Lie algebra 
$$\a_{\, \calC, A}\coloneqq
\big(
\hom_{\mathbb{N}}\big(\calC, {\eend}_A\big), \star, 1 
\big)\ ,$$
whose underling complete graded module is isomorphism to the product $\prod_{n\in \NN}\allowbreak s^{ 1-n}\allowbreak \hom \big(A^{{\otimes} n}, A\big)$.
As a consequence of Section~\ref{sec:CurvedAiAlg} and Proposition~\ref{prop:MCOmegaC}, its Maurer--Cartan elements, that is degree $-1$ maps $\alpha : \calC \to {\eend}_A$ satisfying $\alpha\star \alpha =0$,
are in one-to-one correspondence with 
complete curved $\Ai$-algebra structures on $A$, under the assignment $m_n\coloneqq\alpha(\mu_n)$, for $n\ge 0$.\\

The deformation gauge group associated to $\a$ is equal to 
\[\widetilde{\Gamma}\cong\left(
\F_1 A_{-1} \times 
\calF_1 \hom(A,A)_0\times 
{\displaystyle{\prod}_{n\ge 2}} \hom \big(A^{{\otimes} n}, A\big)_{n-1}
,  \BCH(\; ,\,), 0
\right)
\]
and is filtered isomorphic to 
\[
\widetilde{\mathfrak{G}}\cong 
\left(
\F_1 A_{-1} \times 
\left(1+ \calF_1 \hom(A,A)_0\right)
\times 
{\displaystyle{\prod}_{n\ge 2}} \hom \big(A^{{\otimes} n}, A\big)_{n-1}
,  \circledcirc, 1
\right)\]
under the pre-Lie exponential and pre-Lie logarithm maps, 
by Section~\ref{sec:ComOpDefTh}. In the complete left-unital  pre-Lie algebra $\a$, this deformation  gauge group acts on Maurer--Cartan elements via the following formula of  Theorem~\ref{thm:GaugeActionBis}: 
$$ e^\lambda\cdot\at=\left(e^\lambda \star \at\right) \circledcirc e^{-\lambda}  \ ,$$
as long as $\lambda_0$ and $\lambda_1$ live in the first layer of the filtration, that is $\lambda_0\in \F_1 A_{-1}$ and $\lambda_1\in \calF_1 \hom(A,A)_0$. \\

\begin{remark}
In the monograph \cite{Positselski12}, L. Positselski developed the theory of complete curved $\Ai$-algebras over a complete local ring. He introduces a general notion of $\infty$-morphism which we recover via Definition~\ref{def:InftyMorph}, where 
the above  condition on $f_0$ coincides with the requirement that $f_0$ lives in the maximal ideal.

The comprehensive book \cite{Markl12} of M. Markl on deformation theory treats  the case of free complete modules over a complete local ring.  A gauge group is introduced in this context, see Chapter~$4$ of \textit{loc. cit.}, but it is only made up of $1+ \calF_1 \hom(A,A)_0$ . This is enough to describe the moduli spaces  of associative algebra structures up to isomorphism, in the compete setting.  The present deformation gauge group describes faithfully the moduli spaces of complete curved $\Ai$-algebras up to their $\infty$-isotopies.
\end{remark}

Let us now study the first and easiest example of a gauge action on Maurer--Cartan elements: 
we consider elements of the deformation gauge group $\widetilde{\Gamma}$ supported in arity $0$. Let $a\in \F_1 A_{-1}$, for brevity, we still denote by $a$ the element 
$(a, 0, \ldots)$
of $\widetilde{\Gamma}$ and by $1+a=e^a=(a, 1, 0, \ldots)$ the element of 
$\widetilde{\mathfrak{G}}$. 
By the general theory developed above, the action $e^a \cdot \alpha$ on the complete curved $\calA_\infty$-algebra structure encoded by the Maurer--Cartan element $\alpha$  gives us automatically a new complete curved $\calA_\infty$-algebra structure. 

\begin{proposition}\label{prop:TwCurvedGauge}
The formula for the generating operations $m_n^a$ of the complete curved $\calA_\infty$-algebra $e^a\cdot\alpha$ is
$$m_n^a=\sum_{r_0, \ldots, r_n\ge 0} (-1)^{\sum_{k=0}^n kr_k} m_{n+r_0+\cdots+r_n}\big(a^{r_0}, -, a^{r_1}, -, \ldots,  - , a^{r_{n-1}}, -,a^{r_n}  \big)\ , $$
for $n\ge 0$, where the notation $a^r$ stands  for $a^{\otimes r}$. 
\end{proposition}

\begin{proof}
In the present case, the inverse of the pre-Lie exponential $e^a$ is equal to $e^{-a}=1-a$. Therefore, the formula of the gauge action given in Theorem~\ref{thm:GaugeActionBis} is 
\begin{eqnarray*}e^{-a}\cdot\at=\left(e^{-a} \star \at\right) \circledcirc e^{a} =
\big((1-a) \star \alpha \big) \circledcirc (1+a)=\alpha\circledcirc (1+a)\ ,
\end{eqnarray*}
since $\lambda\star\rho=0$ for any $\rho\in \a$ and since $\star$ is linear on the left-hand side.
The image of the element $\nu_n$ under the $\alpha\circledcirc (1+a)$ is equal to 
$\big(\alpha\circ (1+a)\big)\big(\Delta(\nu_n)\big)$. One can easily see that the part of the image of the element $\nu_n$ under the decomposition map $\Delta$ of the cooperad $\calC$ with only $\nu_0$ and $\nu_1$ on the right-hand side is equal to 
$$
\sum_{r_0, \ldots, r_n\ge 0} (-1)^{\sum_{k=0}^n kr_k} \nu_{n+r_0+\cdots+r_n}\circ \big(\nu_0^{r_0}, \nu_1, \nu_0^{r_1}, \nu_1, \ldots,  \nu_1 , \nu_0^{r_{n-1}}, \nu_1,a^{r_n}  \big)
$$
Finally, the sign appearing in the formula for the $m_n^a$ is the same since the element $1+\lambda$ has degree~$0$. 
\end{proof}

Explicitly, the first of these twisted operations are:
\begin{eqnarray*}
&\text{arity} \ 0\ :& \theta^a\coloneqq m_0^a=\theta+d(a)+m_2(a,a)+m_3(a,a,a)+\cdots\ , \\
&\text{arity} \ 1\ :& d^a\coloneqq m_1^a=d(-)+m_2(a,-)-m_2(-,a)+m_3(a,a,-)-m_3(a,-,a)\\ && \quad \quad\ \   +m_3(-,a,a,)+\cdots\ , \\
&\text{arity} \ 2\ :& m_2^a=m_2(-,-)+m_3(a,-,-)-m_3(-,a,-)+m_3(-,-,a)+\cdots \ .
\end{eqnarray*}

\begin{remark}
Notice the strategy chosen here: first we get a new complete curved $\Ai$-algebra structure by a conceptual argument (gauge group action) and then we make it explicit. Usually, in the literature like in \cite[Chapter~3]{FOOO09I}, the explicit form of the twisted operations  is given first and then proved (by direct computations) to satisfy the relations of a complete curved $\Ai$-algebra. 
\end{remark}

\begin{theorem}[\cite{FOOO09I}]\label{thm:TwProcGp}
Under the formula of Proposition~\ref{prop:TwCurvedGauge}, any element $a\in \F_1 A_{-1}$ of 
a complete curved $\calA_\infty$-algebra $(A, \theta, d, m_2, m_3, \ldots)$ induces a (twisted) complete
curved $\calA_\infty$-algebra 
 \[
(A, \theta^a, d^a, m_2^a, m_3^a,\ldots)
 \]
This twisted complete curved $\calA_\infty$-algebra has a trivial curvature $\theta^a=0$, that is produces an  $\calA_\infty$-algebra, if and only if the element $a$ satisfies the Maurer--Cartan equation: 
\begin{eqnarray}
\theta+d(a)+\sum_{n\ge 2}m_n(a, \ldots, a)=0\ .
\end{eqnarray}
\end{theorem}

This conceptually explains why the twisting procedure on associative algebras or complete $\calA_\infty$-algebras requires the twisting element to satisfy the Maurer--Cartan equation. Without this condition, one would   \emph{a priori}  get a complete curved $\calA_\infty$-algebra. In other words, the (left-hand side of the) Maurer--Cartan equation of an element $a$ is the curvature of the complete curved $\calA_\infty$-algebra twisted by the element $a$. 

\begin{proposition}\label{lem:subGaugeGroup}
The following assignment defines a monomorphism of gr\-ou\-ps 
\begin{eqnarray*}
\left(\F_1 A_{-1}, +, 0\right) &\rightarrowtail& \widetilde{\Gamma}=\big(
\F_1 A_{-1} \times 
\calF_1 \hom(A,A)_0\times 
{\displaystyle{\prod}_{n\ge 2}} \hom \big(A^{{\otimes} n}, A\big)_{n-1}
,
\\ && \qquad\  \mathrm{BCH}(-,-), 0
\big)\\
a &\mapsto& (a, 0, 0, \ldots)\ .
\end{eqnarray*}

\end{proposition}

\begin{proof}
It is enough to check the compatibility  with respect to the group products, that is 
$$\mathrm{BCH}(a, b)=a+b\ ,$$
which holds true since $a$ and $b$ are supported in arity $0$: their brackets appearing in the Baker--Campbell--Hausdorff formula vanish. 
\end{proof}

\begin{corollary}\label{cor:twa+b}
Twisting a complete curved $\calA_\infty$-algebra $(A, \theta, d, m_2, m_3, \ldots)$ first by an element $a$ and then by an element $b$ amounts to twisting it by $a+b$. 
\end{corollary}

\begin{proof}
This is a direct corollary of the preceding Proposition~\ref{lem:subGaugeGroup}. 
\end{proof}

The presence of the curvature element plays another peculiar role in the deformation theory of curved $\Ai$-algebras over a field, as follows. 

\begin{proposition}[{\cite[Section~7.3]{Positselski11}\cite[Theorem~5.4]{ArmstrongClarke15}}]\label{prop:KPphenom}
Any  curved $\calA_\infty$-algebra $(A, \theta, d, m_2, m_3, \ldots)$, for which there exists a linear map 
$A \to \k$ sending $\theta$ to $1$,
is gauge equivalent to its  truncated  curved $\calA_\infty$-algebra $(A, \theta, 0, 0, \ldots, )$.
\end{proposition}

\begin{proof}
For the sake of completeness of the present paper, we provide here a proof in the language of the pre-Lie deformation. Notice that the computations of the following obstruction argument are very close to the arguments of \cite[Section~$7.3$]{Positselski11} and \cite[Theorem~5.4]{ArmstrongClarke15}. As above, we denote the Maurer--Cartan element encoding the curved $\calA_\infty$-algebra $(A, \theta, d, m_2, m_3, \ldots)$ in the convolution pre-Lie algebra $\a$ by $\alpha$, that is $m_n=\alpha(\nu_n)$, for $n\ge 0$. 
Let us denote by $\alpha_{n}\coloneqq(0, \ldots, 0, \nu_{n}\mapsto m_n, 0, \ldots)$, for $n\ge 0$, so that $\alpha=\sum_{n=0}^\infty \alpha_{n}$. 
Let us suppose that there exists an element $\lambda\in \widetilde{\Gamma}$,  such that $\lambda_0=0$ and $\lambda_1=0$, and which satisfies $\ln(1+\lambda).\alpha=\alpha_0$. 
Since $\ln(1+\lambda).\alpha=\left((1+\lambda) \star \alpha \right) \circledcirc (1+\lambda)^{-1}=\alpha_0$, we get the equation $(1+\lambda) \star \alpha =\alpha_0\circledcirc (1+\lambda)=\alpha_0$, that is 
$(1+\lambda_2+\lambda_3+\cdots)\star (\alpha_0+\alpha_1+\cdots)=\alpha_0$. This gives in arity $0$: $\alpha_0=\alpha_0$ and in arity $n\ge 1$: 
\begin{eqnarray}\label{Eqn:OBs}
\underbrace{\alpha_n+\lambda_2\star \alpha_{n-1}+\cdots +\lambda_n\star \alpha_1}_{\delta_n}+\lambda_{n+1}\star \alpha_0=0\ .      
\end{eqnarray}
 We now consider the chain complex \[C_\bullet\coloneqq\Hom\left(
 \calC(\bullet),
\Hom(A^{\otimes \bullet}, A) \right)\cong  \Hom(A^{\otimes \bullet}, A)\ , \] for $\bullet \ge 1$, with the boundary map given by $d(f)\coloneqq f\star\alpha_0$, which lowers the arity by one. Since $\at_0$ is a degree $-1$ element in a (right) pre-Lie algebra, one can see that $(f\star\alpha_0)\star \alpha_0=f\star(\alpha_0\star \alpha_0)=0=d^2(f)$.

Let us first show that this chain complex is acyclic. Let us denote by  
$\theta^*  : A \to \k$ the $\k$-linear map which sends  $\theta$ to $1$. With this map at hand, we construct the following contracting homotopy $h$. Let $\varphi_n$ be an element of  $C_n=\Hom\left(\calC(n),
\Hom(A^{\otimes n}, A)\right)$, so it given by $\nu_n\mapsto f_n$. The image $s(\varphi_n)$ under $h$ of the element $\varphi_n$ is given by the assignment $\nu_{n+1}\mapsto \theta^*\otimes f_n$. It is then straightforward to check that $hd+dh=\id_{C_\bullet}$.

 Then, one can prove by induction that the element $\delta_n\in C_n$ is a cycle, therefore it is a boundary element, that is there exists $\lambda_{n+1}\in C_{n+1}$ such that $-\lambda_{n+1}\star \alpha_0=\delta_n$, which is Equation~(\ref{Eqn:OBs}). 
To do so, one can see that 
\begin{eqnarray*}
\delta_n\star \alpha_0&=&
\at_n \star \at_0+\sum_{k=2}^n(\lambda_k\star \at_{n-k+1})\star\at_0\\
&=&
\at_n \star \at_0-\sum_{k=2}^n(\lambda_k\star \at_0)\star\at_{n-k+1}+\sum_{k=2}^n\lambda_k\star (\at_{n-k+1}\star\at_0)
\end{eqnarray*}
which is equal to the following quantity using the induction hypothesis:
\begin{eqnarray*}
\delta_n\star \alpha_0&=&
\at_n \star \at_0
+\sum_{k=2}^n\at_{k-1}\star\at_{n-k+1}
+\sum_{k=3}^n\sum_{l=2}^{k-1}(\lambda_l\star \at_{k-l})\star\at_{n-k+1}
\\&&+\sum_{k=2}^n\lambda_k\star (\at_{n-k+1}\star\at_0)\\
&=&\sum_{k=1}^{n}\at_{k}\star\at_{n-k}+\sum_{l=2}^n
\sum_{k=l+1}^{n+1}\lambda_l\star(\at_{k-l}\star\at_{n-k+1})=0\ ,
\end{eqnarray*}
by the Maurer--Cartan equation satisfied by $\at$.
\end{proof}

\cref{thm:TwProcGp} shows that if one acts with the arity $0$ elements of the gauge group, one gets the twisting procedure. 
This is used in \cite{FOOO09I} to get a new structure with trivial curvature and thus an underlying chain complex.
On the other hand, \cref{prop:KPphenom} shows that one can act with the elements of the gauge group supported in arity $\ge 2$ in order to   kill the operations except for the curvature.
This second behaviour is sometimes called the \emph{Kontsevich--Positselski vanishing phenomenon} in the literature; as we see, the argument that proves it is ``dual'' to the twisting procedure.
Notice that we do not need the completeness assumption since we consider the action of a gauge element supported in arity greater than $2$ and that this Kontsevich--Positselski vanishing phenomenon always holds true over a field. However, this sort of degenerate behavior is typically avoided for many interesting cases like the curved $\Ai$-algebras used when working with matrix factorizations, as argued in~\cite{ArmstrongClarke15}. 

\section{Curved $\Li$-algebras}
In the case of symmetric operads, we can start with the same kind of simple operad, that is one-dimensional in each arity with trivial symmetric group action. This operad $\mathrm{uCom}=\End_{\k}$ encodes the category of unital commutative (associative) algebras. Its linear dual $\mathrm{uCom}^*$ produces the following complete left unital convolution pre-Lie algebra 
\[\a=\hom_\Sy( \mathrm{uCom}^*, \End_A)\cong \left(\prod_{n\ge 0} \hom\big(A^{\odot n}, A\big),  \star, 1\right) \ ,\]
where $A^{\odot n}\coloneqq A^{\otimes n}/\Sy_n$ stands for the space coinvariants with respect to the symmetric group action; in order words, we are led to consider symmetric maps from $A^{\otimes n}$ to $A$. 

\begin{definition}[Complete shifted curved $\Li$-algebra]\label{def:csLiAlg}
A \emph{complete shifted curved $\Li$-algebra structure} on a complete graded module $(A, \F)$ is a Maurer--Cartan element 
$\alpha=(\ell_0, \ell_1, \ldots, \ell_n,\ldots)$ in the complete left unital pre-Lie convolution algebra $\a\cong \left(\prod_{n\ge 0} \hom\big(A^{\odot n}, A\big),  \star, 1\right)$. Such a data amounts to a collection of filtered maps 
$l_n\ : \ A^{\odot n} \to A$, 
of degree $-1$, for $n\ge 0$, satisfying the following relations, under the usual convention $\theta\coloneqq\ell_0(1)$ and $d\coloneqq\ell_1$:
\begin{eqnarray*}
&\text{arity}\ 0:& d(\theta)=0\ , \\
&\text{arity}\  1:& d^2=- \ell_2(\theta, -)\ , \\
&\text{arity}\  2:& \partial \ell_2=
-\ell_3(\theta, -,-)\ . \\
&\text{arity}\  3:& \partial \ell_3=-\ell_2(\ell_2(-,-), -)-\ell_2(\ell_2(-,-), -)^{(23)}-\ell_2(\ell_2(-,-), -))^{(132)}
\\&&
\qquad \quad -\ell_4(\theta, -, -,-)\ ,
\\
&\text{arity}\ n:& \partial\left(\ell_n\right)= -\sum_{p+q=n+1\atop 2\leq p,q \leq n}
\sum_{\sigma\in \mathrm{Sh}_{p,q}^{-1}}
 (\ell_{p+1}\circ_{1} l_q)^{\sigma}
-\ell_{n+1}(\theta, -, \cdots, -)\ ,
\end{eqnarray*}
where $ \mathrm{Sh}_{p,q}^{-1}$ denotes the set of the inverses of $(p,q)$-shuffles.
\end{definition}

\begin{remark}
As strange  that it may seem, the notion of \emph{shifted} Lie algebra structure seems to appear ``more naturally'' than its classical notion: the shifted Gerstenhaber's Lie bracket on the Hochschild cochain complex with the operadic grading, the shifted Lie bracket on the sheaves of polyvector fields, the shifted $\Li$-algebra formed by the Koszul hierarchy, the shifted Lie algebra made up of the Whitehead product on homotopy groups of topological spaces, the algebraic structure present in the Batalin--Vilkovisky formalism, etc.
\end{remark}

\begin{remark}
In the same way as for $\Ai$-algebras, the truncation and suspension versions of the endormorphism cooperad on a one-dimensional module  give rise to possibly non-shifted and possibly non-curved $\Li$-algebra structures. 
\end{remark}

In order to get the same abovementioned definitions but with symmetric operads, one has to consider the operad $\mathrm{uAss}$ made up of the regular representations of the symmetric groups, that is $\mathrm{uAss}(n)=\k[\Sy_n]\upsilon_n$. Then the associated complete convolution pre-Lie algebra is isomorphic to the previous one since 
$\hom_\Sy( \mathrm{uAss}^*, \End_A)\cong \hom( \mathrm{uAs}^*, \End_A)$. 
Let us denote by $\upsilon_n'$ and by $\nu_n'$ respectively the basis elements of $\mathrm{uCom}(n)$ and $\mathrm{uCom}^*(n)$. 
The morphism of operads 
$$ \mathrm{uAss}\to \mathrm{uCom} \ ,\qquad \upsilon_n^\sigma\mapsto \upsilon'_n$$
induces the following morphism of cooperads 
$$ \varsigma\ : \ \mathrm{uCom}^*\to \mathrm{uAss}^* \ ,\qquad \nu_n' \mapsto \sum_{\sigma \in \Sy_n}\nu_n^\sigma\ ,$$
which,  by pulling back, induces the following morphism of pre-Lie algebras 
$$ \varsigma^*\ : \ \hom_\Sy( \mathrm{uAss}^*, \End_A) \to \hom_\Sy( \mathrm{uCom}^*, \End_A) \ , \quad \alpha \mapsto \alpha\circ\varsigma\ . 
$$
In the case of the endormophism operad on a suspended module, which gives rise to the non-shifted versions, the morphism of operads is defined  similarly, but its linear dual produces $\mathrm{sgn}(\sigma)$ signs. 
Since a morphism of pre-Lie algebras preserves Maurer--Cartan elements, we obtain the following known result, but  in a straightforward way. 

\begin{proposition}[\cite{LadaMarkl95, FOOO09I}]
The symmetrisation $$\ell_n\coloneqq\sum_{\sigma\in \Sy_n} \, m_n^\sigma $$
of a complete shifted (curved) $\Ai$-algebra produces a complete shifted (curved) $\Li$-algebra. 

The antisymmetrisation $$\ell_n\coloneqq\sum_{\sigma\in \Sy_n}\mathrm{sgn}(\sigma) \, m_n^\sigma $$
of a complete (curved) $\Ai$-algebra produces a complete (curved) $\Li$-algebra. 
\end{proposition}

\section{Gauge action and categorical properties of curved $\Li$-algebras}\label{subsec:TwLii}
The above notions of complete (shifted) curved $\Li$-algebra play a seminal role in deformation theory, rational homotopy theory, and higher algebra. It is first the suitable source category for the Deligne--Hinich--Getzler $\infty$-groupoid \cite{Getzler09, Henriques08, DolgushevRogers15}, it serves as rational models for spaces of maps \cite{BuijsMurillo11, Lazarev13, Berglund15}, and its provides us with a suitable higher categorical enrichment to the categories of homotopy algebras \cite{DolgushevRogers17, DolgushevHoffnungRogers14}.  In each case, the twisting procedure, together with its various properties, form the main toolbox. In this section, we show that these properties are actually straightforward consequences of the  above gauge group interpretation. This also allows us to get the most general  version of all these results. \\

Let $(A, \F)$ be a complete graded module and let $\calC$ be the coooperad $\coloneqq \mathrm{uCom}^*$. The deformation gauge group associated to $\a$ is equal to 
\[\widetilde{\Gamma}\cong\left(
\F_1 A_{0} \times 
\calF_1 \hom(A,A)_0\times 
{\displaystyle{\prod}_{n\ge 2}} \hom \big(A^{{\odot} n}, A\big)_0
,  \BCH(\; ,\,), 0
\right)
\]
and is filtered isomorphic to 
\[
\widetilde{\mathfrak{G}}\cong 
\left(
\F_1 A_{0} \times 
\left(1+ \calF_1 \hom(A,A)_0\right)
\times 
{\displaystyle{\prod}_{n\ge 2}} \hom \big(A^{{\odot} n}, A\big)_0
,  \circledcirc, 1
\right)\]
under the pre-Lie exponential and pre-Lie logarithm maps.

\begin{proposition}\label{prop:TwCurvedGaugeLie}

The gauge action of an element $a\in \F_1 A_0$ on a complete shifted curved $\Li$-algebra $(A,\F, \theta=\ell_0, d=\ell_1, \ell_2, \ell_3, \ldots)$ produces the following \emph{twisted} shifted curved $\Li$-algebra structure on $A$:
$$\ell_n^a=\sum_{k\ge 0} {\textstyle \frac{1}{k!}} \ell_{k+n}\big(a^k, -,  \ldots,  -   \big)\ , \ \  \text{for}\  n\ge 0\ . $$
\end{proposition}

\begin{proof}
The proof is easy and identical to the one of Proposition~\ref{prop:TwCurvedGauge}: the new Maurer--Cartan element is equal to $e^{-a}\cdot \alpha=\alpha \circledcirc (1+a)$.
\end{proof}

For instance, the formula for the twisted curvature and the twisted (pre)dif\-fe\-ren\-tial are respectively 
\[\theta^a=\sum_{k\ge 0} {\textstyle \frac{1}{k!}} \ell_{k}\big(a^k \big) \qquad \text{and} \qquad
d^a=\sum_{k\ge 0} {\textstyle \frac{1}{k!}} \ell_{k+1}\big(a^k, - \big)\ .
 \]
An element $a\in \F_1 A_{0}$ is called a \emph{Maurer--Cartan} element in the complete shifted curved $\Li$-algebra $(A, \F, \theta, d, \ell_2, \ell_3, \cdots)$ when $\ell^a_0=0$. The same results as in the above case of complete (shifted) curved $\Ai$-algebras hold here \textit{mutatis mutandis}. Let us mention the following ones, which are heavily used in \textit{op. cit.}

\begin{corollary}[\cite{Getzler09, DolgushevRogers15, DolgushevRogers17}]
Let $(A,\F, \theta, d, \ell_2, \ell_3, \ldots)$  be a complete shifted curved $\Li$-algebra and let 
$a, b\in \F_1 A_0$ be two of its elements. The following formulas hold: 
\begin{enumerate}\itemsep3pt
\item $d\left(\theta^a\right)+\sum_{k\ge 1}{\textstyle \frac{1}{k!}} \ell_{k+1}\big(a^k, \theta^a \big)=0$\ , 
\item $d^a \circ d^a=-\ell_2^a(\theta^a,-)$\ ,
\item $\theta^{a+b}=\theta^a+d^a(b)+ \sum_{k\ge 2} {\textstyle \frac{1}{k!}} \ell_{k}^a\big(b^k)$ , 
\end{enumerate}
\end{corollary}

\begin{proof}\leavevmode
\begin{enumerate}
\item The left-hand side is nothing but $d^a\left(\theta^a\right)$, which is equal to $0$ by the arity $0$ relation of the twisted shifted curved $\Li$-algebra.
\item This formula is the  arity $1$ relation of the twisted shifted curved $\Li$-algebra.
\item The right-hand side is nothing but $\left(\theta^a\right)^b$; so the formula is the arity $0$ part of  Corollary~\ref{cor:twa+b}, in the $\Li$-algebra case.
\end{enumerate}
\end{proof}

Let $(A,\F)$ and $(B, \G)$ be a two complete dg modules. Let $\alpha, \beta \in \MC(\a_{\calC, A\oplus B})$ be two Maurer--Cartan elements  corresponding to two complete shifted curved $\Li$-algebra structures on $(A, \F)$ and $(B, \G)$ respectively. Let $f=(f_0, f_1, \cdots)$ be an $\infty$-morphism from $(A,\alpha)$ to $(B, \beta)$, that is $f\star \alpha=\beta \circledcirc f$. Let $a\in \F_1A_0$ and let us denote by 
\begin{align*} 
f^a\coloneqq f\circledcirc(1+a)=&\big(f(a)\coloneqq\sum_{k\ge 0} {\textstyle \frac{1}{k!}} f_k(a^k), 
f^a_1\coloneqq\sum_{k\ge 0} {\textstyle \frac{1}{k!}} f_{k+1}\big(a^k, - \big), \cdots , 
\\
&\ \ f^a_n\sum_{k\ge 0} {\textstyle \frac{1}{k!}} f_{k+n}\big(a^k, -,  \ldots,  -   \big)
, \cdots
\big)\ .
\end{align*}
Notice that $f(a)\in \F_1 B_0$. We  consider the  following two twisted complete shifted curved $\Li$-algebras on $A$ and $B$ respectively:
\[
\alpha^a\coloneqq(1-a)\cdot \alpha=\alpha\cc(1+a) \quad \text{and}\quad 
\beta^{f(a)}\coloneqq\big(1-{f(a)}\big)\cdot \beta=\beta\cc\big(1+{f(a)}\big)\ .
\]

\begin{proposition}\leavevmode
\begin{enumerate}
\item The element 
\[ \big(1-{f(a)}\big) \cc f \cc (1+{a})=f^a-{f(a)}\]
is an $\infty$-morphism from $\alpha^a$ to $\beta^{f(a)}$. 

\item The curvatures of the two twisted complete shifted curved $\Li$-algebra structures are related by the following formula:
\[
\beta^{f(a)}_0=f_1^a\big(\alpha^a_0\big)=\sum_{k\ge 0} {\textstyle \frac{1}{k!}} f_{k+1}\big(a^k, \alpha^a_0 \big)
 \ .\] 

\item If the element $a$ is a Maurer--Cartan element in the complete shifted curved $\Li$-algebra $\alpha$, then so is its ``image'' $f(a)$ in the complete shifted curved $\Li$-algebra $\beta$. In this case, $f^a-f(a)=\big(0, f_1^a, f_2^a, \ldots\big)$ is a (non-curved) $\infty$-morphism between the two complete shifted $\Li$-algebras $\alpha^a$ and $\beta^{f(a)}$ respectively. 
\end{enumerate}
\end{proposition}

\begin{proof}\leavevmode
\begin{enumerate}
\item The first assertion amounts to proving 
\begin{eqnarray*}
\left(\big(1-{f(a)}\big) \cc f \cc (1+{a})\right)\star \alpha^a=\beta^{f(a)} \cc \big(1-{f(a)}\big) \cc f \cc (1+{a})\ . 
\end{eqnarray*}
The left-hand term is equal to 
\begin{eqnarray*}
\left(\big(1-{f(a)}\big) \cc f \cc (1+{a})\right)\star \alpha^a&=&
 f(a)+ \left(f \cc (1+{a})\right)\star \big( \alpha\cc(1+a)\big)\\
 &=&  f(a)+ (f\star \alpha)\cc  (1+{a})\ .
\end{eqnarray*}
The right-hand term is equal to 
\begin{eqnarray*}
\beta^{f(a)} \cc \big(1-{f(a)}\big) \cc f \cc (1+{a})
&=& \big(1-f(a)\big)\cc\beta \cc \big(1+{f(a)}\big) \cc \big(1-{f(a)}\big) \cc f \\&& \cc (1+{a})\\
&=&f(a)+\beta  \cc f \cc (1+{a})\\
&=&f(a)+ (f\star \alpha) \cc (1+a)
\ . 
\end{eqnarray*}

\item The second assertion is the part of the above relation for the $\infty$-morphism in arity $0$. This latter one is equal to $f^a-f(a)=\big(0, f^a_1, f^a_2, \ldots\big)$ and so it has no constant term. The part of arity $0$ of the equation $\beta^{f(a)}\cc \big(f^a-f(a)\big)=\big(f^a-f(a)\big)\star \alpha^a$ is 
\[
\beta^{f(a)}_0=f_1^a\big(\alpha^a_0\big)=\sum_{k\ge 0} {\textstyle \frac{1}{k!}} f_{k+1}\big(a^k, \alpha^a_0 \big)
 \ .\] 

\item The last assertion is a direct corollary of the previous one: if $\alpha^a_0=0$, then so is $\beta^{f(a)}=0$\ . 
\end{enumerate}
\end{proof}

\begin{remark}
Usually this proposition is formulated in the $\Li$-case \cite{Getzler09, DolgushevRogers15, DolgushevRogers17} but the above short proof also shows that is holds in the curved case as well. 
\end{remark}

Let us continue with $\infty$-morphisms between two complete shifted curved $\Li$-algebras $(A,\F, \alpha)$ and $(B,\G,  \beta)$. Such a map is a collection $(b\coloneqq f_0, f_1, f_2, \ldots)$, where $b\in \G_1 B_0$. Let us denote it by $b+f$, where $f\coloneqq(0, f_1, f_2, \ldots)$.

\begin{lemma}\label{lem:EquivMorph}
Under this convention, with the constant term split apart, a data $b+f$ is an $\infty$-morphism from $\alpha$ to $\beta$ if and only if the data $f$ is an $\infty$-morphism from $\alpha$ to the twisted structure $\beta^b$:
\[ 
b+f\ : \ \alpha \rightsquigarrow \beta \quad \Longleftrightarrow \quad f\ :\  \alpha \rightsquigarrow \beta^b\ .
\]
\end{lemma}

\begin{proof}
The data $b+f$ in an $\infty$-morphism from $\alpha$ to $\beta$ if and only if it satisfies 
\[(b+f)\star \alpha= f\star \alpha= \beta\cc (b+f) \ .\]
The data $f$ in an $\infty$-morphism from $\alpha$ to $\beta^b$ if and only if it satisfies 
\[ 
f\star \alpha= \beta^b\cc f=\beta \cc (1+b)\cc f=\beta \cc (b+f)
\ ,\]
which concludes the proof.
\end{proof}

V.~Dolgushev and C.~Rogers introduced in \cite{DolgushevRogers15} a category whose objects are complete shifted $\Li$-algebras and whose morphisms from $(A,\F, \alpha)$ to $(B,\G, \beta)$  amount to the data of a Maurer--Cartan element $b\in \G_1 B_0$ and an $\infty$-morphism $f : \alpha \rightsquigarrow \beta^b$ without constant term. Let $g : \beta \rightsquigarrow \gamma^c$ be another such morphism; 
they define the composite of morphisms by the formula:
\[
\big(g^b-g(b)\big)\cc f \ : \ \alpha \rightsquigarrow \gamma^{c+g(b)}
\ . \]
This  category is denoted  by $\mathfrak{S}\mathsf{Lie}_\infty^{\text{MC}}$ in loc. cit.

\begin{proposition}
The Dolgushev--Rogers category $\mathfrak{S}\mathsf{Lie}_\infty^{\mathrm{MC}}$ is  the sub-category of the present category of complete shifted curved $\Li$-algebras with $\infty$-morphisms whose objects are 
complete shifted $\Li$-algebras and whose morphisms are $\infty$-morphisms such that the constant term is a Maurer--Cartan element in the target algebra.
\end{proposition}

\begin{proof}
Lemma~\ref{lem:EquivMorph} establishes the equivalence between the two notions of morphisms. 
Under the above convention, the composite of two $\infty$-morphisms $c+g$ and $b+f$ is equal to: 
 \[
 (c+g)\cc (b+f)=\big(c+g(b)\big) + \big(g^b- g(b)\big) \cc f \ ,
\]
which coincides with Dolgushev--Rogers definition.
\end{proof}

Besides giving a conceptual explanation for the Dolgushev--Rogers category, this result also allows us to prove easily the various properties of the composite of morphisms, like the associativity for instance. Notice that this category was used in a crucial way in \cite{DolgushevHoffnungRogers14} to provide an $\infty$-categorical enrichment of the category of homotopy algebras which encodes faithfully their higher homotopy theory. 

\section{Twistable operadic algebras}\label{subsec:TwistableAlg}
The purpose of this section is to describe on the level of the encoding operads which categories of algebras admit a meaningful twisting procedure. To be more precise, our aim is to characterise the (quadratic) operads $\calP$ for which any $\calP_\infty$-algebra can be twisted by any element satisfying a Maurer--Cartan type equation.This explains conceptually the particular form of the Maurer--Cartan equation. \\

Let $\calU\coloneqq(\k u, 0, \ldots)$ be the operad generated by a degree $0$ element $u$ of arity $0$; it encodes the data of a degree $0$ elements in  graded modules.  Let $(E,R)$ be an operadic quadratic-linear data, that is $R\subset E\oplus \calT(E)^{(2)}$, and let $\chi : E(2) \to \k$ be an $\Sy_2$-equivariant linear map of degree $0$, where $\k$ receives the trivial $\Sy_2$-action. We consider the space of relations $R_\chi\subset \calT(E\oplus \k u)$ generated by
\[\mu \circ_1 u-\chi(\mu)\id\quad \text{and}\quad \mu \circ_2 u-\chi(\mu)\id\ ,\]
with $\mu\in E(2)$, and all the other composites of elements of $E(n)$ with at least one $u$, for $n\neq 2$.

\begin{definition}[Unital extension]
The \emph{unital extension} of $\calP\coloneqq\calP(E,R)$ by $\chi$ is the following  operad 
\[u_\chi\calP\coloneqq
\calP(E\oplus \k u, R\oplus R_\chi)=
\frac{\calP(E,R) \vee \calU}{\left(
R_\chi
\right)}\ ,
\]
where $\vee$ stands for the coproduct of operads. 
\end{definition}

The category of $u_\chi\calP$-algebras is the category of $\calP$-algebras with a distinguished degree $0$ element which acts as a unit (with coefficients) for the generating operations of degree $0$ and arity $2$ and which vanishes once composed with any other generating operation. 
Notice that in the trivial case $\chi=0$, the unital extension amounts to  
$u_0 \calP\cong \k u \oplus \calP$
and that, in the general case,  
the underlying graded $\Sy$-module of $u_\chi \calP$ is a coset of $\k u \oplus \calP$. The ``maximal'' case is covered by the following definition. 

\begin{definition}[Extendable quadratic-linear operad]
A quadratic-linear presentation of an operad $\calP=\calP(E,R)$ is called \emph{extendable} when there exists a non-trivial map $\chi : E(2) \to \k$ such that the canonical map 
$
\calP \hookrightarrow u_\chi\calP
$ 
is a monomorphism. 
\end{definition}

This happens if and only if the underlying graded $\Sy$-module of $u_\chi \calP$ is isomorphic to $\k u \oplus \calP$.
We shall now discuss extendability of some classical operads.

\begin{proposition}\label{prop:Extendable}\leavevmode
\begin{enumerate}
\item The quadratic operads $\Com$, $\Gerst$, $\BV$, $\HyperCom$, $\PreLie$, and the  qua\-dra\-tic-linear operad 
 $\BV$ are extendable. 
\item The quadratic ns operads $\As$ is extendable. 
\item The quadratic operads $\Lie$ and $\Perm$ are not extendable. 
\item The quadratic ns operads $\Dias$ and $\Dend$ are not extendable. 
\end{enumerate}
\end{proposition}

\begin{proof}\leavevmode
\begin{enumerate}
\item In the case of commutative associative algebras, the classical definition of the unit $a\cdot 1=1\cdot a =a$, that is $\chi(\cdot)=1$, works to create such an extension. 

In the cases of operads of Gerstenhaber algebras, Batalin--Vilkovisky algebras, and hypercommutative algebras, we note that those are homology operads of topological operads (of little $2$-disks, framed little $2$-disks, and Deligne--Mumford compactifications of genus zero curves with marked points respectively) that admit units topologically (action of the unit corresponds to forgetting about one of the little disks, or about a marked points), hence the unitality remains on the algebraic level. The action of the unit on the generators is forced by degree reasons: all generators except for the binary generator of the commutative suboperad must be annihilated by the unit.

Let us consider the operad of pre-Lie algebras. We shall now show that the assignment $\chi(\star)=1$, that is $1\star a=a\star 1=a$, leads to a unital extension of the maximal possible size. It suffices to show that the insertion of $u$ into any element of the operadic ideal generated by the pre-Lie identity and applying the defining relations of the unital extension that we are considering produces an element of the same ideal. Let us note that insertion of $u$ in any slot of the pre-Lie identity gives zero: 
\begin{gather*}
(1\star b)\star c-1\star(b\star c)-(1\star c)\star b+1\star(c\star b)=0\ ,\\
(a\star1)\star c-a\star(1\star c)-(a\star c)\star 1+a\star(c\star 1)=0\ ,\\
(a\star b)\star 1-a\star(b\star 1)-(a\star 1)\star b+a\star(1\star b)=0\ .
\end{gather*}
Every element of the operadic ideal generated by the pre-Lie identity is a combination of elements which are obtained from the pre-Lie identity by pre- and post-compositions. Consider one such element $\nu$, and look at $\nu\circ_i u$. If the argument $i$ of $\nu$ is one of the arguments of the pre-Lie identity, then the above computation shows that $\nu\circ_i u=0$. Otherwise, the element $\alpha\circ_i u$ is still obtained from the pre-Lie identity by pre- and post-compositions, proving our claim.

\item The proof is the same as in the case  of the operad $\Com$. 

\item In the case of Lie algebras, there is no nontrivial equivariant map from the space of binary operations to the ground field regarded as the trivial module.

In the case of permutative algebras, we can substitute $x_1=u$ in the structural identity $x_1\cdot (x_2\cdot x_3)=x_1\cdot (x_3\cdot x_2)$, and note that it becomes $\chi(\cdot) x_2\cdot x_3=\chi(\cdot) x_3\cdot x_2$, so if the canonical map is a monomorphism, we must have $\chi(\cdot)=0$, and the extension is trivial.

\item Recall the defining relations of the ns operad of dendriform algebras:
\begin{gather*}
(x_1\prec x_2)\prec x_3 = x_1\prec (x_2\prec x_3+x_2\succ x_3)\ ,\label{eq:Dend1}\\
(x_1\succ x_2)\prec x_3 = x_1\succ (x_2\prec x_3)\ , \label{eq:Dend2}\\
(x_1\succ x_2+x_1\prec x_2)\succ x_3 = x_1\succ (x_2\succ x_3)\ .\label{eq:Dend3}
\end{gather*}
Suppose that we consider the unital extension corresponding to the linear function $\chi$. Substituting $x_1=u$ in the first dendriform axiom and then $x_3=u$ in the third one, we get
\begin{gather*}
\chi(\prec) x_2\prec x_3 = \chi(\prec) (x_2\prec x_3+x_2\succ x_3),\\
\chi(\succ) (x_1\succ x_2+x_1\prec x_2)=\chi(\succ) (x_1\succ x_2),
\end{gather*}
so if the canonical map is a monomorphism, we must have $\chi(\prec)=\chi(\succ)=0$, and the extension is trivial. 

Recall the defining relations of the ns operad of diassociative algebras: 
\begin{gather*}
(x_1\vdash x_2) \vdash x_3 = x_1\vdash (x_2\vdash x_3)\ ,\label{eq:Dias1}\\
(x_1\vdash x_2) \vdash x_3 = (x_1\dashv x_2)\vdash x_3\ ,\label{eq:Dias2}\\
(x_1\vdash x_2) \dashv x_3 = x_1\vdash (x_2\dashv x_3)\ ,\label{eq:Dias3}\\
x_1\dashv (x_2 \vdash x_3) = x_1\dashv (x_2\dashv x_3)\ ,\label{eq:Dias4}\\
(x_1\dashv x_2) \dashv x_3 = x_1\dashv (x_2\dashv x_3)\ .\label{eq:Dias5}
\end{gather*}
Suppose that we consider the unital extension corresponding to the linear function $\chi$. Substituting $x_3=u$ in the second axiom of diassociative algebras and then $x_1=u$ in the fourth one, we get
\begin{gather*}
\chi(\vdash) x_1\vdash x_2 = \chi(\vdash) x_1\dashv x_2,\\
\chi(\dashv) x_2\vdash x_3 = \chi(\dashv) x_2\dashv x_3,
\end{gather*}
so if the canonical map is a monomorphism, we must have $\chi(\vdash)=\chi(\dashv)=0$, and the extension is trivial.
\end{enumerate}
\end{proof}

\begin{remark}
A quadratic-linear operad $\calP=\calP(E,R)$ is extendable when it admits a (non-trivial) ``unitary extension'' in the terminology of \cite[Section~2.2]{Fresse17I}.
So the present notion can be seen as a way to produce concrete unitary extensions of operads in the quadratic case. As a consequence, an extendable  operad $\calP$ carries a richer structure of a $\Lambda$-operad,  crucial notion in the recognition of iterated loop spaces \cite{May72}, and  its underlying 
$\Sy$-module carries an $\mathrm{FI}$-module structure, where $\mathrm{FI}$ stands for the category of finite sets and injections. This  notion plays a seminal role in representation theory  \cite{CEF15}. 
\end{remark}

Let us now work over a field $\k$ of characteristic $0$ and suppose that $E(0)=0$ and that $E(n)$ is finite dimensional for any $n\ge 1$. Recall that the Koszul dual operad of a quadratic-linear opead $\calP=\calP(E,R)$ admits the following quadratic presentation $\calP^!=\calP\big(E^\vee, (qR)^\perp\big)$, with $E^\vee\coloneqq s^{-1}{\End}_{\k s^{-1}}\otimes_{\mathrm{H}} E^*$, see \cite[Section~$7.2.3$]{LodayVallette12}.

\begin{lemma}\label{lem:ExtConvAlg}
Under the above-mentioned assumptions, when the Koszul dual operad $\calP^!$ is extendable,  
the (complete)  convolution pre-Lie algebra associated to $\left(u_\chi \calP^!\right)^*$ is isomorphic to 
\[ \hom_\Sy\left(\left(u_\chi \calP^!\right)^*, \eend_A\right)\cong  
A \times  \hom_\Sy\left( \calP^{\ac} , \eend_A\right)\ ,
\]
where the pre-Lie product $\bigstar$ on the right-hand side is given by 
\[
(a,f){\bigstar} (b,g) = 
\left(f(\id)(b), f \star g + f\ast b
\right)\ ,
\]
with $\star$ the pre-Lie product on the convolution algebra $\a_{\calP^{\ac}, A}=\left(\hom_\Sy\left( \calP^{\ac} , \eend_A\right), \star\right)$ and with $f\ast b$ an element involving only $f$ and $b$.
\end{lemma}

\begin{proof}
When the Koszul dual operad $\calP^!$ is extendable, we use the underlying isomorphism 
$u_\chi \calP^! \cong\k u \oplus \calP^!$ to get 
\[ \hom_\Sy\left(\left(u_\chi \calP^!\right)^*, \eend_A\right)\cong  \hom_\Sy\left(u^* \oplus \calP^{\ac} , \eend_A\right)\cong 
A \times  \hom_\Sy\left( \calP^{\ac} , \eend_A\right). 
\]
The arity $0$ part the pre-Lie product $\bigstar$ comes from the  unique way to get the element $u$ in the operad $u_\chi \calP^!$ as a partial composition of two elements: $u=\id \circ_1 u$. Similarly for its other part, given any element $\mu\in \calP^!$, there are two ways to get it as a partial composition of two elements: either with two elements coming from $\calP^!$ or with one (bottom) coming from $\calP^!$ and one (top) which is $u$. Notice that in this latter case, we get a finite sum in each arity since $E(0)=0$ and $E(n)$ is finite dimensional for any $n\ge 1$.
\end{proof}

\begin{theorem}\label{thm:TwCatAlg}
Let $\alpha \in \MC\left(\a_{\calP^{\ac}, A}\right)$ be a complete $\calP_\infty$-algebra structure on $A$ and let $a\in \F_1A_0$. The Maurer--Cartan element $a.\alpha$ in the convolution algebra $\hom_\Sy\left(\left(u_\chi \calP^!\right)^*, \eend_A\right)$ is a $\calP_\infty$-algebra structure if and only if its arity $0$ part $(a.\alpha)(u^*)=0$ vanishes. 
\end{theorem}

\begin{proof}
This is a direct corollary of \cref{lem:ExtConvAlg} which implies that Maurer--Cartan elements in the convolution algebra $\a_{\calP^{\ac}, A}= \hom_\Sy\left( \calP^{\ac} , \eend_A\right)$ are in one-to-one correspondence with Maurer--Cartan elements in the extended convolution algebra $\hom_\Sy\left(\left(u_\chi \calP^!\right)^*, \eend_A\right)$ whose arity $0$ part vanishes. 
\end{proof}

This result prompts the following definition.

\begin{definition}[Twistable homotopy algebras]\label{def:TwHoAlg}
We say that a category of homotopy algebras, that is algebras over an operad $\calP_\infty=\Omega \calP^{\ac}$, where $\calP=\calP(E,R)$ is   an arity-wise finitely generated quadratic-linear operad,  is \emph{twistable} when the Koszul dual operad $\calP^!$ is extendable. The equation $(a.\alpha)(u^*)=0$ is naturally dubbed the \emph{Maurer--Cartan equation}. 
\end{definition}

In plain words, when a category of $\calP_\infty$-algebras is twistable, this means that the ``dual'' category of $\calP^!$-algebras admits a meaningful extension of \emph{unital} $\calP^!$-algebras and thus the category of $\calP_\infty$-algebras admits a meaningful extension of \emph{curved} $\calP_\infty$-algebras. 
In this case, the twisting procedure works as in the case of homotopy associative or homotopy Lie algebras: any $\calP_\infty$-algebra can be twisted by an element to produce a curved $\calP_\infty$-algebra, which turns out to be  a $\calP_\infty$-algebra if and only if its twisted curvature vanishes, that is satisfises the Maurer--Cartan equation. 

\begin{proposition}\label{prop:Twistable}\leavevmode
\begin{enumerate}
\item The categories of homotopy Lie algebras, homotopy Gerstenhaber algebras, homotopy $\BV^!$-algebras, homotopy gravity algebras, and homotopy permutative algebras are twistable.
\item The category of homotopy associative algebras is twistable. 
\item The categories of homotopy commutative algebras and homotopy pre-Lie algebras are not twistable. 
\item The categories of homotopy dendriform algebras and homotopy diassociative algebras are not twistable. 
\end{enumerate}
\end{proposition}

\begin{proof}
This is a direct corollary of  \cref{prop:Extendable} and \cref{thm:TwCatAlg}. 
\end{proof}

One can check that the maps $\chi$ introduced in \cref{prop:Extendable} for the unital extensions of the 
ns operad $\As$  and for the 
operad $\Com$
produce the exact same twisting formulas than the ones given in 
\cref{subsec:TwistGrp} for (curved)  $\Ai$-algebras and in 
\cref{subsec:TwLii} for (curved)  $\Li$-algebras. The other cases are new. Below we make the case of (curved)  homotopy permutative algebras explicit and we leave the other cases to the interested reader.

We already recalled the definition of permutative algebras in Proposition~\ref{prop:Extendable} and its proof. The operad $\Perm$ is known to be Koszul. Its Koszul dual operad is the operad $\PreLie$ encoding pre-Lie algebras \cite[Section~13.4.3]{LodayVallette12}. This latter operad admits for basis the set of rooted trees $\mathrm{RT}$:  
$$\vcenter
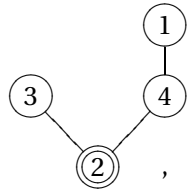
{
\xymatrix@M=5pt@R=10pt@C=10pt{
 & & *+[o][F-]{1}\ar@{-}[d]\\
*+[o][F-]{3}\ar@{-}[dr] &&  *+[o][F-]{4}\ar@{-}[dl] \\
& *+[o][F=]{2} & ,
}}$$
where the lowest vertex is the root. The partial composition products $\tau\circ_i \upsilon$ is given by the insertion of the tree $\upsilon$ at the $i^{\mathrm{th}}$ vertex of the tree $\tau$. The sub-trees attached above the $i^{\mathrm{th}}$ vertex of the tree $\tau$ are then  grafted in all possible onto the vertices of the tree $\upsilon$: 
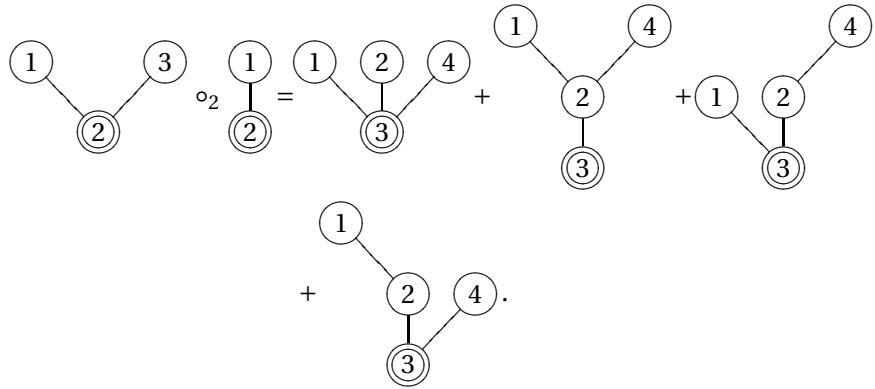
\begin{align*}
\vcenter{
\xymatrix@M=5pt@R=10pt@C=10pt{
*+[o][F-]{1}\ar@{-}[dr] &&  *+[o][F-]{3}\ar@{-}[dl] \\
& *+[o][F=]{2} & 
}}
\ \circ_2 \ 
\vcenter{
\xymatrix@M=5pt@R=10pt@C=10pt{
*+[o][F-]{1}\ar@{-}[d]  \\
*+[o][F=]{2} 
}}=&
\vcenter{
\xymatrix@M=5pt@R=10pt@C=10pt{
*+[o][F-]{1}\ar@{-}[dr] &*+[o][F-]{2}\ar@{-}[d] & *+[o][F-]{4}\ar@{-}[dl]\\
& *+[o][F=]{3} & 
}}+
\vcenter{
\xymatrix@M=5pt@R=10pt@C=10pt{
*+[o][F-]{1}\ar@{-}[dr] & & *+[o][F-]{4}\ar@{-}[dl]\\
 &  *+[o][F-]{2}\ar@{-}[d]& \\
& *+[o][F=]{3} & 
}}+
\vcenter{
\xymatrix@M=5pt@R=10pt@C=10pt{
 & & *+[o][F-]{4}\ar@{-}[dl]\\
*+[o][F-]{1}\ar@{-}[dr] &  *+[o][F-]{2}\ar@{-}[d]& \\
& *+[o][F=]{3} & 
}}\\&+
\vcenter{
\xymatrix@M=5pt@R=10pt@C=10pt{
*+[o][F-]{1}\ar@{-}[dr] & & \\
 &  *+[o][F-]{2}\ar@{-}[d]& *+[o][F-]{4}\ar@{-}[dl]\\
& *+[o][F=]{3} & 
}}\ .
\end{align*}

 As explained in \cref{prop:MCOmegaC}, the data of a shifted $\Perm_\infty$-algebra on a (complete) dg module $A$ amounts to a Maurer--Cartan element in the convolution algebra $\a_{\PreLie^*, A}$. Given a rooted tree $\tau$ and a sub-tree $\upsilon\subset \tau$ (forgetting the labels), we denote by $(\tau/\upsilon, i, \sigma)$ respectively the rooted tree $\tau/\upsilon$ obtained by contracting $\upsilon$ in $\tau$ and by relabelling the vertices,  the label $i$ of the corresponding new vertex, and the overall permutation $\sigma\in\Sy_{|\tau|}$ which produces the labelling of the tree $\tau$ after  the composite $\left(\tau/\upsilon \circ_i \upsilon\right)^\sigma=\tau$. Such a decomposition is not unique, we choose the one for which the associated planar tree 
 \[\vcenter
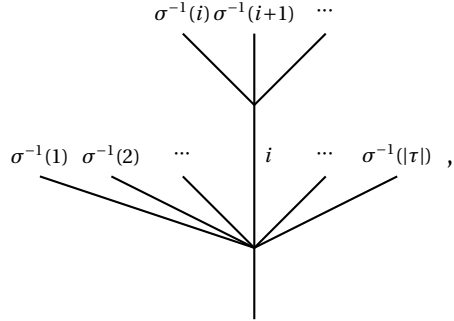
{\hbox{
\begin{tikzpicture}[yscale=0.95,xscale=0.95]

\draw[thick] (0,-1)--(0,0) -- (0,1) -- (0, 3);
\draw[thick] (1,1) -- (0,0) -- (-1,1);
\draw[thick] (-2,1) -- (0,0) ;
\draw[thick] (2,1) -- (0,0) ;
\draw[thick] (-3,1) -- (0,0) ;
\draw[thick] (1,3) -- (0,2) -- (-1,3);

\node at (-3,1.3) {$\scriptstyle\sigma^{-1}(1)$};
\node at (-2,1.3) {$\scriptstyle\sigma^{-1}(2)$};
\node at (-1,1.3) {$\scriptstyle\cdots$};
\node at (-1,3.3) {$\scriptstyle\sigma^{-1}(i)$};
\node at (0,3.3) {$\scriptstyle\sigma^{-1}(i+1)$};
\node at (1,3.3) {$\scriptstyle\cdots$};
\node at (1,1.3) {$\scriptstyle\cdots$};
\node at (2,1.3) {$\scriptstyle\sigma^{-1}(|\tau|)$};
\node at (0.2,1.3) {$\scriptstyle i$};
\end{tikzpicture}}}\ ,\]
where the bottom corolla has arity $|\tau/\upsilon|$ and where the top corolla has arity $|\upsilon|$,
 is a shuffle tree \cite[Section~8.2.2]{LodayVallette12}. 
 
\begin{align*}&
\text{For}\quad
\tau\coloneqq\vcenter{
\xymatrix@M=5pt@R=10pt@C=10pt{
*+[o][F-]{1}\ar@{-}[dr] & & *+[o][F-]{4}\ar@{-}[dl]\\
 &  *+[o][F-]{2}\ar@{-}[d]& \\
& *+[o][F=]{3} & 
}}\quad \text{and} \quad
\upsilon\coloneqq
\vcenter{
\xymatrix@M=5pt@R=10pt@C=10pt{
*+[o][F-]{1}\ar@{-}[dr] & & *+[o][F-]{3}\ar@{-}[dl]\\
 &  *+[o][F=]{2}
}}\ , \\&  \text{we get} \quad
\tau/\upsilon=\vcenter{
\xymatrix@M=5pt@R=10pt@C=10pt{
  *+[o][F-]{1}\ar@{-}[d] \\
*+[o][F=]{2}  
}}\ ,\quad i=1\ , \quad \text{and} \quad \sigma=(34)\ .
\end{align*}
 
 \begin{lemma}
 A shifted homotopy permutative algebra is a graded module $A$ equipped with a collection of degree $-1$ operations 
\[ \left\{m_\tau : A^{\otimes |\tau|}\to A\right\}_{\tau \in \mathrm{RT}}\ ,
\]
where $|\tau|$ stands for the number of vertices of a rooted tree. These generating operations are required to satisfy the relations 
\[
\sum_{\upsilon \subset \tau} \left(m_{\tau/\upsilon} \circ_i m_\upsilon \right)^\sigma=0\ ,
\]
for any tree $\tau\in \mathrm{RT}$.
 \end{lemma}
 
 \begin{proof}
 This is a direct corollary of the definition of the convolution algebra $\a_{\PreLie^*, A}$ and the cooperad structure on 
 $\PreLie^*$ obtained by linear dualisation of the operad structure on $\PreLie$ described above. The choice of decompositions $\left(\tau/\upsilon \circ_i \upsilon\right)^\sigma=\tau$ along shuffle trees does not alter the result: any choice of decomposition would produce the same result in the end, by the properties of the endomorphism operad $\End_A$. In other words, the composite $\left(m_{\tau/\upsilon} \circ_i m_\upsilon \right)^\sigma$ does not depend of the choice of decomposition.
 \end{proof}

Here are the first relations, where we use the  notation $d\coloneqq m_{\vcenter{
\xymatrix@M=4pt@R=10pt@C=10pt{
  *+[o][F=]{\scriptscriptstyle 1}
}}}$\ .

\begin{description}
\item[\sc $\tau=
\vcenter{
\xymatrix@M=5pt@R=10pt@C=10pt{
  *+[o][F=]{1}
}}$\,] 
$d^2=0\ .$

\smallskip 

\item[\sc $\tau=
\vcenter{
\xymatrix@M=5pt@R=10pt@C=10pt{
  *+[o][F-]{2}\ar@{-}[d] \\
*+[o][F=]{1} 
}}$\,]
$\partial m_{\vcenter{
\xymatrix@M=4pt@R=5pt@C=5pt{
  *+[o][F-]{\scriptscriptstyle 2}\ar@{-}[d] \\
*+[o][F=]{\scriptscriptstyle 1} 
}}}\coloneqq
d m_{\vcenter{
\xymatrix@M=4pt@R=5pt@C=5pt{
  *+[o][F-]{\scriptscriptstyle 2}\ar@{-}[d] \\
*+[o][F=]{\scriptscriptstyle 1} 
}}}(-,-)+m_{\vcenter{
\xymatrix@M=4pt@R=5pt@C=5pt{
  *+[o][F-]{\scriptscriptstyle 2}\ar@{-}[d] \\
*+[o][F=]{\scriptscriptstyle 1} 
}}}(d(-), -)+m_{\vcenter{
\xymatrix@M=4pt@R=5pt@C=5pt{
  *+[o][F-]{\scriptscriptstyle 2}\ar@{-}[d] \\
*+[o][F=]{\scriptscriptstyle 1} 
}}}(-, d(-))=0$\ . 

\smallskip 

\item[\sc $\tau=
\vcenter{
\xymatrix@M=5pt@R=10pt@C=10pt{
  *+[o][F-]{3}\ar@{-}[d] \\
  *+[o][F-]{2}\ar@{-}[d] \\
*+[o][F=]{1} 
}}$\,]
$\partial m_{\vcenter{
\xymatrix@M=4pt@R=5pt@C=5pt{
 *+[o][F-]{\scriptscriptstyle 3}\ar@{-}[d] \\
  *+[o][F-]{\scriptscriptstyle 2}\ar@{-}[d] \\
*+[o][F=]{\scriptscriptstyle 1}
}}}=-
m_{\vcenter{
\xymatrix@M=4pt@R=5pt@C=5pt{
  *+[o][F-]{\scriptscriptstyle 2}\ar@{-}[d] \\
*+[o][F=]{\scriptscriptstyle 1} 
}}}\circ_1 
m_{\vcenter{
\xymatrix@M=4pt@R=5pt@C=5pt{
  *+[o][F-]{\scriptscriptstyle 2}\ar@{-}[d] \\
*+[o][F=]{\scriptscriptstyle 1} 
}}}
-
m_{\vcenter{
\xymatrix@M=4pt@R=5pt@C=5pt{
  *+[o][F-]{\scriptscriptstyle 2}\ar@{-}[d] \\
*+[o][F=]{\scriptscriptstyle 1} 
}}}\circ_2
m_{\vcenter{
\xymatrix@M=4pt@R=5pt@C=5pt{
  *+[o][F-]{\scriptscriptstyle 2}\ar@{-}[d] \\
*+[o][F=]{\scriptscriptstyle 1} 
}}}$\ . 

\smallskip

\item[\sc $\tau=
\vcenter{
\xymatrix@M=5pt@R=10pt@C=10pt{
*+[o][F-]{2}\ar@{-}[dr] &&  *+[o][F-]{3}\ar@{-}[dl] \\
& *+[o][F=]{1} & 
}}$\,]
$\partial m_{\vcenter{
\xymatrix@M=4pt@R=5pt@C=5pt{
*+[o][F-]{\scriptscriptstyle 2}\ar@{-}[dr] &&  *+[o][F-]{\scriptscriptstyle 3}\ar@{-}[dl] \\
& *+[o][F=]{\scriptscriptstyle 1} & 
}}}=-
m_{\vcenter{
\xymatrix@M=4pt@R=5pt@C=5pt{
  *+[o][F-]{\scriptscriptstyle 2}\ar@{-}[d] \\
*+[o][F=]{\scriptscriptstyle 1} 
}}}\circ_1 
m_{\vcenter{
\xymatrix@M=4pt@R=5pt@C=5pt{
  *+[o][F-]{\scriptscriptstyle 2}\ar@{-}[d] \\
*+[o][F=]{\scriptscriptstyle 1} 
}}}
-
\big(m_{\vcenter{
\xymatrix@M=4pt@R=5pt@C=5pt{
  *+[o][F-]{\scriptscriptstyle 2}\ar@{-}[d] \\
*+[o][F=]{\scriptscriptstyle 1} 
}}}\circ_1 
m_{\vcenter{
\xymatrix@M=4pt@R=5pt@C=5pt{
  *+[o][F-]{\scriptscriptstyle 2}\ar@{-}[d] \\
*+[o][F=]{\scriptscriptstyle 1} 
}}}\big)^{(23)}$\ . 
\end{description}

\smallskip

This shows that $m_{\vcenter{
\xymatrix@M=4pt@R=5pt@C=5pt{
 *+[o][F-]{\scriptscriptstyle 3}\ar@{-}[d] \\
  *+[o][F-]{\scriptscriptstyle 2}\ar@{-}[d] \\
*+[o][F=]{\scriptscriptstyle 1}
}}}$\ , respectively $m_{\vcenter{
\xymatrix@M=4pt@R=5pt@C=5pt{
*+[o][F-]{\scriptscriptstyle 2}\ar@{-}[dr] &&  *+[o][F-]{\scriptscriptstyle 3}\ar@{-}[dl] \\
& *+[o][F=]{\scriptscriptstyle 1} & 
}}}$\ ,  is a homotopy for the first relation (anti-associativity), respectively for the second relation (partial skew-symmetry),  defining a shifted permutative algebra.\\

We consider the unital extension of operad $\PreLie$ introduced in the proof of \cref{prop:Extendable} that 
we simply denote here by $\uPreLie$. This operad encodes pre-Lie algebras $(A, \star)$ equipped with a degree $0$ element $1$ satisfying 
\[1\star x=x=x\star 1\ ,\]
for any $x\in A$. 

\begin{lemma}\label{lem:uCompPreLie}
In the operad $\uPreLie$, the composite of the arity $0$ element $u$ at a vertex of a rooted tree $\tau$ 
depends on the  position of this latter one as follows:
\begin{description}
\item[\sc At a leaf] if the number of inputs of the vertex supporting the leaf is equal to $n$, then the resulting rooted tree is obtained by removing the leaf, relabelling the other vertices, and multiplying by $2-n$,

\item[\sc At the root or at an internal vertex] when the vertex has just one input, then it is deleted and the remaining vertices relabelled accordingly, otherwise when the vertex has at least two inputs, the upshot is equal to $0$. 
\end{description}
\end{lemma}

\begin{proof}
To demonstrate that, the easiest way is to identify the operad of pre-Lie algebras with the symmetric brace operad \cite{GuinOudom08}. Recall that the symmetric brace operad has generators $\{x_0;x_1,\ldots,x_n\}$, $n\ge 0$, which are symmetric in arguments $x_1,\ldots,x_n$ and satisfy the relations 
\begin{multline}
\{\{x_0;x_1,\ldots,x_n\}; y_1,\ldots,y_m\}=\\ \sum \{x_0; \{x_1;y_{i_{1,1}},\ldots,y_{1,i_{t_1}}\},\ldots,\{x_n;y_{n,i_{n,1}},\ldots,y_{n,i_{t_n}}\},y_{n+1,1},\ldots,y_{n+1,i_{t_{n+1}}}\} \ .
\end{multline}
It is well known that the symmetric brace operad is isomorphic to the operad of pre-Lie algebras, with the isomorphism given by the formulas we already used earlier in the paper, i.e. via identifying the operation $x_1\star x_2$ with $\{x_1;x_2\}$, and defining, inductively,
 \[
\{x_0;x_1,\ldots,x_n\}\coloneqq\{\{x_0;x_1,\ldots,x_{n-1}\};x_n\}-\sum_{i} \{x_0;x_1,\ldots,\{x_i;x_n\},x_{i+1},\ldots,x_{n-1} \}
 \]
With this definition, it is easy to see that $\{1;x_1,\ldots,x_n\}=0$, for $n\ge 2$, and that
 \[
\{x_0;x_1,\ldots,x_{i-1},1,x_{i+1},\ldots,x_n\}=(2-n)\{x_0;x_1,\ldots,x_{i-1},x_{i+1},\ldots,x_n\}. 
 \] 
The brace representation of the operad of pre-Lie algebras corresponds to the ``naive'' way of building a rooted tree from corollas, so the claim follows. 
\end{proof}

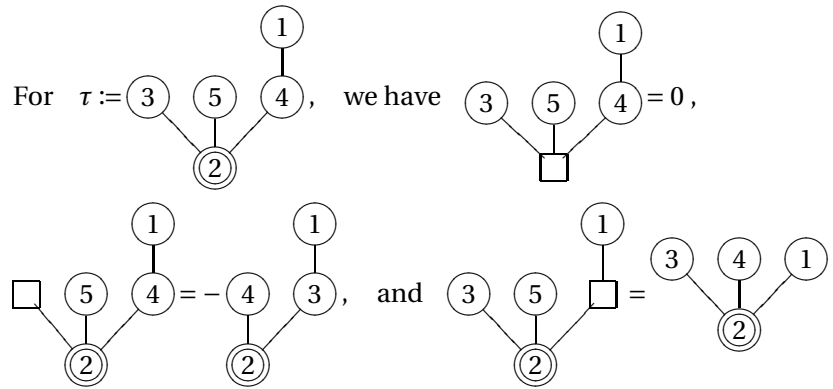
\begin{eqnarray*}
&&\text{For}\quad \tau\coloneqq
\vcenter{\xymatrix@M=5pt@R=10pt@C=10pt{
 & & *+[o][F-]{1}\ar@{-}[d]\\
*+[o][F-]{3}\ar@{-}[dr] &*+[o][F-]{5}\ar@{-}[d]&  *+[o][F-]{4}\ar@{-}[dl] \\
& *+[o][F=]{2} & }}\ ,\quad \text{we have} \quad 
\vcenter{\xymatrix@M=5pt@R=10pt@C=10pt{
 & & *+[o][F-]{1}\ar@{-}[d]\\
*+[o][F-]{3}\ar@{-}[dr] &*+[o][F-]{5}\ar@{-}[d]&  *+[o][F-]{4}\ar@{-}[dl] \\
& *+[o][F**]{}  & 
}}=0\ , \\
&&  
\vcenter{\xymatrix@M=5pt@R=10pt@C=10pt{
 & & *+[o][F-]{1}\ar@{-}[d]\\
*+[o][F**]{}\ar@{-}[dr] &*+[o][F-]{5}\ar@{-}[d]&  *+[o][F-]{4}\ar@{-}[dl] \\
& *+[o][F=]{2} & }}=
-\ \vcenter{\xymatrix@M=5pt@R=10pt@C=10pt{
 & *+[o][F-]{1}\ar@{-}[d]\\
*+[o][F-]{4}\ar@{-}[d]&  *+[o][F-]{3}\ar@{-}[dl] \\
 *+[o][F=]{2} & }}\ , \quad \text{and} \quad 
 \vcenter{\xymatrix@M=5pt@R=10pt@C=10pt{
 & & *+[o][F-]{1}\ar@{-}[d]\\
*+[o][F-]{3}\ar@{-}[dr] &*+[o][F-]{5}\ar@{-}[d]&  *+[o][F**]{}\ar@{-}[dl] \\
& *+[o][F=]{2} & }}=
 \vcenter{\xymatrix@M=5pt@R=10pt@C=10pt{
*+[o][F-]{3}\ar@{-}[dr] &*+[o][F-]{4}\ar@{-}[d]&  *+[o][F-]{1}\ar@{-}[dl] \\
& *+[o][F=]{2} &}}\ .
\end{eqnarray*}
The black vertex represents where the element $u$ is plugged. \\

Given a rooted tree $\tau$, we consider its \emph{unital expansions} which are rooted trees $\widetilde{\tau}$ with two types of vertices: the ``white'' ones which are bijectively labeled by $\{1, \ldots, |\tau|\}$ and the ``black'' ones which receive no label. Every such black and white rooted trees $\widetilde{\tau}$ are moreover required to give $\tau$ under the rule given at \cref{lem:uCompPreLie}.
We denote by $c_{\widetilde{\tau}}$ the coefficient of $\tau$ in the operad $\uPreLie$ obtained by composing $\widetilde{\tau}$ with the black vertices replaced by elements $u$. 

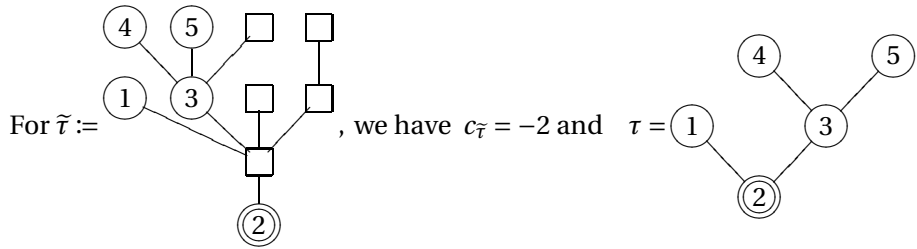
\[
\text{For}\  \widetilde{\tau}\coloneqq
\vcenter{\xymatrix@M=5pt@R=10pt@C=10pt{
 *+[o][F-]{4}\ar@{-}[dr]&*+[o][F-]{5}\ar@{-}[d]& *+[o][F**]{}\ar@{-}[dl]& *+[o][F**]{}\ar@{-}[d]\\
*+[o][F-]{1}\ar@{-}[drr] &*+[o][F-]{3}\ar@{-}[dr] &*+[o][F**]{}\ar@{-}[d]&  *+[o][F**]{}\ar@{-}[dl] \\
&& *+[o][F**]{} \ar@{-}[d]&\\
&& *+[o][F=]{2} & }}\ ,\ \text{we have}\ \ 
c_{\widetilde{\tau}}=-2 \ \text{and} \quad 
\tau=\vcenter{\xymatrix@M=5pt@R=10pt@C=10pt{
 &*+[o][F-]{4}\ar@{-}[dr]&&*+[o][F-]{5}\ar@{-}[dl] \\
*+[o][F-]{1}\ar@{-}[dr] &&*+[o][F-]{3}\ar@{-}[dl]&   \\
& *+[o][F=]{2}& }}\ .
\]

\begin{proposition}
Let $\left(A, \F, \{m_\tau\}_{\tau \in \mathrm{RT}}\right)$ be a complete shifted $\Perm_\infty$-algebra and let $a\in \F_1A_0$.
The twisted operations 
\[ m_\tau^a\coloneqq \sum_{\widetilde{\tau}} {\textstyle \frac{1}{c_{\widetilde{\tau}}}} m_{\widetilde{\tau}}(a, \ldots, a, - ,  \cdots, -)\ ,
\]
where the sum runs over unital expansions of the rooted tree $\tau$ and where the elements $a$ are inserted at the black vertices of $\widetilde{\tau}$, 
define a complete shifted $\Perm_\infty$-algebra if and only if the element $a$ satisfies the following Maurer--Cartan equation 
\[
\sum_{n\ge 1} m_{{\vcenter{
\xymatrix@M=4pt@R=5pt@C=5pt{
  *+[o][F-]{\scriptscriptstyle n}\ar@{..}[d] \\
*+[o][F-]{\scriptscriptstyle 2} \ar@{-}[d]\\  
*+[o][F=]{\scriptscriptstyle 1} 
}}}}(a, \ldots, a)=0
\ .\]
\end{proposition}

\begin{proof}
We apply here \cref{thm:TwCatAlg}: let us denote by $\alpha$ the Maurer--Cartan element in the convolution 
algebra $\a_{\PreLie^*, A}$ corresponding to the complete shifted $\Perm_\infty$-algebra structure, that is $\alpha(\tau^*)\coloneqq m_\tau$. 

First, the  the Maurer--Cartan equation  
 is equal to 
\[(a.\alpha)(u^*)=\sum_{\tau \in \mathrm{RT}} m_\tau (a, \ldots, a)=0\ .\]
But \cref{lem:uCompPreLie} shows that, the composite of  elements $u$ at all the vertices of a rooted tree $\tau$ in the operad $\uPreLie$ is equal to $u$ for ladders and vanishes otherwise. 

Then, the twisted operation $m_\tau^a$ is equal to $\big(\alpha\circledcirc (1+a)\big)(\tau^*)$ in the convolution algebra 
$\a_{\uPreLie^*, A}$, which is thus equal to $m_\tau^a\coloneqq \sum_{\widetilde{\tau}} {\textstyle \frac{1}{c_{\widetilde{\tau}}}} m_{\widetilde{\tau}}(a, \ldots, a, - ,  \cdots, -)$ again by \cref{lem:uCompPreLie}. 
\end{proof}

\begin{remark}
Under the notation $m_n\coloneqq m_{{\vcenter{
\xymatrix@M=4pt@R=5pt@C=5pt{
  *+[o][F-]{\scriptscriptstyle n}\ar@{..}[d] \\
*+[o][F-]{\scriptscriptstyle 2} \ar@{-}[d]\\  
*+[o][F=]{\scriptscriptstyle 1} 
}}}}$\ , one can see that $\big(A, \{m_n\}_{n\ge 1}\big)$ forms a shifted  $\Ai$-algebra. The twisting procedure for shifted $\Perm_\infty$-algebras produces the exact same twisting procedure seen in \cref{subsec:TwistGrp} for this underlying $\Ai$-algebra, ie 
$m_n^a= m^a_{{\vcenter{
\xymatrix@M=4pt@R=5pt@C=5pt{
  *+[o][F-]{\scriptscriptstyle n}\ar@{..}[d] \\
*+[o][F-]{\scriptscriptstyle 2} \ar@{-}[d]\\  
*+[o][F=]{\scriptscriptstyle 1} 
}}}}$ and the same Maurer--Cartan equation: 
\[\sum_{n\ge 1} m_n(a,\allowbreak \ldots, a)=0\ .\] This is not a surprise, since the forgetful functor from permutative algebras to associative algebras is actually produced by pulling back along the epimorphism $\Ass \twoheadrightarrow \Perm$ of operads obtained by sending the (canonical) generator of $\Ass$ to the (canonical) generator of $\Perm$.
Since this latter one comes from a morphism of operadic quadratic data,  it induces a monomorphism of cooperads between the Koszul dual cooperads 
$\Ass^{\ac} \hookrightarrow \Perm^{\ac}$ and thus a monomorphism between the Koszul resolution 
$\Ass_\infty \hookrightarrow \Perm_\infty$, which sends precisely $\mu_n$ to the rooted tree $\vcenter{
\xymatrix@M=4pt@R=5pt@C=5pt{
  *+[o][F-]{\scriptscriptstyle n}\ar@{..}[d] \\
*+[o][F-]{\scriptscriptstyle 2} \ar@{-}[d]\\  
*+[o][F=]{\scriptscriptstyle 1} 
}}.$ The twisting procedure for shifted $\Perm_\infty$-algebras extends the twisting procedure for shifted $\Ai$-algebras since the morphism of cooperads 
$\Ass^{\ac} \hookrightarrow \Perm^{\ac}$ extends to the morphism of cooperads 
$\mathrm{uAss}^* \hookrightarrow \uPreLie^*$.
\end{remark}

\newpage
\chapter{Twisting  nonsymmetric operads}\label{sec:TwNsOp}

In this section, we lift the twisting procedure to the level of complete differential graded non-symmetric operads and we apply it to the example of  multiplicative nonsymmetric operads. 
 This construction is  the nonsymmetric analogue of the theory introduced by T. Willwacher in \cite{Willwacher15} and developed further by V. Dolgushev, C. Rogers, and T. Willwacher \cite{DolgushevRogers12, DolgushevWillwacher15} but in a different way. Everything written here works \textit{mutatis mutandis} for complete (symmetric) dg operads, but we chose this level of presentation in order to make more accessible T. Willwacher's seminal theory. 
 
We give here an alternative definition based on the categorical notion of \emph{coproduct of operads}. While, in the above-mentioned references, the notion of operadic twisting is introduced via its (complicated) underlying collection, we introduce it here by its universal property: it is the operad obtained by the coproduct of the original operad with an arity zero operation. We then show that all the properties of the operadic twisting follow in a straightforward way  from the universal property of operadic coproduct. We do not claim any originality in the presented results: they all come from the above cited papers. Even the idea of this alternative presentation is not new since it was noticed in loc. cit. that  the operadic twisting amounts to a ``completion of an operad by a Maurer--Cartan element''; the operadic coproduct construction creates such a  ``completion''. 

With a pedagogical purpose, we have introduced a "new" complete dg ns operad $\MC \calP$, which models algebras over a multiplicative operad $\calP$ equipped with a given Maurer--Cartan element. In the subsection dealing with the action of the deformation complex on twisted operads, we have emphasised, for the first time, the various dg pre-Lie structures present in this topic.
 
\section{Twisting ns operads}
One can try to twist  any dg pre-Lie algebra $(A, d, \star)$ by Maurer--Cartan elements $d(\mu)+\mu\star\mu=0$ using the usual formulas $(A, d^\mu, \star)$, where the twisted differential is given by 
$$d^\mu(\nu)\coloneqq d(\nu)+\ad_\mu(\nu)\coloneqq d(\nu)+\mu\star \nu - (-1)^{|\nu|} \nu \star \mu\ .$$

\begin{proposition}\label{prop:Twpre-Lie}
The twisted data $(A, d^\mu, \star)$ forms a dg pre-Lie algebra if and only if the Maurer--Cartan element $\mu$ is a \emph{left nucleus element}, that is 
\begin{eqnarray}\label{Eq:MCspecial}
\mu\star(\nu\star \omega)=(\mu\star \nu)\star \omega\ , \quad \text{for any}\ \  \nu, \omega\in A\ . 
\end{eqnarray}
\end{proposition}

\begin{proof}
We only have to check that the twisted operator $d^\mu$ is a derivation which squares to zero. 
This latter point is always true since $\mu$ is a Maurer--Cartan element: 
\begin{eqnarray*}
\left(\dd+\ad_\mu\right)^2(\nu)=
\left(\dd(\mu)+\mu\star\mu\right)\star \nu-
 \nu \star
\left(\dd(\mu)+\mu\star\mu\right)=0\ .
\end{eqnarray*}
To prove the former point, one performs the following computation 
\begin{eqnarray*}
d^\mu(\nu\star \omega)-d^\mu(\nu)\star \omega-(-1)^{|\nu|}\nu\star d^\mu(\omega)&=&
\mu\star(\nu\star \omega)-(\mu\star \nu)\star \omega\ ,
\end{eqnarray*}
which concludes the proof. 
\end{proof}

\begin{remark}
The conceptual explanation for this failure of the general twisting procedure to work in general for pre-Lie algebras has been given in \cref{prop:Twistable}. It comes from  the fact that its Koszul dual operad $\PreLie^!\cong \Perm$ for permutative algebras does not admit an extension by a unit, see \cref{prop:Extendable}. 
\end{remark}

Recall that for any ns operad $\calP$, one defines a Lie bracket by anti-sym\-me\-tri\-zing the pre-Lie product 
$$\mu\star \nu \coloneqq \sum_{i=1}^n \mu \circ_i \nu\  , $$
where $\mu$ lives in $\calP(n)$.

\begin{example}
The Maurer--Cartan elements of the endomorphism operad $\End_{(A,d)}$ are the degree $-1$ maps $ m : A \to A$ satisfying the equation $dm+md+m^2=0$, i.e. perturbations of the differential. 
\end{example}

\begin{definition}[Operadic Maurer--Cartan element]
Let $\left(\calP, \dd\right)$ be a dg ns operad. An \emph{operadic Maurer--Cartan} in $\calP$ is a 
 degree $-1$ and arity $1$ element 
$\mu\in \calP(1)_{-1}$ satisfying 
$$ \dd(\mu)+\mu\star \mu=\dd(\mu)+\mu\circ_1 \mu=0\ .$$
\end{definition}
The purpose of the arity $1$ constraint in the aforementioned definition is to ensure that operadic Maurer--Cartan elements satisfy Equation~(\ref{Eq:MCspecial}). Therefore, one can twist the pre-Lie algebra associated to an ns operad by such elements. This result actually lifts to the level of the dg ns operad itself. 

\begin{proposition}[Twisted operad]
Let $\left(\calP, \dd, \left\{ \circ_i \right\}\right)$ be a dg ns operad and let $\mu$ be one of its Maurer--Cartan elements. 
The degree $-1$ operator  
$$\dd^\mu(\nu)\coloneqq\dd(\nu)+\ad_\mu(\nu)\coloneqq\dd(\nu)+\mu\star \nu - (-1)^{|\nu|} \nu \star \mu$$
is a  square-zero derivation.
Therefore, the data $\calP^\mu\coloneqq\left(\calP,   \dd^\mu,  \left\{ \circ_i \right\}\right)$ defines a dg ns operad, called 
the \emph{ns operad twisted  by the Maurer-Cartan element $\mu$}. 
\end{proposition}

\begin{proof}
The computations are similar to that of Proposition~\ref{prop:Twpre-Lie}. The fact that the twisted operator 
$\ad_\mu$ is a derivation with respect to  the partial composition products $\circ_i$ is equivalent to the equations 
\begin{eqnarray*}
\mu\star(\nu\circ_i \omega)=(\mu\star \nu)\circ_i \omega\ , 
\end{eqnarray*}
which holds true since the Maurer--Cartan element $\mu$ is supported in arity $1$. 
\end{proof}

\begin{example}
The endomorphism operad $\End_{(A,d)}$ of a chain complex $(A,d)$ twisted by a Maurer--Cartan element $m$ is actually the endomorphism operad 
$$\End_{(A,d)}^m=\End_{(A,d+m)} $$
of the twisted chain complex $(A,d+m)$. 
\end{example}

\begin{example}
One can also twist the ns operad $\mathrm{ncBV}$ of non-commutative Batalin--Vilkovisky algebras \cite{DotsenkoShadrinVallette15} by its square-zero element $\Delta$, under cohomological degree convention. 
\end{example}

The operadic twisting operation satifises the following usual functorial property of the twisting procedure. 

\begin{proposition}\label{prop:MorphMCOperadic}
Let $\rho \ :\ \calP \to \calQ$ be a morphism of dg ns operads. The image $\rho(\mu)$ of a Maurer--Cartan $\mu$ of the dg ns operad $\calP$ is a Maurer--Cartan element in the dg ns operad $\calQ$. The morphism $\rho$ of ns operads commutes with the twisted differential, that is yields a morphism of dg ns operads 
$$ 
\widetilde\rho\  :\ \calP^\mu \to \calQ^{\rho(\mu)}
\ . $$
\end{proposition}

\begin{proof}
This is proved by straightforward  computations. 
\end{proof}

This proposition applied to the endomorphism operad $\calQ=\End_A$ will play a  key role in the sequel. 

\begin{proposition}\label{prop:-MC}
Let $\mu$ be a Maurer--Cartan element of a dg ns operad $\calP$. 
\begin{enumerate}
\item  An element $\alpha$ is a Maurer--Cartan element of the twisted dg ns operad $\calP^\mu$ if and only if 
$\alpha-\mu$ is  a Maurer--Cartan element of the original  dg ns operad $\calP$. 

\item 
The element $-\mu$ is a Maurer--Cartan element in the twisted dg ns operad $\calP^\mu$ and twisting this latter one again by this Maurer--Cartan element produces the original operad:
$$ 
\left(\calP^\mu\right)^{-\mu}=\calP\ .
$$
\end{enumerate}
\end{proposition}

\begin{proof}
The first point is proved by straightforward  computations. The second one is a special case for $\alpha=0$.
\end{proof}

All these results hold as well for complete dg ns operad; we will treat a particular case in the next section. 

\section{Twisted $\calA_\infty$-operad}\label{subsec:TwAInftyOp}

\begin{proposition}\label{prop:MCAi}
The data of a complete $\Ai$-algebra structure together with a Maurer--Cartan element is encoded by the complete dg ns operad 
$$\MC\Ai\coloneqq\left(
\widehat\calT\big(\alpha, \mu_2, \mu_3, \ldots\big), \dd
\right) \ , $$
where $\alpha$ has arity $0$ and degree $-1$ and $\mu_n$ has arity $n$ and degree $n-2$, for $n\ge 2$, where the filtration on the space $M=\big(\k \alpha, 0, \k \mu_2, \k \mu_3, \dots\big)$ of generators is given by 
\[
\alpha \in \F_1 M(0), \ \F_2 M(0)=\{0\} \quad \text{and} \quad \mu_n\in \F_0 M(n),\  \F_1 M(n)=\{0\}\ , \ \text{for} \ n\ge 2
\ , \]
 and where the differential is defined by 
\begin{eqnarray*} 
&&\dd(\mu_n)\coloneqq\sum_{p+q+r=n\atop p+1+r, q\geq 2}(-1)^{pq+r+1} \mu_{p+1+r}\circ_{p+1} \mu_q\ ,
\\
&&\dd(\alpha)\coloneqq-\sum_{n\ge 2}\mu_n(\alpha, \ldots, \alpha) \ .
\end{eqnarray*}
\end{proposition}

\begin{proof}
The fact that the map $\dd$ extends to a unique square-zero  derivation is direct corollary of 
\cref{prop:DGAction} applied to the identity map $\As^{\ac}\to \As^{\ac}$. 
So we get a well-defined complete dg ns operad. 
A complete $\MC\Ai$-algebra structure on a complete dg module $(A, \F , d)$ amounts to a morphism of filtered dg ns operads $\MC\Ai \to \eend_A$ according to \cref{thm:CompleteAlg}. Such an assignment $\alpha \mapsto a$ and $\mu_n\mapsto m_n$ is equivalent to an element $a\in F_1 A_{-1}$ and degree $n-2$ preserving maps $m_n : A^{\otimes n}\to A$ satisfying the relations of a Maurer--Cartan elements in an $\Ai$-algebra by the form of the differential $\dd$. 
\end{proof}

\begin{remark}
Under the convention of \cref{subsec:TwistGrp}, we have $\mu_n=s^{-1}\nu_n$. 
\end{remark}

Inside the dg ns operad $\MC \Ai$, we consider the following elements 
$$\mu_n^\alpha\coloneqq\sum_{r_0, \ldots, r_n\ge 0} (-1)^{\sum_{k=0}^n kr_k} \mu_{n+r_0+\cdots+r_n}\big(\alpha^{r_0}, -, \alpha^{r_1}, -, \ldots,  - , \alpha^{r_{n-1}}, -,\alpha^{r_n}  \big)
\ , $$
for $n\ge 0$, that is for example: 
$$\mu_0^\alpha \coloneqq\sum_{n \ge 2} \mu_n(\alpha, \ldots, \alpha)\quad \text{and} \quad \mu_1^\alpha\coloneqq
\sum_{n\ge 2 \atop 1\leq i \leq n}(-1)^{n-i} \, \mu_n\big(\alpha^{i-1}, -, \alpha^{n-i}\big)\ .
$$

\begin{lemma}\label{Lem:Diffmu}
The elements $\mu_n^\alpha$ satisfy the 
following relations: 
\begin{eqnarray}
\label{eqn4}&&\dd\big(\mu_0^\alpha\big)=0\ ,\\
\label{eqn5}&&\dd\big(\mu_1^\alpha\big)=-\mu_1^\alpha\circ_1\mu_1^\alpha\ ,\\
\label{eqn6}&&\dd\big(\mu_n^\alpha\big)\coloneqq\sum_{p+q+r=n\atop p+1+r, q\geq 1}(-1)^{pq+r+1} \mu^\alpha_{p+1+r}\circ_{p+1} \mu^\alpha_q\ , \ \text{for}\  n\ge 2\ .
\end{eqnarray}
\end{lemma}

Notice that the third formula actually includes the second one. The proofs for these relations are straightforward but quite cumbersome; so, in order not to break the flow of exposition, we postpone them to \cref{AppA}.\\

The second formula says that the element $\mu_1^\alpha$ is a Maurer--Cartan element of the complete dg ns operad $\MC \Ai$. So, the operadic twisting procedure produces the following new complete dg ns operad. 

\begin{definition}[Twisted $\Ai$-operad]
The complete dg ns operad obtained by twisting the complete operad $\MC \Ai$ by the Maurer--Cartan $\mu_1^\alpha$ is called the 
 \emph{twisted $\Ai$-operad} and denoted by 
$$\Tw\Ai\coloneqq \left(\MC \Ai\right)^{\mu_1^\alpha}
=
\left(
\widehat\calT(\alpha, \mu_2, \mu_3, \ldots), \dd^{\mu_1^\alpha}
\right) \ . $$
\end{definition}

The twisted differential is actually equal to
\begin{eqnarray*}
\mathrm{d}^{\mu_1^\alpha}(\alpha)&=&\dd(\alpha)+\mu_1^\alpha(\alpha)= -\sum_{n\ge 2}\mu_n(\alpha, \ldots, \alpha)
+\sum_{n\ge 2}n\mu_n(\alpha, \ldots, \alpha)\\&=&\sum_{n\ge 2}(n-1)\mu_n(\alpha, \ldots, \alpha)\ ,\\
\mathrm{d}^{\mu_1^\alpha}(\mu_n)&=&\dd(\mu_n) 
+\mu_1^\alpha \star \mu_n - (-1)^{n}
\mu_n\star \mu_1^\alpha\\
&=&\sum_{p+q+r=n\atop p+1+r, q\geq 2}(-1)^{pq+r+1} \mu_{p+1+r}\circ_{p+1} \mu_q
\\&&+\sum_{k\ge 2\atop 1\leq i \leq k} 
(-1)^{(k-i)(n+1)}
\mu_k\big(\alpha^{i-1}, \mu_n, \alpha^{k-i}\big)\\
&&
-
\sum_{j=1}^n\sum_{k\ge 2 \atop 1\leq i \leq k} (-1)^{n+k-i}\, \mu_n \circ_j \mu_k\big(\alpha^{i-1}, -, \alpha^{k-i}\big)
\ .
\end{eqnarray*}

\begin{proposition}\label{prop:MorhpAiMCAi}
The assignment $\mu_n\mapsto \mu_n^\alpha$ defines a morphism of complete dg ns operads 
$$\Ai \to \Tw \Ai \ . $$
\end{proposition}

\begin{proof}
The only point to check is the commutativity with the differentials on the generators $\mu_n$ of the quasi-free dg ns operad $\Ai$, that is:
\begin{align*}
\dd^{\mu_1^\alpha}\left(\mu_n^\alpha \right)&=
\dd\left(\mu_n^\alpha \right)+\mu_1^\alpha \star \mu^\alpha_n - (-1)^{n}
\mu^\alpha_n\star \mu_1^\alpha\\ &=
\sum_{p+q+r=n\atop p+1+r, q\geq 1}(-1)^{pq+r+1} \mu_{p+1+r}^\alpha\circ_{p+1} \mu_q^\alpha
+\sum_{p=r=0\atop q=n} \mu_{p+1+r}^\alpha\circ_{p+1} \mu_q^\alpha\\&\ \ \ \ \ 
-(-1)^n\sum_{p+1+r=n \atop q= 1} \mu_{p+1+r}^\alpha\circ_{p+1} \mu_q^\alpha\\&=
\sum_{p+q+r=n\atop p+1+r, q\geq 2}(-1)^{pq+r+1} \mu_{p+1+r}^\alpha\circ_{p+1} \mu_q^\alpha \ ,
\end{align*}
thanks to Formula~(\ref{eqn5}) of Lemma~\ref{Lem:Diffmu}. 
\end{proof}

These operadic results actually
provide us with an alternative proof of  Proposition~\ref{prop:TwCurvedGauge} as follows. The data of a complete $\Ai$-algebra structure  together with a Maurer--Cartan element $a$ on a complete dg module  $(A, \F, d)$ is equivalent to a morphism of complete dg ns operads 
$\rho\ : \ \MC\Ai \to \End_{(A,d)}$, where $\rho(\mu_n)=m_n$ and $\rho(\alpha)=a$ under the previous notations. 
By Proposition~\ref{prop:MorphMCOperadic}, the image of the Maurer--Cartan element $\mu_1^\alpha$ of $\MC\Ai$ under the morphism $\rho$ gives a Maurer--Cartan element on the complete endomorphism ns operad $\eend_A$.
As emphasised above, the complete endomorphism operad twisted by this Maurer--Cartan element is equal to the complete endomorphism operad of the twisted chain complex $(A, d^a)$. 
 Therefore, the second point of Proposition~\ref{prop:MorphMCOperadic} shows that 
$$\widetilde\rho\ : \  \Tw\Ai \to \eend_{(A, d^a)}$$
is a morphism of complete dg ns operads. Pulling back with the aforementioned morphism of complete dg ns operads $\Ai \to \Tw \Ai$, one gets  that the twisted operations $m_n^a$ do form an $\Ai$-algebra structure. 

\begin{remark}
Notice however that the morphism of complete ns operads from $\Ai$ to $\MC \Ai$ which sends the generators $\mu_n$ to $\mu_n^\alpha$ \emph{does not} commute with the differentials. There is a conceptual  reason for this: the twisted operations form an $\Ai$-algebra only with the twisted differential and not with the underlying differential. 
\end{remark}

There is a morphism of complete dg ns operads 
$\Tw \Ai \to \Ai$ which sends $\alpha$ to $0$ and $\mu_n$ to $\mu_n$. On the algebra level, this corresponds to twisting an $\Ai$-algebra with the trivial Maurer--Cartan element. 

\section{Twisting multiplicative ns operads}\label{subsec=TwMultiOp}
The following notion arose from the study of the Deligne conjecture on the Hochschild cochain complex, see \cite{McClureSmith02} for instance. 

\begin{definition}[Multiplicative ns operad]
A \emph{multiplicative ns operad} is a complete dg ns operad $\calP$ equip\-ped with a morphism of complete dg ns operads $\calA_\infty \to \calP$. We still  denote by $\mu_2$, $\mu_3$, etc. the images in $\calP$ of the generating operations of $\calA_\infty$. 
\end{definition}

Therefore any complete dg algebra over a multiplicative ns operad acquires a natural complete $\calA_\infty$-algebra structure. \\

The categorical coproduct of two ns operads is denoted by $\calP\vee \calQ$ (respectively by $\calP\hat{\vee} \calQ$ in the complete case). It is given by the free ns operad on the underlying $\mathbb{N}$-modules of $\calP$ and $\calQ$ modulo the ideal generated by planar trees with two vertices labelled both by elements of $\calP$ or  by elements of $\calQ$. Thus planar trees with vertices labelled alternatively by elements from $\calP$ and $\calQ$ are representatives for this coproduct. Its operadic composition is given by the grafting of planar trees, followed possibly by the composite of two adjacent vertices labelled by elements from the same ns operad.

By a slight abuse of notation, we simply denote by $\alpha$ the (trivial) complete ns operad $(\k \alpha, 0,\ldots)$ spanned by the arity $0$ element $\alpha$. Thus  $\calP\hat{\vee} \alpha$ denotes  the complete coproduct of a complete ns operad $\calP$ with it. As a consequence of the above-mentioned description, this coproduct is made up of series indexed, by $n\in \mathbb{N}$, of linear combinations of operations from $\calP$ with $n$ copies of $\alpha$  plugged at their inputs. 

\begin{proposition}\label{prop:MCP}
Let $\calA_\infty \to \calP$ be a multiplicative ns operad. 
The data of a complete $\calP$-algebra structure together with a Maurer--Cartan element is encoded by the complete dg ns operad 
$$\MC\calP\coloneqq\left(
\calP\hat{\vee} \alpha, \dd
\right) \ , $$
where $\alpha$ is a degree $-1$ element of arity $0$ placed in $F_1$ and where $\hat{\vee}$ stands for the coproduct of complete ns operads, and where the differential $\mathrm{d}$ is characterized by 
\begin{eqnarray*} 
&&\dd(\alpha)\coloneqq-\sum_{n\ge 2}\mu_n(\alpha, \ldots, \alpha) \ ,\\
&&\dd(\nu)\coloneqq d_\calP(\nu)\ , \ \text{for} \ \nu\in \calP\ .
\end{eqnarray*}
\end{proposition}

\begin{proof}The fact that the map $\dd$ extends to a unique square-zero  derivation is direct corollary of 
\cref{prop:DGAction} applied to the map $\As^{\ac}\to \calP$ induced by the multiplicative ns operad structure.
The rest of the proof is straightforward from the definition of the coproduct of complete ns operads. 
\end{proof}

\begin{lemma}
The element $\mu_1^\alpha$ is a Maurer--Cartan element of the complete dg ns operad $\MC\calP$. 
\end{lemma}

\begin{proof}
The morphism of complete dg ns operads $\calA_\infty \to \calP$ induces a morphism of complete dg ns operads $\MC\calA_\infty \to \MC\calP$ and we know, by Proposition~\ref{prop:MorphMCOperadic}, that the image of an operadic Maurer--Cartan element is again an operadic Maurer--Cartan element. 
\end{proof}

\begin{definition}[Twisted multiplicative ns operad]
Let $\calA_\infty \to \calP$ be a multiplicative ns operad. 
The complete dg ns operad obtained by twisting the operad $\MC \calP$ by the Maurer--Cartan $\mu_1^\alpha$ is called the 
 \emph{twisted complete ns operad} and denoted by 
$$\Tw\calP\coloneqq \left(\MC \calP\right)^{\mu_1^\alpha}=
\left(
\calP\hat{\vee} \alpha, \dd^{\mu_1^\alpha}
\right) \ . $$
\end{definition}

The twisted differential is actually equal to
\begin{eqnarray*}
\mathrm{d}^{\mu_1^\alpha}(\alpha)&=&\dd(\alpha)+\mu_1^\alpha(\alpha)= -\sum_{n\ge 2}\mu_n(\alpha, \ldots, \alpha)
+\sum_{n\ge 2}n\mu_n(\alpha, \ldots, \alpha)\\&=&\sum_{n\ge 2}(n-1)\mu_n(\alpha, \ldots, \alpha)\ ,\\ \mathrm{d}^{\mu_1^\alpha}(\nu)&=&d_\calP(\nu) 
+\mu_1^\alpha \star \nu - (-1)^{|\nu|}
 \nu\star \mu_1^\alpha\\
 &=&d_\calP(\nu) 
+
\sum_{n\ge 2\atop 1\leq i \leq n} 
(-1)^{(n-i)(|\nu|+1)}
\mu_n\big(\alpha^{i-1}, \nu, \alpha^{n-i}\big)\\
&&
-
\sum_{j=1}^k\sum_{n\ge 2 \atop 1\leq i \leq n} (-1)^{|\nu|+n-i}\, \nu \circ_j \mu_n\big(\alpha^{i-1}, -, \alpha^{n-i}\big)
\ ,
\end{eqnarray*}
for $\nu \in \calP(k)$. 

\begin{proposition}
Any complete dg $\calP$-algebra $(A, d)$ with a given Maurer--Cartan element $a$ gives a complete $\Tw \calP$-algebra with underlying twisted differential $(A, d^a)$.This assignment 
defines a functor $\MC \calP\emph{\textsf{-alg}} \to \Tw \calP\emph{\textsf{-alg}}$, 
which is an isomorphism of categories.
\end{proposition}

\begin{proof}
Let us reproduce here the argument given above in the $\Ai$ case. The data of a complete $\calP$-algebra structure on  $(A,\F, d)$ with a Maurer--Cartan element $a$ is equivalent to a morphism of complete dg ns operads $\MC \calP\to \eend_{(A,d)}$. Since the element $\mu_1^\alpha$ is a Maurer--Cartan element in the complete ns operad $\calP$, we get a morphism between the twisted complete dg ns operads $\Tw \calP \to \End_{(A,d+m_1^a)}$. In the other way round, one uses the exact same arguments but starting with the Maurer--Cartan element $-\mu_1^\alpha$ of the complete dg ns operad $\Tw \calP$, see Proposition~\ref{prop:-MC}. This shows that a complete algebra structure over the complete dg ns operad $\Tw \calP \to \End_{(A,d)}$
induces a complete algebra structure over the complete dg ns operad 
$$\MC \calP=\left(\Tw\calP\right)^{-\mu_1^\alpha} \to \End_{(A,d-m_1^a )}\ , $$
by Proposition~\ref{prop:-MC}. These two functors $\MC \calP\textsf{-alg} \to \Tw \calP\textsf{-alg}$ and 
$\Tw \calP\textsf{-alg} \to \MC \calP\textsf{-alg}$ are  inverse to each other.
\end{proof}

\begin{proposition}\label{prop:MorhpAiMCP}
The assignment $\mu_n\mapsto \mu_n^\alpha$ defines a morphism of complete dg ns operads 
$$\Ai \to \Tw \calP \ . $$
\end{proposition}

\begin{proof}
The aforementioned morphism $\MC\calA_\infty \to \MC\calP$ of complete dg ns operads induces a morphism between the twisted complete dg ns operads $\Tw\calA_\infty \allowbreak \to \Tw\calP$ by Proposition~\ref{prop:MorphMCOperadic}. It just remains to pull back by the map 
$\Ai \to \Tw\Ai$ of Proposition~\ref{prop:MorhpAiMCAi}.
\end{proof}

This proposition shows that the upshot of the operadic twisting procedure gives again a multiplicative ns operad. From now on, we will call the category of complete
dg ns operads under $\Ai$ the \emph{category of multiplicative ns operads}.

\begin{lemma}
The twisting procedure defines an endofunctor on the category of multiplicative ns operads.
\end{lemma}

\begin{proof}
Proposition~\ref{prop:MorhpAiMCP} shows that the result of the operadic twisting lives in the category of multiplicative ns operads. 
Given a morphism$$\xymatrix{&\Ai\ar[dl]\ar[dr]&\\
\calP\ar[rr]^f &&\calQ}$$
of multiplicative ns operads, we define a morphism of complete dg ns operads $\Tw f :  \Tw\calP \to \Tw\calQ$ by $
\alpha \mapsto \alpha$ and 
$\nu \mapsto f(\nu)$, for $\nu \in \calP$.
Then, the compatibility relations $\Tw \id_\calP=  \id_{\Tw\calP}$ and  $\Tw (f\circ g)=\Tw f \circ \Tw g$ are automatic. 
\end{proof}

\begin{lemma}\label{Lemma=TwTw}
Let $\calP$ be a multiplicative ns operad. The complete dg ns operad $\Tw (\Tw \calP)$ is isomorphic to 
$$\Tw \big(\Tw \calP\big)\cong  \big(\calP\hat{\vee} \, \alpha\, \hat{\vee}\,  \beta,  \dd+\ad_{\mu_1^{\alpha+\beta}},  \left\{ \circ_i \right\}\big)\ .$$
\end{lemma}

\begin{proof}
The only point to check is the  twisted differential.
First, one shows that the 
$\mu_1^{\alpha+\beta}$ 
is equal to  $\mu_1^\alpha+\widetilde{\mu}_1^\beta$, where we denote here the twisted operations by $\widetilde{\mu}_n\coloneqq\mu_n^\alpha$:
\begin{eqnarray*}
\mu_1^\alpha+\widetilde{\mu}_1^\beta&=&
\sum_{k\ge 2 \atop 1\leq j \leq k}(-1)^{k-j} \, \mu_k\big(\alpha^{j-1}, -, \alpha^{k-j}\big)+
\sum_{n\ge 2 \atop 1\leq i \leq n}(-1)^{n-i} \, \mu_n^\alpha\big(\beta^{i-1}, -, \beta^{n-i}\big)\\
&=&
\sum_{k\ge 2 \atop 1\leq j \leq k}(-1)^{k-j} \, \mu_k\big(\alpha^{j-1}, -, \alpha^{k-j}\big)+\\
&& \sum_{n\ge 2 \atop 1\leq i \leq n}(-1)^{n-i} \,\sum_{r_0, \ldots, r_n\ge 0} (-1)^{r_i+\cdots +r_n} \mu_{n+r_0+\cdots+r_n}\big(\alpha^{r_0}, \beta, \alpha^{r_1}, \beta, \ldots, \alpha^{r_{i-1}}, - , \\
&&\qquad\qquad\qquad\qquad\qquad\qquad\qquad\qquad\qquad\qquad\quad  \alpha^{r_{i_1}}, \beta, \ldots, \alpha^{r_{n-1}}, \beta,\alpha^{r_n}  \big)\\
&&=
\sum_{k\ge 2 \atop 1\leq j \leq k}(-1)^{k-j} \, \mu_k\big(\alpha^{j-1}, -, \alpha^{k-j}\big)+\\
&&\quad \sum_{n\ge 2 \atop 1\leq i \leq n}\,\sum_{r_0, \ldots, r_n\ge 0} (-1)^{n-i+r_i+\cdots +r_n} \mu_{n+r_0+\cdots+r_n}\big(\alpha^{r_0}, \beta, \alpha^{r_1}, \beta, \ldots, \alpha^{r_{i-1}}, - , \\
&&\qquad\qquad\qquad\qquad\qquad\qquad\qquad\qquad\qquad\qquad\quad 
\alpha^{r_{i_1}}, \beta, \ldots, \alpha^{r_{n-1}}, \beta,\alpha^{r_n}  \big)\\
&&= \sum_{k\ge 2 \atop 1\leq j \leq k}(-1)^{k-j} \, \mu_k\big((\alpha+\beta)^{j-1}, -, (\alpha+\beta)^{k-j}\big)=\mu_1^{\alpha+\beta}\ .
\end{eqnarray*}
Then, one concludes the proof with 
$$\ad_{\mu_1^\alpha}+\ad_{\widetilde{\mu}_1^\beta}=\ad_{\mu_1^\alpha+\widetilde{\mu}_1^\beta}=\ad_{\mu_1^{\alpha+\beta}} \ .$$
\end{proof}

\begin{remark}
This result is an operadic version  of the fact that $\alpha+\beta$ is a Maurer--Cartan in the initial (dg Lie) algebra if and only if $\beta$ is a Maurer--Cartan element in the  (dg Lie) algebra twisted by $\alpha$. Then the (dg Lie) algebra twisted first by $\alpha$ and then by $\beta$ is equal to the (dg Lie) algebra twisted by $\alpha+\beta$, see Proposition~\ref{lem:subGaugeGroup}. 
\end{remark}

\begin{corollary}\label{cor:ComStr}
Given any multiplicative ns operad $\calP$, the assignment 
\begin{eqnarray*}
\begin{array}{rrcl}
\Delta(\calP)\ \ \  : &\Tw \calP\cong \calP\hat{\vee} \alpha&\to &\Tw\, \big(\Tw \calP\big)\cong \calP\hat{\vee} \alpha \hat{\vee} \beta\\
&\alpha &\mapsto & \alpha+\beta\\
&\nu&\mapsto&\nu\ ,
\end{array}
\end{eqnarray*}
for $\nu\in\calP$, 
defines a morphism of multiplicative ns operads. 
\end{corollary}

\begin{proof}
Again, the only point to check is the compatibility with the differential. Let us denote the above given morphism by $f : \Tw \calP \to \Tw \big(\Tw \calP\big)$. For any element $\nu \in \calP$, we have
\begin{eqnarray*}
f\left(
d_\calP(\nu)+\ad_{\mu_1^\alpha}(\nu)
\right)
=d_\calP(f(\nu))+\ad_{\mu_1^{\alpha+\beta}}(f(\nu))\ ,
\end{eqnarray*}
and
\begin{eqnarray*}
f\circ\big(
\dd+\ad_{\mu_1^\alpha}
\big)(\alpha)&=&
f\left(\sum_{n\ge 2}(n-1)\mu_n(\alpha, \ldots, \alpha)\right)=
\sum_{n\ge 2}(n-1)\mu_n(\alpha+\beta, \ldots, \alpha+\beta)\\
&=&\left(\dd +\ad_{\mu_1^{\alpha+\beta}}\right)\circ (\alpha+\beta)
= \left(\dd +\ad_{\mu_1^{\alpha+\beta}}\right)\circ f(\alpha)\ . 
\end{eqnarray*}
We conclude with Lemma~\ref{Lemma=TwTw}.
\end{proof}

Similarly to the case of the twisted $\Ai$ operad, there is a morphism 
of multiplicative ns operads 
$\varepsilon(\calP)\ : \ \Tw \calP \to \calP$ defined by sending $\alpha$ to $0$ and $\nu\in \calP$ to $\nu$. 

\begin{theorem}
The two aforementioned morphisms $\Delta(\calP) \ :  \Tw \calP \to \Tw(\Tw\calP)$ and $\varepsilon(\calP)\ : \ \Tw \calP \to \calP$ of multiplicative ns operads provide the endofunctor $\Tw$ with a comonad structure. 
\end{theorem}

\begin{proof}
Let  $\Ai\to\calP$ be a  multiplicative ns operad, we have to check the counit relations 
\begin{eqnarray*}
\big(\varepsilon(\calP)\circ(\id_{\Tw \calP})\big)
(\Delta(\calP))=\id_{\Tw \calP}= 
\big((\id_{\Tw \calP})\circ \varepsilon(\calP)\big)
\end{eqnarray*}
and the coassociativity relation 
\begin{eqnarray*}
\Delta_{\Tw \calP}(\Delta(\calP))
=\Tw(\Delta_{\calP})(\Delta(\calP))\ .
\end{eqnarray*}
In each cases, the image of any element $\nu \in \calP$ is send to itself. The left counit relation is given by 
$$\alpha \mapsto \alpha +\beta \mapsto \alpha$$
since the second morphism sends $\alpha$ to $\alpha$ and $\beta$ to $0$. The right counit relation is proved similarly since the second morphism sends $\alpha$ to $0$ and $\beta$ to $\alpha$.
Both sides for the coassociativity relation give 
$$\alpha \mapsto \alpha +\beta \mapsto \alpha+\beta+\gamma \ , $$
which concludes the proof. 
\end{proof}

\begin{definition}[Twistable operad]
We call \emph{twistable operad} a multiplicative ns operad which admits a $\Tw$-coalgebra structure. 
\end{definition}

The aforementioned definition means that, given a multiplicative ns operad $\Ai\to \calP$, there exists a morphism $\Delta_\calP$ of multiplicative ns operads 
$$\xymatrix@C=20pt@R=20pt{&\Ai\ar[dl]\ar[dr]&\\
\calP\ar[rr]_{\Delta_\calP} &&\Tw \, \calP}$$
 satisfying 
\begin{eqnarray}
&&\xymatrix@C=30pt{\calP \ar[r]^(0.42){\Delta_\calP} \ar@/_1pc/[rr]_{\id_\calP}& \Tw\, \calP \ar[r]^(0.56){\varepsilon(\calP)}& \calP \ ,} \label{Rel1}\\
&&\xymatrix@C=30pt@R=30pt{\calP \ar[r]^{\Delta_\calP} \ar[d]^{\Delta_\calP} &   \Tw\, \calP \ar[d]^{\Tw(\Delta_\calP)} \\ 
\Tw\, \calP \ar[r]^(0.45){\Delta(\calP)}& \Tw(\Tw\, \calP)\ .}\label{Rel2}
\end{eqnarray}

In terms of type of algebras, the map $\Delta_\calP$ gives a concrete way to produce functorial complete $\calP$-algebra structures on any complete $\calP$-algebra endowed with a Maurer--Cartan element coming from the internal $\Ai$-algebra structure but with twisted differential.
Indeed, as explained above, the data of a complete $\calP$-algebra structure with a Maurer--Cartan element $a$ is faithfully encoded in 
a morphism of dg ns operads $\MC\calP \to \eend_{(A,d)}$, which gives rise to a morphism of twisted complete dg ns operads 
$\Tw\calP \to \End_{(A,d^a)}$ by Proposition~\ref{prop:MorphMCOperadic}. Pulling back with the morphism of complete dg ns operads $\Delta(\calP) : \calP \to \Tw\calP$ produces the twisted $\calP$-algebra structure. \\

The fact that the structure map $\Delta_\calP$ is a morphism of multiplicative ns operads says that one has to twist the $\Ai$-operations in $\calP$ as usual, that is according to the formulas given in Proposition~\ref{prop:TwCurvedGauge}.
Relation~(\ref{Rel1}) expresses the fact that the twisted operation associated to any  $\nu\in \calP$ is the sum of two terms: the first one being equal to $\nu$ itself and the second one begin the sum of perturbation terms which all contain at least one Maurer--Cartan element. Relation~(\ref{Rel2}) amounts to say that the operations twisted twice under the same formulas, first by a Maurer--Cartan element $a$ and then by a second Maurer--Cartan element $b$, are equal to the operations twisted once by the Maurer--Cartan element $a+b$, thanks to Lemma~\ref{Lemma=TwTw} and Corollary~\ref{cor:ComStr}. These are the constrains of twistable multiplicative ns operad. 

\begin{example}
The ns operad $\Ai$ is the paradigm of twistable ns operad. Its $\Tw$-coalgebra structure map $\Ai \to \Tw\, \Ai$ is given by Proposition~\ref{prop:MorhpAiMCAi}.  
\end{example}

\begin{proposition}\label{prop:Twistable}
Let $\calP$ be a multiplicative ns operad with trivial differential $d_\calP=0$. 
If $\calP$ is twistable, then any element $\nu\in \calP$ satisfies $\ad_{\mu_2(\alpha, -)-\mu_2(-, \alpha)}(\nu)=0$ in $\Tw\calP$.  
When the multiplicative structure of $\calP$ factors through the canonical resolution 
$\Ai \twoheadrightarrow \As \to \calP$, the  reverse statement  holds true.
\end{proposition}

\begin{proof}
If the operad $\calP$ is twistable, it  admits a morphism of multiplicative ns operads 
$\Delta_\calP: \calP \to \Tw\calP$ satisfying the commutative diagrams \eqref{Rel1} and \eqref{Rel2}. 
So the image of any element $\nu\in \calP(n)$ has the form 
\[\Delta_\calP(\nu)=\nu+\sum_{k\ge 1} \omega_k\ ,\]
where $\omega_k$ is an element of $\Tw\calP(n)$ made up of a finite sum of elements of $\calP(n+k)$ composed with $k$ elements $\alpha$. The compatibility with respect to the differentials shows that 
\[
\dd^{\mu_1^\alpha}\big(\Delta_\calP(\nu)\big)=
\ad_{\mu_2(\alpha, -)-\mu_2(-, \alpha)}(\nu)+\sum_{k\ge 2} \widetilde{\omega}_k=0\ , 
\]
where $\widetilde{\omega}_k$ is made up of a finite sum of elements of $\calP(n+k)$ composed with $k$ elements $\alpha$. Therefore, $\ad_{\mu_2(\alpha, -)-\mu_2(-, \alpha)}(\nu)$ vanishes since it is made up of only one element $\alpha$.

In the other way round, the multiplicative structure of $\calP$ factors through the canonical resolution 
$\Ai \twoheadrightarrow \As \to \calP$ if and only if the elements $\mu_3, \mu_4, \ldots$ vanish in $\calP$. In this case, the twisted differential is equal to $\dd^{\mu_1^\alpha}\allowbreak=\allowbreak\ad_{\mu_2(\alpha, -)-\mu_2(-, \alpha)}$.
The condition $\ad_{\mu_2(\alpha, -)-\mu_2(-, \alpha)}(\nu)=0$,  for all $\nu\in \calP$, is equivalent to the fact  that the canonical morphism of ns operads $\calP \hookrightarrow \Tw\calP$ is a chain map. The commutative diagrams \eqref{Rel1} and \eqref{Rel2} are then straightforward to check. 
\end{proof}

\begin{example}
The ns operads $\ncGerst$ and $\ncBV$, introduced in \cite[Section~3]{DotsenkoShadrinVallette15} as noncommutative analogues of the classical operads, are not twistable, see \cref{prop:ncGerstnotTwistable} in \cref{subsec:ncGerst} and \cref{prop:ncBVnotTwistable} in \cref{subsec:ncBV} respectively.
\end{example}

\begin{example}
The operad $\Gerst$ is twistable but the operad $\BV$ is not,  
see respectively 
\cref{prop:GerstTwistable} in \cref{subsec:TwGerst} and \cref{prop:BVnotTwistable} in  \cref{subsec:TwBV} respectively.
\end{example}

Among the homological properties, let us recall the following stability principle. 

\begin{proposition}\label{prop:TwQI}
The endofunctor $\Tw$ preserves quasi-isomorphisms. 
\end{proposition}

\begin{proof}
A straightforward proof  is given in \cite[Theorem~$5.1$]{DolgushevWillwacher15}. 
\end{proof}

\section{Action of the deformation complex}
Let $\calP$ be a complete dg ns operad. We consider the  total space 
$$\hom\left(\As^{\ac}, \calP\right)\coloneqq\prod_{n\ge 1} \hom\left(\As^{\ac}(n), \calP(n)\right)\cong \prod_{n\ge 1} s^{1-n}\calP(n)\ , $$ 
where we identify any map $\rho(n) : \As^{\ac}(n)\cong \End_{\k s}(n)^* \to \calP(n)$ with its images $\rho_n \coloneqq \rho\left(\nu_n\right)$. 
Let us recall from \cref{subsec:CompConvAlg} that this forms a  complete left-unital dg pre-Lie algebra and thus a complete dg Lie algebra by anti-symmetrization. 

Any  element $\rho=(\rho_1, \rho_2, \ldots)\in \hom\left(\As^{\ac}, \calP\right)$ induces the following derivation $\D_\rho$ on the complete ns operad $\calP\hat{\vee} \{\alpha\}$ by sending its generators to 
\begin{eqnarray*}
&&\alpha \mapsto -\sum_{n\ge 1 } \rho_n\big(\alpha^n\big)\ ,\\
&&\nu \mapsto 0, \quad \text{for} \ \nu \in \calP \ .
\end{eqnarray*}
We denote by $\Der\big(\calP\hat{\vee} \alpha \big)$ the set of operadic derivations and,  
by a slight abuse of notation, we still denote by $d_\calP$ the differential on $\calP\hat{\vee} \alpha$ induced by that of $\calP$.

\begin{lemma}\label{lem:MorphdgLie}
The assignment 
\begin{eqnarray*}
\left(\hom\left(\As^{\ac}, \calP\right), \partial, [\; , \,] \right)&\to& 
\left(\Der\big(\calP\hat{\vee} \alpha\big), [d_\calP, -], [\; ,\,]\right)\\
\rho &\mapsto& \D_\rho
\end{eqnarray*}
is a morphism of dg Lie algebras.
\end{lemma}

\begin{proof}
Notice first that the Lie bracket on the right-hand side is given by the skew-symmetrization of the following binary product $\D\circ^{\mathrm{op}}\D'\coloneqq -(-1)^{|\D||\D'|} \allowbreak \D'\circ \D$ (which individually does not produce a derivation). Since the Lie bracket on the left-hand side is given by the skew-symmetrization of the pre-Lie product $\star$, we prove that the assignment $\rho\mapsto\D_\rho$ preserves these two products. Let $\rho, \xi \in \hom\left(\As^{\ac}, \calP\right)$. It is enough to check the relation 
$\D_{\rho\star\xi} = \D_\rho\circ^{\mathrm{op}} \D_\xi$ 
on the generators of $\calP\hat{\vee} \alpha$: this is trivial for $\nu \in \calP$ and for $\alpha$ this is given by 
\begin{align*}
\D_{\rho\star \xi}(\alpha)&=-
\sum_{n\ge 1 } (\rho\star \xi)_n\big(\alpha^n\big)=-\sum_{n\ge 1 \atop p+q+r=n}  
(-1)^{p(q+1)+|\xi|(p+r)}
\rho_{p+1+r}\circ_{p+1} \xi_{q} \big(\alpha^n\big)=
\\&=
-(-1)^{|\rho||\xi|}\D_\xi\circ \D_\rho(\alpha) 
=
\D_\rho\circ^{\mathrm{op}}\D_\xi(\alpha)\ .
\end{align*}
We also check the commutativity of the differentials $\D_{\partial(\rho)}=[d_\calP, \D_\rho]$ on the generators of $\calP\hat{\vee} \alpha$: this is again trivial for $\nu \in \calP$ and for $\alpha$ this is given by 
\begin{align*}
\D_{\partial(\rho)}(\alpha)&=\D_{d_\calP \circ \rho}(\alpha)=-\sum_{n\ge 1} d_\calP(\rho_n)(\alpha^n)= d_\calP\left(
\D_\rho(\alpha)\right)-(-1)^{|\rho|}\D_\rho(d_\calP(\alpha))\\&=[d_\calP, \D_\rho](\alpha)\ .
\end{align*}
\end{proof}

A morphism $\Ai \to \calP$  of  complete dg ns operads is equivalent to a degree $-1$ element 
$\mu
 \coloneqq(0, \mu_2, \mu_3, \ldots)$, notation which agrees with that of \cref{subsec=TwMultiOp}, 
satisfying the Maurer--Cartan equation 
$$ \partial (\mu)+\mu \star \mu=0\ ,$$ 
 in this dg pre-Lie algebra. Therefore one can twist the associated dg Lie algebra  with this Maurer--Cartan element, that is consider the twisted differential 
 $$\partial^\mu\coloneqq\partial + \ad_\mu\ . $$
 (One cannot twist the dg pre-Lie algebra, unless $\mu$ satisfies Equation~\eqref{Eq:MCspecial}, which happens only when it vanishes completely).
 
\begin{definition}[Deformation complex of morphisms of complete dg ns operads \cite{MerkulovVallette09I, MerkulovVallette09II}]
The \emph{deformation complex} of the morphism $ \Ai\to \calP$ of complete dg ns operads is the complete  twisted dg Lie algebra 
$$\Def\big(\Ai\to \calP\big)\coloneqq\left(\hom\left(\As^{\ac}, \calP\right), \partial^\mu, [\; , \,] \right)
\ . $$
\end{definition}

\begin{proposition}\label{prop:DGAction}
The assignment 
\begin{eqnarray*}
\left(\hom\left(\As^{\ac}, \calP\right), \partial^\mu, [\; , \,] \right)&\to&
\left(\Der\big(\calP\hat{\vee} \alpha\big), [d_\calP+\D_\mu, -], [\; ,\,]
\right)\\
\rho &\mapsto& \D_\rho
\end{eqnarray*}
is a morphism of dg Lie algebras.
 In plain words, this 
defines a dg Lie action by derivation of the deformation complex $\Def\big(\Ai\to \calP\big)$ on the Maurer--Cartan operad 
$\MC\calP$. 
\end{proposition}

\begin{proof}
This is a direct corollary of the morphism of dg Lie algebras established in \cref{lem:MorphdgLie}: the Maurer--Cartan element $\mu$ on the left-hand side is sent to the Maurer--Cartan $\D_\mu$ on the right-hand side. This proves that $d_\calP+\D_\mu$ is a square-zero derivation on the complete ns operad $\calP\hat{\vee} \alpha$. In the end, we get a morphism between the respectively twisted dg Lie algebras. 
\end{proof}

\begin{remark}\label{rem:}
\cref{prop:DGAction} provides us with an alternative proof of \cref{prop:MCP} defining the complete dg ns operad 
\[
\MC\calP\coloneqq\left(
\calP\hat{\vee} \alpha, \dd\coloneqq d_\calP+\D_\mu\right)\ .
\]
\end{remark}

\begin{remark}
It is also worth noticing that, due to the twisting, this does  
 not define a pre-Lie action, but just a Lie action. 
\end{remark}

In order to reach the same kind of results for the complete dg ns operad $\Tw \calP$, whose differential contains one more term then that of  $\MC \calP$, 
we need to consider the following extension of  the deformation complex. 
Notice first that the dg Lie algebra action of \cref{prop:DGAction} on the complete dg ns operad $\calP \hat{\vee} \alpha$ reduces to a dg Lie algebra action on the dg Lie algebra made up of the arity $1$ elements 
$\big(\calP \hat{\vee} \alpha\big)(1)$. This gives rise to the following semi-direct product of dg Lie algebras: 
\begin{eqnarray*}
\left(\hom\left(\calA s^{\ac}, \calP\right), \partial, [\; , \,] \right) \ltimes 
\left(\big(\calP\hat{\vee} \alpha\big)(1), d_\calP, [\; ,\,]\right)\ .
\end{eqnarray*}

\begin{lemma}\label{lem:semidirectPreLie}
The above-mentioned semi-direct product dg Lie algebra 
 comes from the skew-sym\-me\-tri\-za\-tion of the following semi-direct product of dg pre-Lie algebras
\begin{eqnarray*}
\left(\hom\left(\As^{\ac}, \calP\right), \partial, \star \right) \ltimes 
\left(\big(\calP\hat{\vee} \alpha\big)(1), d_\calP, \circ_1\right)
\coloneqq
\left(\hom\left(\As^{\ac}, \calP \right)\oplus \left(\calP\hat{\vee} \alpha\right)(1), 
\partial+d_\calP, \bigstar
\right)\ ,
\end{eqnarray*}
where 
\[
(\rho, \nu)\bigstar (\xi, \omega)\coloneqq 
\left(\rho\star \xi, \nu \circ_1 \omega - (-1)^{|\xi||\nu|}\D_\xi(\nu)
\right)\ .\]
\end{lemma}

\begin{proof}
The proof of \cref{lem:MorphdgLie} shows that the assignement $\rho\mapsto \D_\rho$ defines a right dg pre-Lie action of $\left(\hom\left(\As^{\ac}, \calP\right), \partial, \star \right)$ on $\left(\big(\calP\hat{\vee} \alpha\big)(1), d_\calP, \circ_1\right)$. It is however not always true that dg pre-Lie actions give rise to semi-direct product dg pre-Lie algebras under  formulas like that of $\bigstar$. It is the case here since the action is by derivation, see  \cite{ManchonSaidi08} for another occurrence of this construction. If we denote the associator of a binary product $\star$ by $\mathrm{Assoc_\star}$, we have  
\begin{align*}
&\mathrm{Assoc_\bigstar\big( (\rho, \nu), (\xi, \omega), (\theta, \lambda) \big)}\\
=&
\left(
\mathrm{Assoc}_\star( \rho, \xi, \theta), \mathrm{Assoc}_{\circ_1}( \nu, \omega, \lambda)-(-1)^{|\xi||\nu|}\D_\xi(\nu)\circ_1 \lambda -
(-1)^{|\theta|(|\nu|+|\omega|)} \D_\theta(\nu)\circ_1 \omega
\right)\ , 
\end{align*}
which is right symmetric, see also \cite{ManchonSaidi08}. The compatibility of the differentials follows from \cref{lem:MorphdgLie}.

Finally, the skew-symmetrization of this semi-direct product pre-Lie algebra gives
\[
\left[(\rho, \nu),  (\xi, \omega)\right]=
\left(\left[\rho, \xi\right], \left[\nu,  \omega\right] +\D_\rho(\omega)- (-1)^{|\xi||\nu|}\D_\xi(\nu)
\right)\ ,\]
which is  the formula for the Lie bracket of the semi-direct product Lie algebra. 
\end{proof}

\begin{lemma}\label{lem:MorphdgLieBIS}
 The assignment 
\begin{eqnarray*}
 \left(\hom\left(\As^{\ac}, \calP\right), \partial, [\; , \,] \right) &\to& \left(\hom\left(\As^{\ac}, \calP\right), \partial, [\; , \,] \right)
\ltimes 
\left(\big(\calP\hat{\vee} \alpha\big)(1), d_\calP, [\; ,\,]\right)
\\
\rho &\mapsto& \left(\rho, \rho_1^\alpha\right)\ ,
\end{eqnarray*}
with $\rho_1^\alpha\coloneqq
\displaystyle  \sum_{n\ge 1 \atop 1\leq i \leq n}(-1)^{n-i} \, \rho_n(\alpha^{i-1}, -, \alpha^{n-i})$, defines a morphism of dg Lie algebras.
\end{lemma}

\begin{proof}
By \cref{lem:semidirectPreLie}, 
it is enough to prove that such an assignment 
 defines a morphism of dg pre-Lie algebras. 
To this extend, we  first prove 
\begin{equation}\label{eqn:PreLieAction}
(\rho \star \xi)^\alpha_1=\rho_1^\alpha \circ_1 \xi_1^\alpha -(-1)^{|\rho||\xi|}\D_\xi \left(\rho_1^\alpha\right)\ ,
\end{equation}
for any $\rho, \xi\in \hom\left(\calA s^{\ac}, \calP\right)$. The left-hand side is equal to 
\[
(\rho \star \xi)^\alpha_1=\sum_{ p+q+r=n\geq 1 \atop 1\leq i \leq n}
(-1)^{p(q-1)+|\xi|(p+r)+n-i}\rho_{p+1+r}\circ_{p+1}\xi_q\left(\alpha^{i-1}, -, \alpha^{n-i}\right)\ ,
\]
which splits into three components according to the value of $i$: (i) when $1\leq i\leq p$, (ii) when $p+1\leq i\leq p+q$, and (iii) when $p+q+1\leq i\leq n=p+q+r$\ . The first term on the right-hand side of \eqref{eqn:PreLieAction} corresponds to the component (ii) and the second term on the right-hand side of \eqref{eqn:PreLieAction} corresponds to the sum of the  two components (i) and (iii). Explicitly, we first have 
\begin{align*}
\rho_1^\alpha \circ_1\xi_1^\alpha&=
\sum_{p+q+r=n\geq 1\atop 1\leq j \leq q} (-1)^{r+q-j} \rho_{p+1+r}(\alpha^p, -, \alpha^r)\circ_1\xi_q\left(\alpha^{j-1}, -, \alpha^{q-j}\right)\\
&=\sum_{p+q+r=n\geq 1\atop 1\leq j \leq q} 
(-1)^{p(q-1)+|\xi|(p+r)+r+q-j} 
\rho_{p+1+r}\circ_{p+1} \xi_q \left(\alpha^{p+j-1}, -, \alpha^{r+q-j}\right)
\end{align*}
This gives (ii) with $i=j+p$, since then $n-i=r+q-j$. Regarding the second term on the right-hand side of \cref{eqn:PreLieAction}, since 
$\rho_1^\alpha=
\displaystyle  \sum_{n\ge 2 \atop 1\leq i \leq n}(-1)^{n-i} \, \rho_n(\alpha^{i-1}, -, \alpha^{n-i})$ and since $\D_\xi$ vanishes on $\rho_n$, for $n\geq 1$, we get two terms: the first one when $\D_\xi$ applies to the $\alpha$'s on the left-hand side of the input slot and the second one when $\D_\xi$ applies to the $\alpha$'s on the right-hand side of the input slot. The former term gives component (iii) and the latter term gives component (i). The last point of the proof  amounts to check the various signs.  Under the notation 
$\rho_1^\alpha=
\sum_{p+1+r=n\geq 1\atop 1\leq j\leq r} 
(-1)^{r-j}
\rho_{p+1+r}\allowbreak\left(\alpha^{p+1+j-1}, \allowbreak -, \allowbreak \alpha^{r-j}\right)$, the former term becomes 
\begin{align*}
&\sum_{p+1+r=n\geq 1\atop 1\leq j\leq r} 
(-1)^{r-j+|\xi|r}
\rho_{p+1+r}\left(\alpha^{p},\xi_q\left(\alpha^q\right),\alpha^{j-1}, -, \alpha^{r-j}\right)\\
=&
\sum_{p+1+r=n\geq 1\atop 1\leq j\leq r} 
(-1)^{p(q-1)+|\xi|(p+r)+r-j}
\rho_{p+1+r}\circ_{p+1}\xi_q\left(\alpha^{p+q+j-1}, -, \alpha^{r-j}\right)\ , 
\end{align*}
which  is equal to (iii) with $i=j+p+q$.
Under the notation 
$\rho_1^\alpha \allowbreak = \allowbreak 
\sum_{p+1+r=n\geq 1\atop 1\leq j\leq p}\allowbreak 
(-1)^{p+1+r-j}\allowbreak
\rho_{p+1+r}\allowbreak \left(\alpha^{j-1},-, \alpha^{p+1+r-j}\right)$, 
the latter term becomes 
\begin{align*}
&\sum_{p+1+r=n\geq 1\atop 1\leq j\leq r} 
(-1)^{p+1+r-j+|\xi|(r-1)}
\rho_{p+1+r} \left(\alpha^{j-1},-, \alpha^{p}, \xi_q\left(\alpha^q\right),\alpha^{r-j}\right)\\
=&
\sum_{p+1+r=n\geq 1\atop 1\leq j\leq r} 
(-1)^{p(q-1)+|\xi|(p+r)+p+q+r-j}
\rho_{p+1+r}\circ_{p+1}\xi_q\left(\alpha^{j-1},-, \alpha^{p+q+r-j}\right)\ , 
\end{align*}
which  is equal to (i) with $i=j$.

The commutativity of the differentials comes from the relation 
$\left(d_\calP(\rho)\right)_1^\alpha=d_\calP\left(\rho_1^\alpha\right)$, which is straightforward. 
\end{proof}

\begin{lemma}\label{lem:MorphdgLieTER}
The assignment 
\begin{eqnarray*}
\left(\hom\left(\As^{\ac}, \calP\right), \partial, [\; , \,] \right)
\ltimes 
\left(\big(\calP\hat{\vee} \alpha\big)(1), d_\calP, [\; ,\,]\right)
&\to& 
\left(\Der\big(\calP\hat{\vee} \alpha\big), [d_\calP, -], [\; ,\,]\right)\\
(\rho, \nu) &\mapsto& \D_\rho+\ad_\nu
\end{eqnarray*}
defines a morphism of dg Lie algebras.
\end{lemma}

\begin{proof}
The compatibility with respect to the Lie brackets amounts to proving that 
\begin{align*}
\D_{[\rho,\xi]}+\ad_{[\nu, \omega]+\D_\rho(\omega)-(-1)^{|\xi||\nu|}\D_\xi(\nu)}
=
[\D_\rho, \D_\xi]+[\ad_\nu, \ad_\omega]+[\D_\rho, \ad_\omega]+[\ad_\nu, \D_\xi]\ .
\end{align*}
The first two terms are equal by \cref{lem:MorphdgLie}. The second two terms are equal since the adjonction is always a morphism of Lie algebras. The relation $[\D_\rho, \ad_\omega]=\ad_{\D_\rho(\omega)}$ is also a general fact  in Lie representation theory. 

After \cref{lem:MorphdgLie}, in order to prove the compatibility with respect to the differentials, it remains to show that 
$\ad_{d_\calP(\nu)}=[d_\calP, \ad_\nu]$, which come from the fact that $d_\calP$ is an operadic derivation. 
\end{proof}

\begin{theorem}\label{thm:DGActionOnTw}
The assignment 
\begin{eqnarray*}
\left(\hom\left(\As^{\ac}, \calP\right), \partial^\mu, [\; , \,] \right)&\to&
\left(\Der\big(\calP\hat{\vee} \alpha\big), [d_\calP+\D_\mu+\ad_{\mu_1^\alpha}, -], [\; ,\,]
\right)\\
\rho &\mapsto& \D_\rho+\ad_{\rho_1^\alpha}
\end{eqnarray*}
is a morphism of dg Lie algebras.
 In plain words, this 
defines a dg Lie action by derivation of the deformation complex $\Def\big(\Ai\to \calP\big)$ on the twisted operad 
$\Tw\calP$. 
\end{theorem}

\begin{proof}
The arguments are the same as in the proof of \cref{prop:DGAction}, using 
the composite of dg Lie algebra morphisms given respectively  in 
\cref{lem:MorphdgLieBIS} and in \cref{lem:MorphdgLieTER}, and twisting in the end by the Maurer--Cartan element $\mu$.
\end{proof}

\begin{remark}\label{rem:}
\cref{thm:DGActionOnTw} gives another way, actually the original one from \cite[Appendix~I]{Willwacher15}, to define the twisted complete dg ns operad 
\[
\Tw\calP\coloneqq\left(
\calP\hat{\vee} \alpha, \dd^{\mu_1^\alpha}\coloneqq d_\calP+\D_\mu+\ad_{\mu_1^\alpha}\right)\ .
\]
\end{remark}

The relationship between the deformation complex and the twisted operad is actually a bit more rich. 

\begin{proposition}\label{prop:DefTwP0}
Let $\Ai\to \calP$ be a multiplicative ns operad satisfying  $\calP(0)=0$. Up to a degree shift, the deformation complex is isomorphic to the chain complex made up of the arity $0$ component of the twisted operad: 
$$ 
\left(\left(\Tw\, \calP\right) (0), \dd^{\mu_1^\alpha}\right)\cong \left(s^{-1}\hom\left(\As^{\ac}, \calP\right), \partial^\mu\right)
\ .$$
\end{proposition}

\begin{proof}
On the level of the  underlying spaces, these two chain complexes satisfy 
\begin{eqnarray*}
\left(\Tw\, \calP\right) (0)=
\prod_{n\ge 0} \calP(n)\otimes \alpha^{\otimes n}\cong 
\prod_{n\ge 0} s^{-n}\calP(n)
\cong
 s^{-1}\hom\left(\As^{\ac}, \calP\right)\bigoplus \calP(0)\ .
\end{eqnarray*}
When $\calP(0)=0$, one way to realise the  isomorphism $\hom\left(\As^{\ac}, \calP\right) \stackrel{\cong}{\to} s\left(\Tw\, \calP\right) (0)$ is given   by 
$$\rho \mapsto (-1)^{|\rho|}s \D_\rho(\alpha)=-\sum_{n\ge 1} (-1)^{|\rho|}\rho_n\big(\alpha^{\otimes n}\big)\ .$$ 
It remains to show that it commutes with the respective differentials, that is to prove the following relation:
\[\D_{\partial(\rho)}(\alpha)+\D_{[\mu, \rho]}(\alpha)=d_\calP\left( \D_\rho(\alpha) \right)+\D_\mu\left(\D_\rho(\alpha)\right)+\ad_{\mu_1^\alpha}\left( \D_\rho(\alpha)\right)\ .\]
We have already seen in \cref{lem:MorphdgLie} that $\D_{\partial(\rho)}(\alpha)=d_\calP\left( \D_\rho(\alpha) \right)$ and that $\D_{[\mu, \rho]}(\alpha)=\D_\mu\left(\D_\rho(\alpha)\right)-(-1)^{|\rho|}\D_\rho\left(\D_\mu(\alpha)\right)$\ . So it remains to show that $\ad_{\mu_1^\alpha}\left( \D_\rho(\alpha)\right)=\allowbreak-(-1)^{|\rho|}\allowbreak\D_\rho\left(\D_\mu(\alpha)\right)$, which comes from the fact that both are explicitly equal to 
\[-\sum_{p, r\geq 0\atop q\leq 1} (-1)^{|\rho|r}\mu_{p+1+r}\left(\alpha^p, \rho_q\left(\alpha^q\right), \alpha^r\right)\ .\]
\end{proof}

\section{Generalisations}\label{Sec:Gen}
\cref{sec:TwNsOp} deals with the twisting procedure for ns operads where we used in a crucial way the dg ns operad 
$\Ai$. The  entire same theory holds as well with the dg ns operad $\sAi\coloneqq {\End}_{\k s}\otimes \Ai$ encoding shifted $\Ai$-algebras; in this case, the signs are nearly all  trivial. In order to get the twisting procedure for (symmetric) operads, one has to start with the dg operad $\Li$ encoding homotopy Lie algebras or the dg operad $\sLi\coloneqq {\End}_{\k s}\otimes\Li$ encoding shifted homotopy Lie algebras. The various proofs are performed with similar computations, and thus are left to the reader. In this way, one gets the theory developed by T. Willwacher but with a presentation different from \cite{Willwacher15, DolgushevRogers12, DolgushevWillwacher15}. Let us now give a more detailed presentation. 

\begin{proposition}
The complete dg operad encoding the data of a shifted homotopy Lie algebra together with a Maurer--Cartan element is 
$$\MC\sLi\coloneqq\left(
\widehat\calT\big(\alpha, \lambda_2, \lambda_3, \ldots\big), \dd
\right) \ , $$
where the generator $\alpha$ has arity $0$ and degree $0$ and 
where the generator $\lambda_n$ has arity $n$, degree $-1$, and trivial $\Sy_n$-action, for $n\ge 2$.
\end{proposition}

It admits the following Maurer--Cartan element 
\[
\lambda_1^\alpha\coloneqq\sum_{n\ge 2} {\textstyle  \frac{1}{(n-1)!}}\lambda_n\left(\alpha^{n-1}, -\right)\ .
\]

\begin{definition}[Twisted $\sLi$-operad]
The twisted $\sLi$-operad is 
\[
\Tw\sLi\coloneqq \left(\MC\sLi\right)^{\lambda_1^\alpha}\cong 
\left(
\widehat\calT\big(\alpha, \lambda_2, \lambda_3, \ldots\big), \dd^{\lambda_1^\alpha}
\right)\ ,
\]
with 
$\dd^{\lambda_1^\alpha}(\alpha)=\sum_{n\ge 2} {\textstyle  \frac{n-1}{n!}}\lambda_n\left(\alpha^{n}\right)$.
\end{definition}

\begin{proposition}
The assignment $\lambda_n\mapsto \lambda_n^\alpha$, where 
\[
\lambda_n^\alpha\coloneqq\sum_{r\ge 0} {\textstyle  \frac{1}{r!}}\lambda_{n+r}\left(\alpha^{r}, -, \ldots, -\right)\ .
\]
 defines a morphism of complete dg  operads 
$$\sLi \to \Tw \sLi \ . $$
\end{proposition}

\begin{proposition}\label{prop:MCPBIS}
Let $\sLi \to \calP$ be a morphism of complete dg operads. 
The data of a complete $\calP$-algebra structure together with a Maurer--Cartan element is encoded by the complete dg ns operad 
$$\MC\calP\coloneqq\left(
\calP\hat{\vee} \alpha, \dd
\right) \ , $$
where $\alpha$ is a degree $0$ element of arity $0$ placed in $F_1$ and where $\hat{\vee}$ stands for the coproduct of complete operads, and where the differential $\mathrm{d}$ is characterized by 
\begin{eqnarray*} 
&&\dd(\alpha)\coloneqq-\sum_{n\ge 2}{\textstyle  \frac{1}{n!}}\lambda_n(\alpha, \ldots, \alpha) \ ,\\
&&\dd(\nu)\coloneqq d_\calP(\nu)\ , \ \text{for} \ \nu\in \calP\ .
\end{eqnarray*}
\end{proposition}

\begin{definition}[Twisted operads under $\sLi$]
Let $\sLi \to \calP$ be a morphism of complete dg operads. 
The complete dg  operad obtained by twisting the operad $\MC \calP$ by the Maurer--Cartan $\lambda_1^\alpha$ is called the 
 \emph{twisted complete  operad} and denoted by 
$$\Tw\calP\coloneqq \left(\MC \calP\right)^{\lambda_1^\alpha}=
\left(
\calP\hat{\vee} \alpha, \dd^{\lambda_1^\alpha}
\right) \ . $$
\end{definition}

The twisted differential is actually equal to
\begin{eqnarray*}
\mathrm{d}^{\lambda_1^\alpha}(\alpha)&=&\sum_{n\ge 2}{\textstyle  \frac{n-1}{n!}}\lambda_n(\alpha, \ldots, \alpha)\ ,\\ \mathrm{d}^{\lambda_1^\alpha}(\nu)&=&d_\calP(\nu) 
+
\sum_{n\ge 2} {\textstyle  \frac{1}{(n-1)!}}\lambda_n\left(\alpha^{n-1}, \nu\right)
-
(-1)^{|\nu|}\sum_{j=1}^k {\textstyle  \frac{1}{(n-1)!}}\nu \circ_j \lambda_n(\alpha^{n-1}, -)\ ,
\end{eqnarray*}
for $\nu \in \calP(k)$. 

\begin{proposition}\label{prop:MorhpAiMCPBIS}
The assignment $\lambda_n\mapsto \lambda_n^\alpha$ defines a morphism of complete dg operads 
$$\sLi \to \Tw \calP \ . $$
\end{proposition}

We consider the following morphisms of complete dg operads
\begin{eqnarray*}
\begin{array}{rrcl}
\Delta(\calP)\ \ \  : &\Tw \calP\cong \calP\hat{\vee} \alpha&\to &\Tw\, \big(\Tw \calP\big)\cong \calP\hat{\vee} \alpha \hat{\vee} \beta\\
&\alpha &\mapsto & \alpha+\beta\ ,\\
&\nu&\mapsto&\nu\ , \ \text{for}\ \nu\in\calP\ ,
\end{array}
\end{eqnarray*}
and 
\begin{eqnarray*}
\begin{array}{rrcl}
\varepsilon(\calP)\ \ \  : &\Tw \calP&\to & \calP \\
&\alpha &\mapsto & 0\ ,\\
&\nu&\mapsto&\nu\ , \ \text{for}\ \nu\in\calP\ .
\end{array}
\end{eqnarray*}

\begin{theorem}
The two aforementioned morphisms $\Delta(\calP) \ :  \Tw \calP \to \Tw(\Tw\calP)$ and $\varepsilon(\calP)\ : \ \Tw \calP \to \calP$ of operads under $\sLi$ provide the endofunctor $\Tw$ with a comonad structure. 
\end{theorem}

\begin{definition}[Twistable operad]
We call \emph{twistable operad} an operad under $\sLi$ which admits a $\Tw$-coalgebra structure. 
\end{definition}

\begin{example}
Let us recall from \cite{DolgushevWillwacher15} that the operads for Lie algebras and Gerstenhaber algebras, as well as their shifted versions and their minimal models, are twistable, see \cref{prop:GerstTwistable}. Notice for instance, that the operads encoding pre-Lie algebras and Batalin--Vilkovisky algebras is not twistable, see \cref{prop:BVnotTwistable}. 
\end{example}

\begin{proposition}\label{prop:TwQIBIS}
The endofunctor $\Tw$ preserves quasi-isomorphisms. 
\end{proposition}

Let $\calP$ be a complete dg operad. We consider the  total space 
$$\hom_\Sy\left(\big(\calS\Lie\big)^{\ac}, \calP\right)\coloneqq\prod_{n\ge 1} \hom_\Sy\left(\Com(n)^*, \calP(n)\right)\cong \prod_{n\ge 1} \calP(n)^{\Sy_n}\ .$$ 

\begin{definition}[Deformation complex of morphisms of complete dg operads \cite{MerkulovVallette09I, MerkulovVallette09II}]
The \emph{deformation complex} of the morphism $ \sLi\to \calP$ of complete dg  operads is the complete  twisted dg Lie algebra 
$$\Def\big(\sLi\to \calP\big)\coloneqq\left(\hom_\Sy\left(\Com^*, \calP\right), \partial^\lambda, [\; , \,] \right)
\ . $$
\end{definition}

\begin{theorem}\label{thm:DGActionOnTwBIS}
The assignment 
\begin{eqnarray*}
\left(\hom_\Sy\left(\Com^*, \calP\right), \partial^\lambda, [\; , \,] \right)&\to&
\left(\Der\big(\calP\hat{\vee} \alpha\big), [d_\calP+\D_\lambda+\ad_{\lambda_1^\alpha}, -], [\; ,\,]
\right)\\
\rho &\mapsto& \D_\rho+\ad_{\rho_1^\alpha}
\end{eqnarray*}
is a morphism of dg Lie algebras, that is it  
defines a dg Lie action by derivation of the deformation complex $\Def\big(\sLi\to \calP\big)$ on the twisted operad 
$\Tw\calP$. 
\end{theorem}

\begin{proposition}\label{prop:DefTwP0BIS}
Let $\sLi\to \calP$ be a morphism of complete dg operads with  $\calP(0)=0$. The deformation complex is isomorphic to the chain complex made up of the arity $0$ component of the twisted operad: 
$$ 
\left(\left(\Tw\, \calP\right) (0), \dd^{\lambda_1^\alpha}\right)\cong \left(\hom_\Sy\left(\Com^*, \calP\right), \partial^\lambda\right)
\ .$$
\end{proposition}

\begin{remark}
One might try to develop the same twisting procedure of (ns) operads starting from a quadratic operad $\calP=\calP(E,R)$ whose associated category of $\calP_\infty$-algebras is twistable, according to \cref{def:TwHoAlg}, that is when its Koszul dual operad $\calP^{\ac}$ is extendable. It turns out that this more general situation is much more subtle. The various proofs given here, like the one of \cref{lem:MorphdgLie}, rely on the crucial fact that the ns cooperad $\calA s^{\ac}$ is one-dimensional in any arity and that its partial coproduct is the sum of all the possible way to compose operations in an ns operads. In other words, this amounts to the universal property satisfied by the ns operad $\As$ (respectively the operad $\Lie$): it is the unit for the Manin's black product of (finitely generated) binary quadratic ns operads (respectively  binary quadratic operads) \cite{GinzburgKapranov94, GinzburgKapranov95, Vallette08}. The analogous universal property satisfied by the dg ns operad $\Ai$ (respectively dg operad $\Li$) can be found in \cite{Wierstra16, Robert-Nicoud17}. This gives a hint on how to extend the twisting procedure to other kinds of algebraic structures like cyclic operads, modular operads or properads, for instance. 
\end{remark}

\begin{remark}
One easy generalisation of the formalism developed here concerns replacing operads by coloured operads. For example, if one considers a ns coloured operad $\calP$ into which the operad $\Ai$ (with all inputs and the output of the same colour) maps, the theory of operadic twisting effortlessly adapts in this case, allowing one to recover some classical constructions. For example, if one considers the cofibrant replacement of the coloured operad encoding the pairs $(A,M)$ where $A$ is an associative algebra and $M$ is a left $A$-module, the arising twisted differentials have, for example, been studied by T. Kadeishvili~\cite{Kadeishvili80} in the context of $\infty$-twisted tensor products.  
\end{remark}

There is however one way to ``extend'' the abovementioned theory with the following ``mise en abyme'' of the operadic twisting theory. Given a twistable complete operad $\Delta_\calG : \calG\to\Tw\, \calG$ and a complete operad $f : \calG\to \calP$ under $\calG$, the twisted complete operad $\Tw\, \calP$ is naturally an operad under $\calG$ by 
$$\xymatrix@C=30pt{\calG \ar[r]^(0.43){\Delta_\calG}& \Tw\,\calG  \ar[r]^(0.47){\Tw(f)} &\Tw\, \calP\ .}$$
Therefore, the twisting construction induces a comonad in the category of operads under the operad $\calG$. 

\begin{definition}[$\calG$-twistable operad]
A \emph{$\calG$-twistable operad} is a complete operad $\calP$ under $\calG$ which is a coalgebra for the twisting comonad $\Tw$. 
\end{definition}

The interpretation in terms of types of algebras is the same as above except that one should twist the operation of $\calP$ coming from $\calG$ as they are twisted in $\calG$. 

\begin{example}
One obvious example is given by the twistable operad $\calG=\Gerst$ encoding Gerstenhaber algebras and the operad $\calP=\BV$ under it which encodes Batalin--Vilkovisky algebras. One can also consider their Koszul resolution $\Gerst_\infty\to \BV_\infty$, where the latter one is given in \cite{GCTV12}. 
\end{example}

The concept of operadic twisting was introduced by T. Willwacher in \cite{Willwacher15} with the following motivation. Recall that M. Kontsevich considered in \cite{Kontsevich97} an operad $\Gra$, made up of graphs, which is the operad of natural operations acting on the sheaf of polyvector fields of $\mathbb{R}^n$. This operad includes the operad $\Gerst$ of Gerstenhaber algebras; it is therefore an  operad under the (cohomologically) shifted operad $\dsLi\coloneqq {\End}_{\k s^{-1}}\otimes\Li$. The deformation complex $\Def\big(\dsLi \to \Gra\big)$ coincides with  Kontsevich's graph complex and T. Willwa\-ch\-er proved that its $0^{\text{th}}$ (co)homology group is isomorphic to the Grothendieck--Tei\-chm\"uller Lie algebra $\mathfrak{grt}_1$. The theory of operadic twisting ensures that there is a natural action of the deformation dg Lie algebra $\Def\big(\dsLi \allowbreak \to \Gra\big)$ on the twisted operad $\Tw\, \Gra$ by derivation (\cref{thm:DGActionOnTw}). Since this latter one is quasi-isomorphic to the operad $\Gerst$ after \cite{Kontsevich99, LambrechtsVolic14}, one gets a natural action of the Grothendieck--Teichm\"uller Lie algebra $\mathfrak{grt}_1$ on $\mathrm{H}^0\big(\Der(\Gerst_\infty)\big)$. 
Actually T. Willwacher was able to prove that the Grothendieck--Teichm\"uller Lie algebra encaptures all the homotopy derivations: 
$$\mathfrak{grt}\coloneqq \mathfrak{grt}_1 \rtimes \k \cong \mathrm{H}^0\big(\Der\left(\Gerst_\infty\right)\big)\ .$$
This gives a proof of the fact that the group of homotopy automorphisms of the rational completion of the little disks operad is isomorphic to the pro-unipotent Grothendieck--Teichm\"uller group, see also the unstable approach of B. Fresse \cite{Fresse17II} using  rational homotopy theory for operads.

\section{Proof of Lemma~\ref{Lem:Diffmu}}\label{AppA}

\begin{proof}[Proof of Lemma~\ref{Lem:Diffmu}]
We have $\dd\big(\mu_0^\alpha\big)=0$ by the computation performed in the proof of Proposition~\ref{prop:MCAi}.
Let us now prove Relation~(\ref{eqn6}) for any $n\ge 1$; the special case $n=1$ will give Relation~(\ref{eqn5}). 

On the one hand, we have 
\begin{align*}
\dd\left(\mu_n^\at\right)&=
\sum_{r_0, \ldots, r_n\ge 0} (-1)^{\sum_{k=0}^n kr_k} \dd\left(\mu_{n+r_0+\cdots+r_n}\right)\big(\alpha^{r_0}, -, \alpha^{r_1}, -, \ldots,  - , \alpha^{r_{n-1}}, -,\alpha^{r_n}  \big)\\
&\ \ \ \ +
\sum_{r_0, \ldots, r_n\ge 0} 
\sum_{0\leq i \leq n\atop 1\leq j\leq r_i}
(-1)^{N+j+1} \mu_{n+r_0+\cdots+r_n}\big(\alpha^{r_0}, -, \ldots,-,  \at^{j-1}, \dd(\at), \at^{r_i-j}, -,\\ 
&\qquad\qquad\qquad\qquad\qquad\qquad\qquad\qquad\qquad\qquad\qquad\qquad\qquad\qquad\qquad
 \ldots, -,\alpha^{r_n}  \big)\\
&=\sum_{r_0, \ldots, r_n\ge 0} \sum_{p+q+r=n+r_0+\cdots+r_n\atop p+1+r, q\ge 2}
(-1)^{\sum_{k=0}^n kr_k+pq+r+1}\left(\mu_{p+1+r}\circ_{p+1} \mu_q\right)\big(\alpha^{r_0}, -, \alpha^{r_1}, -, \\
&\qquad\qquad\qquad\qquad\qquad\qquad\qquad\qquad\qquad\qquad\qquad\qquad\qquad
\ldots,  - , \alpha^{r_{n-1}}, -,\alpha^{r_n}  \big)
\\
&\ \ \ \  +
\sum_{r_0, \ldots, r_n\ge 0} 
\sum_{0\leq i \leq n\atop 1\leq j\leq r_i}
\sum_{m\ge 2}
(-1)^{N+j} \mu_{n+r_0+\cdots+r_n}\big(\alpha^{r_0}, -, \ldots,-,  \at^{j-1}, \mu_m\left(\at^m\right), \at^{r_i-j}, 
\\ 
&\qquad\qquad\qquad\qquad\qquad\qquad\qquad\qquad\qquad\qquad\qquad\qquad\qquad\qquad 
-,\ldots, -,\alpha^{r_n}  \big)\\
&=\sum_{r_0, \ldots, r_n\ge 0} \sum_{p+q+r=n+r_0+\cdots+r_n\atop p+1+r, q\ge 2}
(-1)^{\sum_{k=0}^n kr_k+pq+r+1}\left(\mu_{p+1+r}\circ_{p+1} \mu_q\right)\big(\alpha^{r_0}, -, \alpha^{r_1}, -,\\
&\qquad\qquad\qquad\qquad\qquad\qquad\qquad\qquad\qquad\qquad\qquad\qquad\qquad
 \ldots,  - , \alpha^{r_{n-1}}, -,\alpha^{r_n}  \big)\\
& +
\sum_{r_0, \ldots, r_n\ge 0} 
\sum_{0\leq i \leq n\atop 1\leq j\leq r_i}
\sum_{m\ge 2}
(-1)^{N+j+(m-2)(r_0+\cdots+r_{i-1}+j-1)}\left( \mu_{n+r_0+\cdots+r_n}\circ_{r_0+\cdots+r_{i-1}+j}\mu_m\right)\\
&\qquad\qquad\qquad\qquad\qquad\qquad\qquad\quad
\big(\alpha^{r_0}, -, \ldots,-,  \at^{j-1}, \at^m, \at^{r_i-j}, -,\ldots, -,\alpha^{r_n}  \big)\ , 
\end{align*}
where $N\coloneqq\sum_{k=0}^n kr_k +n+r_{i}+\cdots+r_n$. 
The first term is equal to the sum over all planar trees with two vertices (of arity at least $2$) and with leaves labelled by $\alpha$'s except for  $n$ of them.

The second term is equal to the sum over all planar trees with two vertices (of arity at least $2$), with leaves labelled by $\alpha$'s except for at least $n$ of them, and such that the leaves attached to the upper vertex are all labelled by $\alpha$'s.
If we write the elements of the second term with the same kind of indices used to describe the elements in the firm term, that is 
\begin{eqnarray*}
&&r'_0\coloneqq r_0,\  \ldots, r'_{i-1}\coloneqq r_{i-1}, \ r'_i\coloneqq r_i+m-1, \ r'_{i+1}\coloneqq r_{i+1},\ \ldots , \ r'_n\coloneqq r_n\\
&&p'\coloneqq r_0+\cdots+r_{i-1}+i+j-1=
r'_0+\cdots+r'_{i-1}+i+j-1,\ q'\coloneqq m,\\
&& r'\coloneqq r_i-j+r_{i+1}+\cdots+r_n+n-i
=r'_i+\cdots+r'_n+n-i-j-m+1\ , 
\end{eqnarray*}
we get the following sign
\begin{align*}
&(-1)^{\sum_{k=0}^n kr_k +n+r_{i}+\cdots+r_n+j+(m-2)(r_0+\cdots+r_{i-1}+j-1)}\\ 
=&
(-1)^{\sum_{k=0}^n kr'_k -i(q'-1) +n+r'_{i}+\cdots+r'_n-q'+1+j+q'(p'-i)}\\
=&(-1)^{\sum_{k=0}^n kr'_k +p'q'
+n+r'_{i}+\cdots+r'_n+j+i-q'+1}\\
=&(-1)^{\sum_{k=0}^n kr'_k +p'q'+r'}\ .
\end{align*}
Therefore the term of second kind cancel with the same term on the first sum. 

Let us introduce the following convention: for a sequence $r_0, \ldots, r_n$ of integers, we denote by $|r|\coloneqq r_0+\cdots+r_n$ their sum. 
Let us now consider the right-hand side of Relation~(\ref{eqn6}):
\begin{align*}
&\sum_{p+q+r=n\atop p+1+r, q\ge 1}(-1)^{pq+r+1} \mu^\alpha_{p+1+r}\circ_{p+1} \mu^\alpha_q=\\
&
\sum_{p+q+r=n\atop p+1+r, q\ge 1}
\sum_{l_0,\ldots, l_{p+1+r}\ge 0\atop s_0,\ldots, s_q\ge 0}
(-1)^{M}
\mu_{p+1+r+|l|}(\at^{l_0}, -, \ldots, -, \at^{l_{p+1+r}})\circ_{p+1}
\mu_{q+|s|}(\at^{s_0}, -, \ldots, -, \at^{s_q})\\
&=
\sum_{p+q+r=n\atop p+1+r, q\ge 1}
\sum_{l_0,\ldots, l_{p+1+r}\ge 0\atop s_0,\ldots, s_q\ge 0}
(-1)^{M+(q+|s|)|l|+|s|(l_{p+1}+\cdots+l_{p+1+r})}
\mu_{p+1+r+|l|}\circ_{p+l_0+\cdots+l_p+1}\\
&\ \ \ \ \mu_{q+|s|}
(\at^{l_0}, -, \ldots, -,\at^{l_{p-1}}, -, \at^{l_{p}+s_0}, -, \at^{s_1}, \ldots, -, \at^{s_{q-1}}, -, \at^{s_q+l_{p+1}}, -,\at^{l_{p+2}}, -, \ldots, \\ 
&\qquad \qquad \qquad \qquad \qquad \qquad \qquad \qquad \qquad \qquad \qquad \qquad \qquad 
\qquad \qquad \qquad -, \at^{l_{p+1+r}})\ , 
\end{align*}
where $M\coloneqq pq+r+1+\sum_{k=0}^{p+1+r}kl_k+\sum_{k=0}^{q}ks_k$. It is 
which is therefore made up of a sum over all planar trees with two vertices (of arity at least $2$) and with leaves labelled by $\alpha$'s except for  $n$ of them, and such that at least one leaf 
 attached to the upper vertex is not labelled by an $\alpha$. So, by the previous computation, these terms correspond, up to sign, to the remaining terms of 
$\dd\left(\mu_n^\at\right)$. Let us again write this sign in term of the indices convention of the first term of $\dd\left(\mu_n^\at\right)$: 
\begin{eqnarray*}
&&r'_0\coloneqq l_0\ ,\  \ldots, r'_{p-1}\coloneqq l_{p-1}\ , \ r'_p\coloneqq l_p+s_0\ , \ r'_{p+1}\coloneqq s_1\ ,\ \ldots , \ r'_{p+q-1}\coloneqq s_{q-1}\ , \\&&
r'_{p+q}\coloneqq s_{q}+l_{p+1}\ , \ r'_{p+q+1}\coloneqq l_{p+2}\ ,\  \ldots, \ r'_n\coloneqq l_{p+1+r}\ , 
\\
&&p'\coloneqq l_0+\cdots+l_{p}+p=\ ,\ 
q'\coloneqq q+|s|\ ,\ 
 r'\coloneqq l_{p+1}+\cdots+l_{p+1+r}+r
\ .
\end{eqnarray*}
With this convention, the respective two signs agree:
\begin{align*}
&(-1)^{\sum_{k=0}^n kr'_k+p'q'+r'+1}\\
=&
(-1)^{\sum_{k=0}^{p+1+r} kl_k+\sum_{k=0}^{q} ks_k+(q-1)\sum_{k=p+1}^{p+1+r} l_k+p|s|+
\left(\sum_{k=0}^{p} l_k + p
\right)
\left(
q+|s|
\right)
+\sum_{k=p+1}^{p+1+r} l_k+r+1}\\
=&(-1)^{pq+r+1+\sum_{k=0}^{p+1+r} kl_k+\sum_{k=0}^{q} ks_k+
q\sum_{k=p+1}^{p+1+r} l_k+
\left(\sum_{k=0}^{p} l_k
\right)
\left(
q+|s|
\right)}
\\
=&(-1)^{pq+r+1+\sum_{k=0}^{p+1+r} kl_k+
\sum_{k=0}^{q} ks_k+
\left(
q+|s|
\right)|l|+
|s|(l_{p+1}+\cdots +l_{p+1+r})}
\ ,
\end{align*}
which concludes the proof. 
\end{proof}

\newpage
\chapter{Examples}\label{sec:Computations}

In this section, we develop some examples of operadic twisting, both for nonsymmetric and symmetric operads. 
In each case, we study whether the (nonsymmetric) operad is twistable and we compute the (co)homology of its twisted version. 
This will show that the (co)homology of twisted operads is a rather unpredictable functor. 

We will deal first with the ns operads encoding respectively  noncommutative Gerstenhaber and noncommutative Batalin--Vilkovisky algebras, notions introduced recently in \cite[Section~3]{DotsenkoShadrinVallette15}. Then, we will consider the cases of the operads encoding respectively Gerstenhaber and Batalin--Vilkovisky algebras. The results related to the former operad are not new, but we show them using a different method which is ad hoc and thus shorter. This also allows us to fix the notations and methods to address the latter case. 
In the end, this shows that the classical operads and their ns analogues behave in a different way with respect to the twisting procedure. 

In this section, we work exceptionally with cohomological degree convention. 

\section{Twisting the nonsymmetric operad \texorpdfstring{$\ncGerst$}{ncGerst}} \label{subsec:ncGerst}
In this section, we work over a  ring  $\k$.

\smallskip

Recall after \cite[Section~3.1]{DotsenkoShadrinVallette15} that the nonsymmetric operad $\ncGerst$ can be defined as follows. As a graded $\k$-module the space $\ncGerst(n)$, for $n\ge 1$, is freely spanned by possibly disconnected linear graphs, called \emph{bamboos}, where the vertices are ordered from left to right, for instance,
\begin{equation*}
\vcenter{\xymatrix@M=5pt@R=10pt@C=10pt{
		*+[o][F-]{1}\ar@{-}[r] &*+[o][F-]{2}\ar@{-}[r]&  *+[o][F-]{3} & *+[o][F-]{4} &*+[o][F-]{5}\ar@{-}[r]&  *+[o][F-]{6} 
}
}
\quad \in \ncGerst(6)
\end{equation*}
The arity $0$ space $\ncGerst(n)=0$ is trivial.
Each edge carries cohomological degree $1$ and the total cohomological degree is thus equal to the number of edges. 
(These elements corresponds to right-combs of binary generators with the presentation given in \cite[Section~3.1]{DotsenkoShadrinVallette15}.)
It is convenient to think that the edges are ordered and reordering generates the sign; for a given bamboo, we assume,  by default, the ordering of the edges from left to right. The operadic composition $\circ_i$ of two bamboos, $\Gamma_1\in \ncGerst(n)$ and $\Gamma_2\in \ncGerst(k)$ amounts to placing the bamboo $\Gamma_2$ at the place of the vertex $i$ of the bamboo $\Gamma_1$, globally relabelling the vertices from left to right by $1,\dots,n+k-1$. 
The edge $(i-1,i)$ (resp. $(i+k-1,i+k)$) belongs to the resulting graph if and only if the edge $(i-1,i)$ (resp., $(i,i+1)$) belongs to the graph $\Gamma_1$. A Koszul sign is generated by reordering the edges; it is given by the parity of the product between the number of edges of $\Gamma_1$ on the right of vertex $i$ and the
 number of edges of $\Gamma_2$. For instance,
\begin{align*}
& \vcenter{\xymatrix@M=5pt@R=10pt@C=10pt{
		*+[o][F-]{1}\ar@{-}[r] &*+[o][F-]{2}\ar@{-}[r]&  *+[o][F-]{3} & *+[o][F-]{4} &*+[o][F-]{5}\ar@{-}[r]&  *+[o][F-]{6} 
	}
}
\quad \circ_5 \quad
\vcenter{\xymatrix@M=5pt@R=10pt@C=10pt{
		*+[o][F-]{1}\ar@{-}[r] &*+[o][F-]{2}&  *+[o][F-]{3} 
	}
}
\\
& = (-1)\cdot\ \vcenter{\xymatrix@M=5pt@R=10pt@C=10pt{
		*+[o][F-]{1}\ar@{-}[r] &*+[o][F-]{2}\ar@{-}[r]&  *+[o][F-]{3} & *+[o][F-]{4} &*+[o][F-]{5}\ar@{-}[r] &*+[o][F-]{6}&  *+[o][F-]{7} \ar@{-}[r]&  *+[o][F-]{8} 
	}
}
\
.
\end{align*}\\

A natural ns multiplicative structure $\Ai\twoheadrightarrow \As \to \ncGerst$ is given by the assignment 
\[
\mu_2\mapsto \quad
\vcenter{\xymatrix@M=5pt@R=10pt@C=10pt{
		*+[o][F-]{1} &*+[o][F-]{2} 
	}
}
\
,
\] 
or, alternatively, a natural shifted ns multiplicative structure $\calS^{-1}\Ai \twoheadrightarrow \calS^{-1}\As\to \ncGerst$ is given by the assignment  
\[
\mu_2\mapsto \quad
\vcenter{\xymatrix@M=5pt@R=10pt@C=10pt{
		*+[o][F-]{1}\ar@{-}[r] &*+[o][F-]{2} 
	}
}
\
.
\]
So, one can twist the ns operad $\ncGerst$ in the way described above or by using a cohomologically shifted version of the above-mentioned twisting procedure. In this case, we work with the dg ns operad  $\calS^{-1}\Ai$ of cohomologically shifted $\Ai$-algebras, that is with opposite signs than the one of shifted $\Ai$-algebras.
The subsequent computations are equivalent for these two structures, but the shifted one is slightly more convenient for the presentation. So, we perform all computations for the shifted one, and we denote by $\Tw\ncGerst$ the ns operad twisted with respect to the shifted ns multiplicative structure.\\

The underlying  $\k$-module of  $\Tw\ncGerst(n)_d$ is made up of  
by series indexed by $k\geq 0$ of finite sums of 
the bamboos with $n$ white vertices (labelled by $1,\dots,n$ from left to right), $k$ black vertices, and $d$ edges.
The twisted differential $\dd^{\mu_1^\alpha}$ is equal to the sum of the following five types of summands:
\begin{enumerate}
\item we attach a black vertex from the left to the leftmost vertex of the bamboo: 
$\vcenter{\xymatrix@M=5pt@R=10pt@C=10pt{
		*+[o][F**]{} \ar@{-}[r]& *+[o][F.]{} \ar@{..}[r] &   }}$\ , 
\item we attach a black vertex from the right to the rightmost vertex of the bamboo: 
$\vcenter{\xymatrix@M=5pt@R=10pt@C=10pt{
		 \ar@{..}[r] & *+[o][F.]{} \ar@{-}[r] &*+[o][F**]{}  }}$\ , 

\item we replace a white vertex by a black vertex connected by an edge to the white vertex on the left: 
$\vcenter{\xymatrix@M=5pt@R=10pt@C=10pt{
	\ar@{..}[r]&	*+[o][F**]{} \ar@{-}[r]& *+[o][F-]{i} \ar@{..}[r] &   }}$\ , 

\item we replace a white vertex by a black vertex connected by an edge to the white vertex on the right: 
$\vcenter{\xymatrix@M=5pt@R=10pt@C=10pt{
	\ar@{..}[r]&*+[o][F-]{i}  \ar@{-}[r]&	*+[o][F**]{}  \ar@{..}[r] &   }}$\ ,

\item we replace a black vertex by two black vertices connected by an edge: 
$\vcenter{\xymatrix@M=5pt@R=10pt@C=10pt{
		 \ar@{..}[r] & *+[o][F**]{} \ar@{-}[r] &*+[o][F**]{}  \ar@{..}[r]& }}$\ . 
\end{enumerate}
The Koszul type sign is given by counting how many edges from left to right that the new edge has to jump over, with an extra $-1$ sign for the terms $(3)$, $(4)$, and $(5)$. 

For instance,
\begin{align*}
& \mathrm{d}^{\mu_1^\alpha}\left( 
\vcenter{\xymatrix@M=5pt@R=10pt@C=10pt{
		*+[o][F-]{1}\ar@{-}[r] &*+[o][F-]{2}&  *+[o][F**]{} 
	}
}\right)
\ =
\vcenter{\xymatrix@M=5pt@R=10pt@C=10pt{
		*+[o][F-]{1}\ar@{-}[r] &*+[o][F-]{2}&*+[o][F**]{}\ar@{-}[l]&  *+[o][F**]{}}}
\ 
,
\end{align*}
and the terms 
\begin{align*}
&
 \vcenter{\xymatrix@M=5pt@R=10pt@C=10pt{
		*+[o][F**]{}\ar@{-}[r]&*+[o][F-]{1}\ar@{-}[r] &*+[o][F-]{2}&  *+[o][F**]{} 
	}}\ , \quad 
	\vcenter{\xymatrix@M=5pt@R=10pt@C=10pt{
		*+[o][F-]{1}\ar@{-}[r] &  *+[o][F-]{2}&  *+[o][F**]{}  &*+[o][F**]{}\ar@{-}[l] }}
		\ 
\quad \text{and}\quad
\vcenter{\xymatrix@M=5pt@R=10pt@C=10pt{
		*+[o][F-]{1}\ar@{-}[r] & *+[o][F**]{}\ar@{-}[r]& *+[o][F-]{2}&  *+[o][F**]{} }}
\end{align*}
appear in the expression for the differential twice, but with the opposite signs. 

\begin{proposition} \label{prop:ncGerstnotTwistable}
The ns operad $\ncGerst$ is not twistable. 
\end{proposition}

\begin{proof}
We apply \textit{mutatis mutandis} \cref{prop:Twistable} with the following computation 
\begin{align*}
\ad_{\mu_2(\alpha, -)+\mu_2(-, \alpha)}\left(
\vcenter{\xymatrix@M=5pt@R=10pt@C=10pt{
		*+[o][F-]{1} &*+[o][F-]{2}}}
		\right)=
(-1)\cdot\  \vcenter{\xymatrix@M=5pt@R=10pt@C=10pt{
		*+[o][F-]{1}\ar@{-}[r] & *+[o][F**]{} & 
		 *+[o][F-]{2}}}
+ (-1)\cdot\  \vcenter{\xymatrix@M=5pt@R=10pt@C=10pt{
		*+[o][F-]{1} & *+[o][F**]{}\ar@{-}[r] & 
		 *+[o][F-]{2}}}
		\neq 0\ .
\end{align*}
\end{proof}

The ns operad $\calS^{-1}\As$, denoted $\As_1$ in \cite{DotsenkoShadrinVallette15}, of cohomologically shifted associative algebras, is isomorphic to the ns suboperad of $\ncGerst$ generated by the element $b\coloneqq 
\vcenter{\xymatrix@M=5pt@R=10pt@C=10pt{
		*+[o][F-]{1}\ar@{-}[r] &*+[o][F-]{2} }}$\ .
We consider the complete ns operad 
\[\calS^{-1}\As^+\coloneqq \frac{\calS^{-1}\As\hat{\vee} \, \gamma}{\big( b(\gamma, -), \ b(-,\gamma) \big)}\ ,
\]
where $\gamma$ is an arity $0$ degree $0$ element placed in $\F_1$. 

\begin{theorem}  \label{thm:HTwncGerst} 
The map of complete dg ns operads 
\[\calS^{-1}\As\hat{\vee}\, \gamma \to \Tw\ncGerst \quad \text{defined by} \quad b\mapsto \vcenter{\xymatrix@M=5pt@R=10pt@C=10pt{
		*+[o][F-]{1}\ar@{-}[r] &*+[o][F-]{2} }}\ , \ \ 
\gamma \mapsto \vcenter{\xymatrix@M=5pt@R=10pt@C=10pt{
			 *+[o][F**]{}  &*+[o][F**]{} }}\]
induces the isomorphism of complete ns operads 
\[  \mathrm{H}\left(\Tw\ncGerst\right)\cong \calS^{-1}\As^+\ .\]
\end{theorem}

\begin{proof} Since the elements of $\Tw \ncGerst$ are series indexed by $k\geq 0$ of finite sums of 
the bamboos with  $k$ black vertices and since the  differential $\dd^{\mu_1^\alpha}$ increases the number of black vertices by one, it is enough to consider the case of finite series, i.e. sums. 

Note that the differential preserves the number of the connected components $K\geq 1$ and the number of the white vertices $N_1,\dots,N_K\geq 0$ on these components. 
The subgraph that consists of all black vertices and edges connecting them has $K+\sum_{i=1}^K N_i$
 disjoint connected components (some of them can be empty in a particular graph). The chain complex of all graphs with fixed $K$ and fixed $N_1,\dots,N_K$ is isomorphic to the tensor product of the $K+\sum_{i=1}^K N_i$ chain complexes disjoint black components described below. 

Consider a connected black component of length $n$, for $n\ge 0$. Under the  action of the differential, it is replaced with a connected component of length $n+1$, with a coefficient $c_n$ that depends on its position within the ambient graph. Here is a full list of the possible cases:
\begin{enumerate}
	\item The black component is connected to white vertices both on the left and on the right: $c_n=0$ for even $n$ and $c_n=\pm 1$ for odd $n$. 
	\item The black component is connected to a white vertex only on the left and it is not the rightmost component of the ambient bamboo: $c_n=\pm 1$, for even $n$, and $c_n=0$, for odd $n$. The same for the interchanged left and right. 
	\item The black component is connected to a white vertex only on the left and it is the rightmost component of the ambient bamboo: $c_n=0$, for even $n$, and $c_n=\pm 1$, for odd $n$. The same for the interchanged left and right. 
	\item The black component is not connected to white vertices and is neither the leftmost nor the rightmost component of the ambient bamboo: $c_n=0$, for even $n$, and $c_n=\pm 1$ for odd $n$. Note that in this case $n\geq 1$. 
	\item The black component is not connected to white vertices and it is the rightmost but not the leftmost component of the ambient bamboo: $c_n=\pm 1$, for even $n$, and $c_n=0$ for odd $n$. The same for the interchanged left and right. Note that in this case $n\geq 1$. 
	\item The black component is not connected to white vertices and it is simultaneously the rightmost and the leftmost component of the ambient bamboo, that is this black component is the whole ambient bamboo graph satisfying $K=1$, $N_1=0$: $c_n=0$, for even $n$, and $c_n=\pm 1$, for odd $n$. Note that in this case $n\geq 1$. 
\end{enumerate}
In the cases (2), (4), and (6), the corresponding chain complex is acyclic, thus the cohomology is equal to zero. In the cases (1) and (3), the cohomology is one-dimensional represented by a black component of length $0$. In the case (5), the cohomology is one-dimensional represented by a black component of length $1$. 

Thus, for the ambient graph, we either have the one-dimensional cohomology group for $K=1$, $N_1\geq 1$, represented by the connected bamboos  with only white vertices or 
we have the one-dimensional cohomology group for $K=2$, $N_1=N_2=0$, represented by two disjoint black vertices. These graphs generate the complete ns operad $\calS^{-1}\As \hat{\vee} \, \gamma
$, and it is clear that the substitution of the latter graph into a connected bamboo of white vertices of length at least two gives a boundary of the differential, cf. the acyclic case (2) above.
\end{proof}

\begin{remark} Note that in this case the cohomology of the deformation complex $\Def\big(\calS^{-1}\Ai\to \allowbreak\ncGerst\big)$ is the one-di\-men\-sion\-al Lie algebra with the trivial Lie bracket, since it is isomorphic to the cohomology of the arity $0$ component of $\Tw\ncGerst$ by \cref{prop:DefTwP0}. The above computations show that its action on $\mathrm{H}\left(\Tw\ncGerst\right)$ is trivial. 

As usual in deformation theory, this result can be interpreted, see for instance \cite[Section~12.2]{LodayVallette12},  as  a strong rigidity statement about the class of ns operad morphisms $\calS^{-1}\Ai\to \ncGerst$: there is no non-trivial  infinitesimal or formal deformation of the map given here.
\end{remark}

\section{Twisting the nonsymmetric operad $\mathrm{ncBV}$} \label{subsec:ncBV}
In this section, we work over a  ring  $\k$.

\smallskip

The nonsymmetric operad $\ncBV$ from \cite[Section~3.1]{DotsenkoShadrinVallette15} can be defined as follows. As a graded $\k$-module the space $\ncBV(n)$, for $n\ge 1$, is freely spanned by possibly disconnected linear graphs (bamboos) with at most one tadpole edge at each vertex, where the vertices are ordered from left to right. For instance, we have 
\begin{align*}
\rule{0pt}{22pt}
\ncBV(1) = \left\langle
\vcenter{
	\xymatrix @M=5pt@R=5pt@C=10pt{
		*+[o][F-]{1}
	}
}\ \ ,\ \ 
\vcenter{
	\xymatrix @M=5pt@R=5pt@C=10pt{
		*+[o][F-]{1} \ar@{-}@(ru,u) 
	}
}\ \ \
\right\rangle ; \qquad
\vcenter{
	\xymatrix @M=5pt@R=5pt@C=10pt{
		*+[o][F-]{1}\ar@{-}[r] &*+[o][F-]{2}\ar@{-}[r] \ar@{-}@(ru,u) &  *+[o][F-]{3} & *+[o][F-]{4} \ar@{-}@(ru,u)&*+[o][F-]{5}\ar@{-}[r] \ar@{-}@(ru,u)&  *+[o][F-]{6} 
	}
}
\quad \in \ncBV(6)\ .
\end{align*}
Each edge, including the tadpoles, has cohomological degree equal to $1$ and the total cohomological degree is equal to the number of edges. 
(To match with the presentation given in \cite[Section~3.2]{DotsenkoShadrinVallette15}, these elements are in one-to-one correspondance with right-combs of binary generators labelled, at the very top of them, with nothing or one copy of the generator $\Delta$ at each leaf.)
 It is convenient to think that the edges are ordered and reordering generates the sign, like in the above $\ncGerst$ case. For a given bamboo with tadpoles, by default, we order first the edges from left to right first and then the tadpoles from left to right. \\

 The operadic composition $\circ_i$ of two  bamboos with tadpoles, $\Gamma_1\in \ncBV(n)$ and $\Gamma_2\in \ncBV(k)$, amounts to placing the graph $\Gamma_2$ at the place of the vertex $i$ of the graph $\Gamma_1$, globally relabelling the vertices from left to right by $1,\dots,n+k-1$. The edge $(i-1,i)$ (resp. $(i+k-1,i+k)$) belongs to the resulting graph if and only if the edge $(i-1,i)$ (resp., $(i,i+1)$) belongs to the graph $\Gamma_1$. If there is a tadpole at the vertex $i$ of $\Gamma_1$, then it becomes either a new tadpole at one of the vertices $i,i+1,\dots,i+k-1$ of $\Gamma_1\circ_i\Gamma_2$, when no tadpole was yet present, or a new edge connecting two consecutive vertices from this set. A Koszul sign is generated by reordering the edges and tadpoles. For instance,
\begin{align*}
\rule{0pt}{22pt}
 \vcenter{\xymatrix@M=5pt@R=10pt@C=10pt{
		*+[o][F-]{1}\ar@{-}[r] &*+[o][F-]{2}\ar@{-}@(ru,u)
	}
}
\ \ \circ_2 \ \
\vcenter{\xymatrix@M=5pt@R=10pt@C=10pt{
		*+[o][F-]{1}\ar@{-}[r] &*+[o][F-]{2} \ar@{-}@(ru,u) &  *+[o][F-]{3} \ar@{-}@(ru,u) 
	}
}\ \
 = \ \ & (-1)\cdot\ \vcenter{\xymatrix@M=5pt@R=10pt@C=10pt{
 		*+[o][F-]{1}\ar@{-}[r] & *+[o][F-]{2}\ar@{-}[r] \ar@{-}@(ru,u) &*+[o][F-]{3} \ar@{-}@(ru,u) &  *+[o][F-]{4} \ar@{-}@(ru,u)
 	}
 }
\\ & +\rule{0pt}{30pt}
(-1)\cdot\ \vcenter{\xymatrix@M=5pt@R=10pt@C=10pt{
		*+[o][F-]{1}\ar@{-}[r] & *+[o][F-]{2}\ar@{-}[r]  &*+[o][F-]{3} \ar@{-}@(ru,u) \ar@{-}[r]&  *+[o][F-]{4} \ar@{-}@(ru,u)
	}
}
\
.
\end{align*}

As is the previous example, we consider a cohomologically shifted ns multiplicative structure $s^{-1}\Ai \twoheadrightarrow s^{-1}\As\to \ncBV$ given by the assignment 
\[
\mu_2\mapsto \quad
\vcenter{\xymatrix@M=5pt@R=10pt@C=10pt{
		*+[o][F-]{1}\ar@{-}[r] &*+[o][F-]{2} 
	}
}
\
.
\]
The underlying  $\k$-module of  $\Tw\ncBV(n)_d$ is made up of  
by series indexed by $k\geq 0$ of finite sums of 
the bamboos with tadpoles, with $n$ white vertices (labelled by $1,\dots,n$ from left to right), $k$ black vertices, and a total of $d$ edges and tadpoles.
The components of the twisted differential $\mathrm{d}^{\mu_1^\alpha}$ are the same as in the previous case of the dg ns operad $\Tw \ncGerst$. The only new case comes with vertices (white or black) having tadpoles: they are again replaced by two vertices connected by an edge (black-white plus white-black, or black-black) where the tadpole distributes over the two vertices. The signs remain the same as in the $\Tw \ncGerst$ case: 
the Koszul type sign is given by counting how many edges from left to right that the new edge has to jump over. With the  order considered on edges and tadpoles, this means that tadpoles will never be taken into account when computing this sign. There is an extra $-1$ sign for the terms which replace a vertex by two vertices. 
 For instance,
\begin{align}\label{Eqn:RelationHTwncBV}
\rule{0pt}{22pt} 
\mathrm{d}^{\mu_1^\alpha}\left(\ 
\vcenter{\xymatrix@M=5pt@R=10pt@C=10pt{
		*+[o][F-]{1} \ar@{-}@(ru,u) 
	}
}\ \right)
\ \
=
\ \ (-1)\cdot \
\vcenter{\xymatrix@M=5pt@R=10pt@C=10pt{
		*+[o][F-]{1} \ar@{-}[r] &*+[o][F**]{} \ar@{-}@(ru,u) 
	}
}
\ \
+(-1)\cdot\ 
\vcenter{\xymatrix@M=5pt@R=10pt@C=10pt{
		*+[o][F**]{} \ar@{-}@(ru,u) & *+[o][F-]{1} \ar@{-}[l] 
	}
}
\ 
,
\end{align}
and the graphs 
\begin{align*}
\rule{0pt}{22pt}
\vcenter{\xymatrix@M=5pt@R=10pt@C=10pt{
		*+[o][F-]{1} \ar@{-}[r] \ar@{-}@(ru,u) &*+[o][F**]{}  
	}
}
\ \ \
\text{ and }
\ \ \
\vcenter{\xymatrix@M=5pt@R=10pt@C=10pt{
		*+[o][F**]{} & *+[o][F-]{1} \ar@{-}[l] \ar@{-}@(ru,u) 
	}
}
\end{align*}
appear in the expression for the differential twice, with the opposite signs. 

\begin{proposition} \label{prop:ncBVnotTwistable}
The ns operad $\ncBV$ is not twistable. 
\end{proposition}

\begin{proof}
The same argument and computation as in the case of the ns operad, see $\ncGerst$ \cref{prop:Twistable}, hold here and prove the result. 
\end{proof}

In order to describe the cohomology ns operad $\mathrm{H}(\Tw\ncBV)$, we consider the following extension of the complete ns operad introduced in the previous $\ncGerst$ case:
\[\calS^{-1}\As^{++}\coloneqq \frac{\calS^{-1}\As\hat{\vee}\, \gamma\, \hat{\vee}\, \zeta}
{\big( b(\gamma, -),\ b(-,\gamma), b(\zeta, -)+b(-,\zeta)\big)}\ ,
\]
where $\zeta$ is an arity $0$ degree $+1$ element placed in $\F_1$. 

\begin{theorem} \label{thm:HTwncBV} 

The map of complete dg ns operads 
\[
\rule{0pt}{22pt}
\calS^{-1}\As\hat{\vee}\, \gamma\, \hat{\vee}\, \zeta \to \Tw\ncBV \quad \text{defined by} \quad b\mapsto \vcenter{\xymatrix@M=5pt@R=10pt@C=10pt{
		*+[o][F-]{1}\ar@{-}[r] &*+[o][F-]{2} }}\ , \ \ 
\gamma \mapsto \vcenter{\xymatrix@M=5pt@R=10pt@C=10pt{
			 *+[o][F**]{}  &*+[o][F**]{} }}
			 \ , \ \ 
\zeta \mapsto 
\vcenter{\xymatrix@M=5pt@R=10pt@C=10pt{
		*+[o][F**]{} \ar@{-}@(ru,u)  }}
			 \]
induces the isomorphism of complete ns operads 
\[  \mathrm{H}\left(\Tw\ncBV\right)\cong \calS^{-1}\As^{++}\ .\]
\end{theorem}

\begin{proof} As in the previous case, since the elements of $\Tw \ncBV$ are series indexed by $k\geq 0$ of finite sums of 
the bamboos with tadpoles with  $k$ black vertices and since the  differential $\dd^{\mu_1^\alpha}$ increases the number of black vertices by one, it is enough to consider the case of finite series, i.e. sums. 

At fixed arity $n\ge 0$, we consider the increasing filtration $F_p \Tw\ncBV(n)$ span\-ned by bamboos containing  at least $-p$ black vertices with tadpoles. 
The twisted differential $\dd^{\mu_1^\alpha}$ preserves this filtration, which is exhaustive and bounded below since a black vertex with a tadpole carries cohomological degree $+1$. So its associated spectral sequence converges to the cohomology of $\Tw\ncBV(n)$. The differential $d^0$ of the first page is made up of the components of 
 $d^{\mu_1^\alpha}$ which do not increase the number of black vertices with tadpoles, that is the ones which increase only the number of black vertices without tadpoles. In order to compute its cohomology groups, we apply the same arguments as in the proof of \cref{thm:HTwncGerst}: here the black and white vertices with tadpoles play the same role as the white vertices without tadpoles. Therefore, the second page $E^1$ is spanned by the cohomology class $\gamma$ represented by $\vcenter{\xymatrix@M=5pt@R=10pt@C=10pt{
		*+[o][F**]{}  &*+[o][F**]{}}}$ and by the connected bamboos made up of white vertices with or without tadpoles and black vertices with tadpoles. The differential $d^1$ creates a black vertex with tadpole and an edge on the left and on the right of any white vertex with tadpole, that is 
\begin{align*}
\rule{0pt}{22pt} d^1\left(\ 
\vcenter{\xymatrix@M=5pt@R=10pt@C=10pt{
		*+[o][F-]{\phantom{1}} \ar@{-}@(ru,u) }}
\ \right)
=
\ 
(-1)\cdot \
\vcenter{\xymatrix@M=5pt@R=10pt@C=10pt{
		*+[o][F**]{} \ar@{-}@(ru,u) & *+[o][F-]{\phantom{1}} \ar@{-}[l] 
	}
}
\ + \ 
(-1)\cdot \
\vcenter{\xymatrix@M=5pt@R=10pt@C=10pt{
		*+[o][F-]{\phantom{1}} \ar@{-}[r] &*+[o][F**]{} \ar@{-}@(ru,u) 
	}}
\
, \ \ 
d^1\left( \ 
\vcenter{\xymatrix@M=5pt@R=10pt@C=10pt{
		*+[o][F**]{} \ar@{-}@(ru,u) 
	}
}
\ \right)
= 
d^1\left( \ 
\vcenter{\xymatrix@M=5pt@R=10pt@C=10pt{
		*+[o][F-]{\phantom{1}} 
	}
}
\ \right)
= 
d^1\left( \ 
\vcenter{\xymatrix@M=5pt@R=10pt@C=10pt{
		*+[o][F**]{} 
	}
}
\ \right)
= 0 \ .
\end{align*} 
We consider the filtration of this cochain complex $(E^1,d^1)$ where $\F_p E^1$ is spanned by the above mentioned bamboos with at least $-p$ sub-bamboos of the form 
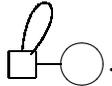
\begin{align*}
\rule{0pt}{22pt} \vcenter{\xymatrix@M=5pt@R=10pt@C=10pt{
		*+[o][F**]{} \ar@{-}@(ru,u) & *+[o][F-]{\phantom{1}} \ar@{-}[l] }}\ .
\end{align*}
This filtration is again exhaustive and bounded below, so 
its associated spectral sequence converges to the cohomology $E^2=\mathrm{H}(E^1, d^1)$. 
The differential $\dd^0$ of the first page $\mathrm{E}^0$ is made up of the second above component of $d^1$, that is the one which produces a black vertex with a tadpole on the right-hand side of a white vertex with tadpole. The cochain subcomplexes of $\mathrm{E}^0$ made up of bamboos containing $k$ sub-bamboos of the form 
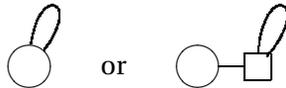
\begin{align*}
\rule{0pt}{22pt} \vcenter{\xymatrix@M=5pt@R=10pt@C=10pt{
		*+[o][F-]{\phantom{1}} \ar@{-}@(ru,u) }}
\qquad \text{or}\qquad 
\vcenter{\xymatrix@M=5pt@R=10pt@C=10pt{
		*+[o][F-]{\phantom{1}} \ar@{-}[r] &*+[o][F**]{} \ar@{-}@(ru,u) 
	}
}
\end{align*}
are acyclic, for $k\ge 1$, since they are isomorphic to the tensor product of $k$ acyclic cochain complexes. 
The cohomology groups $\mathrm{E}^1$ is thus spanned by the following bamboos: 
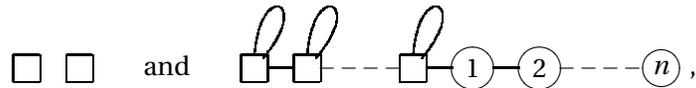
\begin{align*}
\rule{0pt}{22pt}
\vcenter{\xymatrix@M=5pt@R=10pt@C=10pt{
		*+[o][F**]{}  &*+[o][F**]{}}} 
		\qquad \text{and} \qquad	
		\vcenter{\xymatrix@M=5pt@R=10pt@C=10pt{
		*+[o][F**]{} \ar@{-}@(ru,u)  \ar@{-}[r] &
		*+[o][F**]{} \ar@{-}@(ru,u)  \ar@{--}[rr] & &
		*+[o][F**]{} \ar@{-}@(ru,u)  \ar@{-}[r] & *+[o][F-]{1} \ar@{-}[r] & *+[o][F-]{2} \ar@{--}[rr] & & *+[o][F-]{n}
	}
}
\ ,
\end{align*} 
with $j\ge 0$ black vertices with tadpoles and $n\ge 0$ white vertices without tadpoles. On such bamboos, the differential $\dd^1$ vanishes, since it is given by the 
first above component of $d^1$,  the one which produces a black vertex with a tadpole on the left-hand side from a white vertex with tadpole. So the second spectral sequence collapses at $\mathrm{E}^1$ and the first spectral sequence collapses at $E^2$ with basis given by these latter bamboos. 

The three elements 
\[
\rule{0pt}{22pt}
b\longleftrightarrow \vcenter{\xymatrix@M=5pt@R=10pt@C=10pt{
		*+[o][F-]{1}\ar@{-}[r] &*+[o][F-]{2} }}\ , \ \ 
\gamma \longleftrightarrow \vcenter{\xymatrix@M=5pt@R=10pt@C=10pt{
			 *+[o][F**]{}  &*+[o][F**]{} }}
			 \ , \ \quad \text{and} \quad 
\zeta \longleftrightarrow
\vcenter{\xymatrix@M=5pt@R=10pt@C=10pt{
		*+[o][F**]{} \ar@{-}@(ru,u)  }}\]
		are clearly generators of the cohomology ns operad $\mathrm{H}(\Tw\ncBV)$. The computation performed at \cref{Eqn:RelationHTwncBV} gives the relation between $b$ and $\zeta$ introduced in the definition of 
$\calS^{-1}\As^{++}$. So the above assignement induces a morphism 
$\calS^{-1}\As^{++} \to \mathrm{H}\left(\Tw\ncBV\right)$
of complete ns operads, which turns out to be an isomorphism since the dimensions of the underlying graded $\mathbb{N}$-modules of $\calS^{-1}\As^{++}$ coincide with the number of bamboos spanning $\mathrm{H}(\Tw\ncBV)$. 
\end{proof}

\begin{remark} Note that in this case the cohomology of  $\Def\big(\calS^{-1}\Ai\to \ncBV\big)$ is isomorphic to $\k \gamma \oplus \k[\hbar]\zeta$, where the degrees are given by $|\gamma|=0$, $|\zeta|=1$, $|\hbar|=2$, with the trivial Lie bracket. The  previous computations  show that its action on $\mathrm{H}\left(\Tw\ncBV\right)$ is trivial. 

This result can be 
 interpreted as  a rigidity statement: 
 since the dimension of $\mathrm{H}^1\big(\Def\big(\calS^{-1} \Ai\to \ncBV\big)\big)$ is equal to $1$, 
 there exists just one class of non-trivial infinitesimal deformations of the morphism 
 of dg ns operads   $\calS^{-1}\Ai\to \ncBV$ given here.  
\end{remark}

\section{Twisting the operad \texorpdfstring{$\Gerst$}{Gerst}} \label{subsec:TwGerst}
In this section, we work over a field $\k$ of characteristic $0$. In this section, we work with the twisting procedure of complete dg (symmetric) operads explained at the beginning of \cref{Sec:Gen} of \cref{sec:TwNsOp}. 

\smallskip

A Gerstenhaber algebra is made up of a commutative  product and a degree $1$ Lie bracket  satisfying together the Leibniz relation, see \cite[Section~13.3.10]{LodayVallette12} for more details. The associated operad $\Gerst$ is thus generated by an arity $2$ degree $0$ element $\mu$ with trivial $\Sy_2$ action and by an arity $2$ degree $1$ element $\lambda$ also with trivial $\Sy_2$ action. This latter one induces a  multiplicative structure 
$\dsLi \twoheadrightarrow \calS^{-1}\Lie \to \Gerst$. So, one can twist the operad $\Gerst$ using a cohomological version of the above-mentioned twisting procedure, that is working with the dg operad  $\dsLi$ of cohomologically shifted $\Li$-algebras. This amounts simply to considering the same operads as above but with opposite degree convention. 
The various properties of the corresponding twisted operad $\Tw\Gerst\coloneqq \left(
 \Gerst\hat{\vee} \alpha , \dd^{\lambda_1^\alpha}
 \right)$ follow from the general statements developed in~\cite[Section 5]{DolgushevWillwacher15} for operad under distributive law.  We give below others but ad hoc and thus shorter proofs for these properties. 
 
\begin{proposition}[{\cite[Corollary~5.12]{DolgushevWillwacher15}}] \label{prop:GerstTwistable}
The canonical morphism of operads $\Gerst \hookrightarrow \Tw\Gerst$ defines a $\Tw$-coalgebra structure. 
\end{proposition}

\begin{proof}
This a direct corollary of the second statement of \cref{prop:Twistable} \textit{mutatis mutandis}. 
The Jacobi relation implies $\dd^{\lambda_1^\alpha}(\lambda)=\ad_{\lambda(\alpha, -)}(\lambda)=0$ and  
the Leibniz relation implies $\dd^{\lambda_1^\alpha}(\mu)=\ad_{\lambda(\alpha, -)}(\mu)=0$; we refer the reader to \cite[Section~13.3.12]{LodayVallette12} for these relations in this ``degree 1'' setting. This completes the proof. 
\end{proof}

\begin{remark}
Such a result says, in an operadic way, that the commutative product and the degree $1$ Lie bracket of any dg Gerstenhaber algebra form again a dg  Gerstenhaber algebra structure with the twisted differential produced by any Maurer--Cartan element. 
\end{remark}

\begin{theorem}[{\cite[Corollary~5.12 and Corollary~5.13]{DolgushevWillwacher15}}] \label{thm:TwGerst} 
The canonical morphisms of complete dg operads 
\[\Tw\Gerst \xrightarrow{\sim} \Gerst\qquad \text{and} \qquad 
\Tw\Gerst_\infty \xrightarrow{\sim} \Gerst_\infty\]
are quasi-isomorphisms. 
\end{theorem}

\begin{proof} 
Let us begin with the first statement about the operad $\Gerst$. The arguments given in the above proof of \cref{prop:GerstTwistable} show that 
 the only non-trivial part of the twisted differential $\dd^{\lambda_1^\alpha}$ is on $\alpha$, where it is equal to 
$\dd^{\lambda_1^\alpha}(\alpha)=\frac12 \lambda(\alpha, \alpha)$\ .

First, we consider  arity $0$ part of $\Tw\Gerst$. We recall that the operad $\Gerst\cong \Com\circ s^{-1}\Lie$ satisfies the distributive law method, see  \cite[Section~8.6]{LodayVallette12} for instance. The  Jacobi relation ensures that 
$\widehat{s^{-1}\Lie}(\alpha)\cong \k \alpha \oplus \k \lambda(\alpha, \alpha)$. Since this latter term has degree $1$, it can only appear only once in $\widehat{\Gerst}(\alpha)$; therefore we get $\Tw\Gerst(0)\cong \widehat{\Com}(\alpha) \oplus \widehat{\Com}\big(\alpha; \lambda(\alpha, \alpha)\big)$. We denote by $\mu^{k}$ any composite of $k$ times $\mu$ in the operad $\Gerst$. 
Since 
$\dd^{\lambda_1^\alpha}\left(\mu^{k-1}\left(\alpha^k\right)\right)=\frac{k}{2}\mu^{k-1}\left(\alpha^{k-1},\lambda(\alpha, \alpha) \right)$ and since
$\dd^{\lambda_1^\alpha}\big(\lambda(\alpha, \alpha)\big)=0$, this cochain complex is acyclic. 

In higher arity, we use the notation $\lambda^k\coloneqq (\cdots( \lambda\circ_1 \lambda)\cdots )\circ_1 \lambda$ for the composite of $k$ operations $\lambda$ at the first input. Recall that a basis of the operad $s^{-1}\Lie$ is given by the elements $\left(\lambda^{n-1}\right)^\sigma$, for $n\ge 1$, where $\sigma$ runs over the permuations of $\Sy_n$ which fix $1$. 
For any finite set $J$, we denote by 
\[\lambda^{\left(\sigma, \bar{k}\right)}\coloneqq \lambda^{|\bar{k}|+|J|-1}
\left(-, \alpha^{k^1},-, \alpha^{k^2}, -, \cdots, -, \alpha^{k^{|J|}}\right)^\sigma\ ,
\]
where $\sigma$ is a permutation of $J$ fixing its least element and where 
$\bar{k}=\left( 
k^1, \ldots, k^{|J|}
\right)$ 
stands for a $|J|$-tuple of non-negative integers; such elements form a $\k$-linear basis of the operad $s^{-1}\Lie\vee \alpha$\ . 
As a consequence, the $\k$-module $\Tw\Gerst(n)_d$ is generated by the linearly independent elements of the form 
\[
\mu^{m+p+q-1}\left(
\alpha^m, \lambda(\alpha, \alpha)^p, \lambda^{\left(\sigma_1, \overline{k_1}\right)}, \ldots, \lambda^{\left(\sigma_q, \overline{k_q}\right)}
\right)\ ,
\]
where $m\ge 0$, $p=0$ or $p=1$, where the permutations $\sigma_1, \ldots, \sigma_q$ are associated to a partition $\sqcup_{i=1}^q J_i=\{1,\dots,n\}$\ , and where the total number of  $\lambda$'s is equal to $d$.  (By degree reason, the only series that can appear are indexed by $m$; the rest are finite terms).
Since $\frac12 \lambda(-, \lambda(\alpha, \alpha)) =-\lambda(\lambda(-,\alpha), \allowbreak \alpha)$, the differential $\mathrm{d}^{\mu_1^\alpha}$ preserves such basis elements. 

So the cochain complex $\Tw\Gerst(n)$ splits into the direct sum of cochain complexes indexed by the decompositions $\sqcup_{i=1}^q J_i=\{1,\dots,n\}$ and the permutations $\sigma_1, \ldots, \sigma_q$. The form of the differential $\mathrm{d}^{\mu_1^\alpha}$ implies that  each of these direct summand is isomorphic to the tensor product of $1+q$ cochain complexes, where the first one is isomorphic to $\k \oplus \Tw\Gerst(0)$ and where 
the $q$ other ones are spanned by the elements $\lambda^{\left(\sigma_i, \overline{k_i}\right)}$, for any possible $|J_i|$-tuples $\overline{k_i}$.
The above result about the  arity zero case implies that the cohomology of  the first factor is  one-dimensional.
It is straightforward to see that $\mathrm{d}^{\mu_1^\alpha} \left(
\lambda^k\left(-, \alpha^k\right)
\right) = - \lambda^{k+1}\left(-, \alpha^{k+1}\right)$, for odd $k$, and $\mathrm{d}^{\mu_1^\alpha} \left(
\lambda^k\left(-, \alpha^k\right)
\right)  = 0 $, for even $k$. 
Thus each of the other $q$ tensor factors spanned by the elements $\lambda^{\left(\sigma_i, \overline{k_i}\right)}$ has one-dimensional cohomology represented by  $\lambda^{\left(\sigma_i, \bar{0}\right)}$ 
In the end, the only non-trivial class for each of these direct summand is represented by the basis element 
\[
\mu^{q}\left(\lambda^{\left(\sigma_1, \bar{0}\right)}, \ldots, \lambda^{\left(\sigma_q, \bar{0}\right)}
\right)
\]
that has no $\alpha$'s at all. These representatives form a natural basis for the operad $\Gerst$, which concludes the proof. 

The second statement about the operadic resolution  $\Gerst_\infty \xrightarrow{\sim} \Gerst$ follows directly from the result about the operad $\Gerst$ and \cref{prop:TwQI}. 
\end{proof}

\begin{remark} 
This result shows that the cohomology of the deformation complex  $\Def\big(\dsLi\allowbreak \to \allowbreak \Gerst \allowbreak \big)$ is trivial in this case.
So any of its actions on the operads $\MC \Gerst$ and $\Tw \Gerst$ are cohomologically trivial. 
This result can be 
 interpreted as a strong rigidity statement: there is no non-trivial infinitesimal or formal deformation of the morphism of dg operads $\dsLi\allowbreak \to \allowbreak \Gerst$ given here.
\end{remark}

\section{Twisting the operad \texorpdfstring{$\BV$}{BV}} \label{subsec:TwBV}
In this section, we work again over a field $\k$ of characteristic $0$. 

\smallskip

A Batalin--Vilkovisky (BV) algebra is a Gerstenhaber algebra endowed with a compatible degree $1$ square-zero linear operator, see \cite[Section~13.7]{LodayVallette12} for more details. 
The associated operad , denoted $\BV$, is thus generated by the same kind of elements  $\mu$ and $\lambda$ as above plus an arity $1$ and degree $1$ element $\Delta$. 
By the same argument, it acquires a multiplicative structure 
$\dsLi \twoheadrightarrow s^{-1}\Lie \to \BV$. 
 We consider the corresponding twisted operad $\Tw\BV$.

\begin{proposition}\label{prop:BVnotTwistable}
The operad $\BV$ is not twistable. 
\end{proposition}

\begin{proof}
This a direct corollary of the first statement of \cref{prop:Twistable} \textit{mutatis mutandis}. 
In the operad $\Tw\BV$, the differential is given by 
\[\dd^{\lambda_1^\alpha}(\Delta)=\ad_{\lambda(\alpha, -)}(\Delta)=-\lambda(\Delta(\alpha),-)\neq 0\ ,\] which concludes the proof. 
\end{proof}

We consider the complete operad 
\[\Gerst^+\coloneqq \frac{\Gerst\hat{\vee}\, \eta}{\big( \lambda(\eta, -)\big)}\ ,
\]
where $\eta$ is an arity $0$ degree $+1$ element placed in $\F_1$. For  degree reasons 
$\mu(\eta, \eta)=\lambda(\eta, \eta)=0$, so $\Gerst^+(0)\cong \k \eta$ is one-dimensional and $\Gerst^+(1)\cong \k \id \oplus \k \mu(\eta, -)$ is two-dimensional. 

\begin{theorem} \label{thm:HTwBV}
The map of complete dg operads 
\[\Gerst\hat{\vee}\, \eta \to \Tw\BV \quad \text{defined by} \quad \mu\mapsto \mu\ , \ \lambda\mapsto \lambda\ , \ 
\eta \mapsto \Delta(\alpha)\]
induces the isomorphism of complete operads 
\[  \mathrm{H}\left(\Tw\BV\right)\cong \Gerst^+\ .\]
\end{theorem}

\begin{proof}We use the same notations and arguments as in the proof of \cref{thm:TwGerst}. For instance, the distributive law method gives  $\BV\cong \Com\circ s^{-1}\Lie \circ \k[\Delta]$. The differential in the twisted operad $\Tw\BV$ is given by 
\[\dd^{\lambda_1^\alpha}(\lambda)=\allowbreak \dd^{\lambda_1^\alpha}(\mu)=\allowbreak\dd^{\lambda_1^\alpha}\left(\Delta(\alpha)\right)=\allowbreak0\] and by $\dd^{\lambda_1^\alpha}(\Delta)=-\lambda(\Delta(\alpha),-)$\ .

We first consider the part of arity zero of the operad $\Tw\BV$. Let us denote by $\Delta\Lie$ the sub-operad generated by $\Delta$ and $\lambda$ in the operad $\BV$; its underlying graded $\Sy$-module is isomorphic to 
$\Delta\Lie\cong \calS^{-1}\Lie \circ \k[\Delta]$. The complete sub-operad $\Delta\Lie\hat{\vee} \alpha$ of $\Tw \BV$ is stable under the differential $\dd^{\lambda_1^\alpha}$. 
We denote by $M\coloneqq \big(\Delta\Lie\hat{\vee} \alpha(0), 0,  \ldots \big)$ the dg $\Sy$-module concentrated in arity $0$; it satifises the following isomorphism of dg $\Sy$-modules 
\[
\big( 
\Tw\BV(0), 0, \ldots\big)
\cong 
\Com \circ M\ .
\]
The operadic K\"unneth formula \cite[Proposition~6.2.3]{LodayVallette12} implies that the cohomology of $\Tw\BV(0)$
is isomorphic to $\big(\Com\circ \mathrm{H}(M)\big)(0)$. Then, we consider the following isomorphism of dg modules 
$\big(\Delta\Lie\hat{\vee}\alpha\big)(0) \cong 
\big(\calS^{-1}\Lie \hat{\vee}\alpha\circ \k[\Delta(\alpha)]\big)(0)$.
The arguments and computations given in the proof of \cref{thm:TwGerst} show that 
$\mathrm{H}\left(\calS^{-1}\Lie \hat{\vee}\alpha, \dd^{\lambda_1^\alpha}
\right)\cong \calS^{-1}\Lie$. Applying again the operadic K\"unneth formula, we get 
$\mathrm{H}\left(\big(\calS^{-1}\Lie \hat{\vee}\alpha\circ \k[\Delta(\alpha)]\big)(0)\right)\cong 
\big(\calS^{-1}\Lie \circ \k[\Delta(\alpha)]\big)(0)$. So the cohomology of $\Tw\BV(0)$ is isomorphic to 
$\left(\Com\circ\big(\calS^{-1}\Lie \circ \k[\Delta(\alpha)]\big)\right)(0)$. By degree reasons, since $|\Delta(\alpha)|=1$, we get in the end 
$\mathrm{H}(\Tw\BV(0))\cong \k\Delta(\alpha)$\ .

To treat the case of arity $n\geq 1$, we consider, for any finite set $J$, the elements of the form 
\[\lambda^{\left(\sigma, \bar{k}, \bar{*}\right)}\coloneqq \lambda^{|\bar{k}|+|J|+d-1}
\left(*^1, \alpha^{k^1},*^2, \alpha^{k^2}, -, \cdots, *^{|J|+d}, \alpha^{k^{|J|+d}}\right)^\sigma\ ,
\]
where $\sigma$ and $\bar{k}$ are as in the proof of \cref{thm:TwGerst} and where $\bar{*}=\left(*^1, \ldots, *^{|J|+d}\right)$ is  a $|J|+d$-tuple with $*^1$ equals to $\Delta$ or $-$ and with $*^i$ equals to 
$\Delta$, $-$, or $\Delta(\alpha)$, such that the number of $\Delta(\alpha)$'s is equal to $d$. The $\k$-module $\Tw\BV(n)$ is generated as above by the elements of the form 
\[
\mu^{m+p+q+r-1}\left(
\alpha^m, \lambda(\alpha, \alpha)^p, 
\lambda^{\overline{k_1}}, \ldots, \lambda^{\overline{k_q}}, 
\lambda^{\left(\sigma_1, \overline{l_1}, \overline{*_1}\right)}, \ldots, \lambda^{\left(\sigma_r, \overline{l_r}, \overline{*_r} \right)}
\right)\ ,
\]
where the permutations $\sigma_1, \ldots, \sigma_r$ are associated to a partition $\sqcup_{i=1}^r J_i=\{1,\dots,n\}$\ .

Let us compute the cohomology of the cochain complex $\Lambda(J)$ generated by the elements of the form 
$\lambda^{\left(\sigma, \bar{k}, \bar{*}\right)}$ for a fixed finite set $J$. We consider the increasing filtration $\F_p \Lambda(J)$ made up of the elements $\lambda^{\left(\sigma, \bar{k}, \bar{*}\right)}$ containing  at least $-p$ elements $\Delta(\alpha)$. The differential $\dd^{\lambda_1^\alpha}$ preserves this filtration. It is exhaustive and bounded below, so its associated spectral sequence converges to the cohomology of $\Lambda(J)$. The differential of the first page of the associated spectral sequence is equal to the sole term $\dd^0(\alpha)=\frac12 \lambda(\alpha, \alpha)$. 

The same argument as in the proof of \cref{thm:TwGerst} shows that the second page $\mathrm{E}^1$ is generated by the elements of the form 
$\lambda^{\left(\sigma, \bar{0}, \bar{*}\right)}$. The differential of this second page is given by the sole term   
$d^1(\Delta)=-\lambda(\Delta(\alpha),-)$, thus 
$
d^1\big(\lambda(\nu, \Delta)\big)=-\lambda(\lambda(\nu, \Delta(\alpha)),-)-\lambda(\lambda(\nu, -),\Delta(\alpha))-
\lambda\big(d^1(\nu), \Delta\big)$, for $\nu$ made up of at least one $\lambda$,  and 
$d^1\big(\lambda(\Delta, *)\big)= \lambda(\lambda(-, \Delta(\alpha),*)+\lambda\left(\Delta, d^1(*)\right) $ \ .
We consider the filtration of $\mathrm{E}^1$ defined by counting the numbers of $\Delta$'s at the first input: for $p\leq 0$, $F_p \mathrm{E}^1$ is generated by the elements of the form $\lambda^{\left(\sigma, \bar{0}, \bar{*}\right)}$ with $*^1=-$, and $F_p\mathrm{E}^1=\mathrm{E}^1$, for $p\ge 1$. This filtration is exhaustive and bounded below, so it converges to the cohomology of $\mathrm{E}^1$. The differential $d^0$ of the first page $E^0$ of the associated spectral sequence is given by $\dd^1$, except when $\Delta$ labels the first input: in this case $d^0$ vanishes on it. Therefore, the cochain complex $(E^0, d^0)$ is isomorphic to two copies, labeled respectively by the input $\Delta$ or $-$  of the first leaf, of the same cochain complex. 
We consider the coaugmented coalgebra 
$C\coloneqq \k 1\oplus \k x \oplus \k y \oplus \k z$, with $|x|=0$, $|y|=|z|=1$,  where $x$ and $z$ are primitive elements, and where the (reduced) coproduct of $y$ is equal to $x\otimes z -z \otimes x$.
Under the correspondance $x\leftrightarrow -$, $y\leftrightarrow \Delta$, and $z\leftrightarrow \Delta(\alpha)$, the cochain complex $(E^0, d^0)$ is isomorphic to the cobar construction (with "unsual" cohomological degree convention, see \cite[Section~2.2.2]{LodayVallette12}) of $C$. Since the Koszul dual algebra of $C$ is the Koszul algebra 
$C^{\ac}\cong T(X,Z)/(X\otimes Z-Z\otimes X)$, with $|X|=1$ and $|Z|=2$, the second page $E^1$ is isomorphic to $\left(\k x \oplus \k y \right)\otimes C^{\ac}$, with admits for basis $xZ^kX^l$ and $yZ^kX^l$, for $k,l\ge 0$. The differential $d^1$ is given by $d^1\left(xZ^kX^l\right)=0$ and $d^1\left(yZ^kX^l\right)=(-1)^{l+1}xZ^{k+1}X^l$. So this  spectral sequence collapses at $E^2$, where is it spanned by the elements $xX^l$, for $l\ge 0$. In other words, the first spectral sequence collapses at $\mathrm{E}^2\cong \mathrm{H}\left(\Lambda(J)\right)$, which is spanned by the elements of the form $\lambda^{\left(\sigma, \bar{0}, \bar{*}\right)}$ with $\bar{*}=(-, \ldots, -)$. 

In the end, the only non-trivial cohomology class of $\big(\Tw\BV(n), \dd^{\lambda_1^\alpha}\big)$  
is represented by the basis element 
\[
\mu^{r+p-1}\left(\Delta(\alpha)^p, \lambda^{\left(\sigma_1, \bar{0}\right)}, \ldots, \lambda^{\left(\sigma_r, \bar{0}\right)}
\right)\ ,
\]
where $p=0,1$ and where the permutations $\sigma_1, \ldots, \sigma_r$ are associated to a partition $\sqcup_{i=1}^r J_i=\{1,\dots,n\}$. These elements corresponds to a basis of the operad $\Gerst^+$ under the correspondence $\eta\leftrightarrow \Delta(\alpha)$, which concludes the proof. 
\end{proof}

\begin{remark} Note that the cohomology of  $\Def\big(\S1\Lie_\infty\to \BV\big)$ is one-di\-men\-sio\-nal Lie algebra concentrated in degree $1$ and  with the abelian Lie bracket. The computations performed in the above proof show that  its action on $\mathrm{H}\left(\Tw\BV\right)\allowbreak\cong \Gerst^+$ is equal to the degree $1$ differential of $\Gerst^+$ which assigns $\mu \mapsto \lambda$. 

This result provides us with the following 
 rigidity statement: there exists just one class of non-trivial infinitesimal deformations of the morphism 
 of dg  operads   $\calS^{-1}\Ai\to \BV$ given here. 
\end{remark}

\newpage
\bibliographystyle{amsplain}
\bibliography{bib}

\end{document}